\documentclass[11pt,reqno,final]{amsart}

\usepackage{color}
\usepackage{amsmath,amsthm,amssymb,amsfonts,enumerate,enumerate, mathrsfs}
\usepackage[pdftex]{graphicx}

\usepackage{verbatim}
\usepackage[utf8]{inputenc}
\usepackage[spanish,USenglish]{babel}
\usepackage{tikz}
\usepackage[hidelinks]{hyperref}
\usepackage{mathtools}
\mathtoolsset{showonlyrefs}
\usepackage{bbm}
\usepackage{mathabx}
\allowdisplaybreaks
\usepackage{showlabels}
\usepackage{wasysym}
\usepackage{xcolor}
\usepackage{stmaryrd}
\usepackage{todonotes}
\usepackage{nicefrac}
\usepackage{csquotes}
\usepackage{euflag}
\usepackage{doi}

\DeclareMathOperator*{\esssup}{ess\,sup}

\newtheorem{theorem}{Theorem}[section]
\newtheorem{lemma}[theorem]{Lemma}
\newtheorem{proposition}[theorem]{Proposition}
\newtheorem{corollary}[theorem]{Corollary}

\theoremstyle{definition}
\newtheorem{assumption}[theorem]{Assumption}
\newtheorem{definition}[theorem]{Definition}
\newtheorem{example}[theorem]{Example}

\theoremstyle{remark}
\newtheorem{remark}[theorem]{Remark}

\newcommand\bN{\mathbb{N}}

\newcommand\bR{\mathbb{R}}

\renewcommand\P{\mathbb{P}}

\newcommand\cB{\mathcal{B}}

\newcommand\cF{\mathcal{F}}

\newcommand\cH{\mathcal{H}}

\newcommand\cL{\mathcal{L}}
\newcommand\calL{\mathcal{L}}
\newcommand\cM{\mathcal{M}}

\newcommand\cP{\mathcal{P}}
\newcommand\cO{\mathcal{O}}

\newcommand\cR{\mathcal{R}}
\newcommand\cS{\mathcal{S}}

\newcommand\cV{\mathcal{V}}
\newcommand\cW{\mathcal{W}}
\newcommand\cX{\mathcal{X}}

\newcommand\cZ{\mathcal{Z}}

\newcommand{\E}{\mathbb{E}}
\newcommand{\1}{\mathbbm{1}}
\newcommand{\one}{\mathbbm{1}}
\newcommand{\Ly}{L\'evy }

\newcommand{\cl}{c\`adl\`ag }

\newcommand{\Nt}{\widetilde{N}}
\newcommand{\wt}{\widetilde}

\DeclareMathOperator{\dv}{div}

\DeclareMathOperator{\supp}{supp}

\DeclareMathOperator{\Tr}{Tr}

\newcommand{\<}{\langle}
\newcommand{\lb}{\langle}
\renewcommand{\>}{\rangle}
\newcommand{\rb}{\rangle}

\newcommand{\vertiii}[1]{{\left\vert\kern-0.25ex\left\vert\kern-0.25ex\left\vert #1
\right\vert\kern-0.25ex\right\vert\kern-0.25ex\right\vert}}

\newcommand{\eps}{\varepsilon}

\newcommand{\Dom}{\mathscr{O}}
\newcommand{\R}{\mathbb{R}}

\newcommand{\of}{\overline{f}}

\DeclareFontEncoding{FMS}{}{}
\DeclareFontSubstitution{FMS}{futm}{m}{n}
\DeclareFontEncoding{FMX}{}{}
\DeclareFontSubstitution{FMX}{futm}{m}{n}
\DeclareSymbolFont{fouriersymbols}{FMS}{futm}{m}{n}
\DeclareSymbolFont{fourierlargesymbols}{FMX}{futm}{m}{n}
\DeclareMathDelimiter{\VERT}{\mathord}{fouriersymbols}{152}{fourierlargesymbols}{147}

\DeclareFontFamily{U}{matha}{\hyphenchar\font45}
\DeclareFontShape{U}{matha}{m}{n}{
      <5> <6> <7> <8> <9> <10> gen * matha
      <10.95> matha10 <12> <14.4> <17.28> <20.74> <24.88> matha12
      }{}
\DeclareSymbolFont{matha}{U}{matha}{m}{n}
\DeclareFontSubstitution{U}{matha}{m}{n}

\DeclareFontFamily{U}{mathx}{\hyphenchar\font45}
\DeclareFontShape{U}{mathx}{m}{n}{
      <5> <6> <7> <8> <9> <10>
      <10.95> <12> <14.4> <17.28> <20.74> <24.88>
      mathx10
      }{}
\DeclareSymbolFont{mathx}{U}{mathx}{m}{n}
\DeclareFontSubstitution{U}{mathx}{m}{n}

\DeclareMathDelimiter{\vvvert}{0}{matha}{"7E}{mathx}{"17}

\begin{document}

\author[S. Bechtel]{Sebastian Bechtel}
\address{Universit\'e Paris-Saclay, CNRS, Laboratoire de Math\'ematiques d'Orsay, 91405 Orsay, France}
\email{sebastian.bechtel@universite-paris-saclay.fr}

\author[F. Germ]{Fabian Germ}
\address{Delft Institute of Applied Mathematics\\
Delft University of Technology \\ P.O. Box 5031\\ 2600 GA Delft\\The
Netherlands} \email{F.Germ@tudelft.nl}

\author[M. Veraar]{Mark Veraar}
\address{Delft Institute of Applied Mathematics\\
Delft University of Technology \\ P.O. Box 5031\\ 2600 GA Delft\\The
Netherlands} \email{M.C.Veraar@tudelft.nl}

\thanks{The first-named author is supported by the Alexander von Humboldt foundation by a Feodor Lynen grant. The second-named and third-named authors are supported by the VICI subsidy VI.C.212.027 of the Netherlands Organisation for Scientific Research (NWO). This project has received funding from
the European Union’s Horizon 2020 research and innovation programme under the Marie Skłodowska-Curie grant agreement No 101034255 \euflag{}}


\begin{abstract}
The critical variational setting was recently introduced and shown to be applicable to many important SPDEs not covered by the classical variational setting. In this paper, we extend the critical variational setting in several ways. We introduce a flexibility in the range space for the nonlinear drift term, due to which certain borderline cases can now also be included. An example of this is the Allen-Cahn equation in dimension two in the weak setting. In addition to this, we allow the drift to be singular in time, which is something that naturally arises in the study of the skeleton equation for large deviation principles for SPDEs. Last but not least, we present the theory in the case of L\'evy noise for which the critical setting was not available yet.
\end{abstract}

\title[Variational setting for critical SPDEs with L\'evy noise]{An extended variational setting for critical SPDEs with L\'evy noise}

\maketitle

\setcounter{tocdepth}{1}
\tableofcontents

\section{Introduction}
The variational setting for both deterministic and stochastic evolution equations can be highly effective in applications. It provides global well-posedness for a large class of nonlinear problems. After its introduction in~\cite{Lio69} in the deterministic setting, it was extended to the stochastic setting with Gaussian noise in a series of works \cite{BT73, KR79, Par75} (see also the monograph \cite{LR15}).
Of the several abstract conditions in this framework,
we would like to discuss the monotonicity condition in more detail. In the deterministic framework the setting is as follows: consider
\begin{equation}
\label{eq:detde}
    \begin{split}
        u'(t)+A(t,u(t)) &= 0,\\
        u(0)&=u_0,
    \end{split}
\end{equation}
where $A:[0,\infty)\times \cV\to \cV^*$ and $u_0\in \cH$, and where $(\cV, \cH, \cV^*)$ is a Gelfand triple.
The monotonicity conditions typically read
\begin{align*}
-\langle A(t, u)-A(t, v),  u-v\rangle & \leq K \|u-v\|_{\cH}^2 &  \text{(weak monotonicity),}
\\ -\langle A(t, u)-A(t, v),  u-v\rangle& \leq K (1+\|v\|_{\cV})(1+\|v\|_{\cH}^{\gamma})\|u-v\|_{\cH}^2 & \text{(local monotonicity)}.
\end{align*}
The weak condition is more restrictive than the local condition. There have been many attempts to make the monotonicity conditions more flexible. In particular, it is desirable to have that $\|u\|_{\cH}$  and $\|u\|_{\cV}$ appear on the right-hand side as well. This is, for instance, needed in the so-called strong setting (i.e.\ $\cH = H^{1}(\R^d)$, $\cV = H^{2}(\R^d)$, $\cV^* = L^2(\R^d)$). The benefit of this strong setting is that the Sobolev embedding improves, so that it is easier to bound nonlinearities. For instance, such embeddings are needed for proving growth estimates for $A$, which often are of the form
\begin{align}\label{eq:bddAintro}
\|A(t, u)\|_{\cV^*} \leq C (1+\|u\|_{\cV})(1+\|u\|_{\cH}^{b}).
\end{align}
When a noise appears in the equation, the strong setting requires estimates of derivatives of the nonlinearity. From an abstract point of view this requires that the nonlinearity in front of the noise is allowed to have super-linear growth, which will be included in our setting. 

\subsection{Critical nonlinearities}
In the recent paper \cite{agresti2022critical} of Agresti and the third-named author, under the structural condition $A(t,u) = A_0(t,u)u-F(t,u)$, the monotonicity condition and boundedness condition were replaced by the following condition on $F$ (and similarly for the stochastic terms): for all $T>0$ and $n\geq 1$ there is a constant $C_{n,T}$ such that for all $u,v\in \cV$ with $\|u\|_{\cH}, \|v\|_{\cH}\leq n$,
\begin{align}\label{eq:Fintro1}
\|F(t,u) - F(t,v)\|_{\cV^*} &\leq C_{n,T} (1+\|u\|_{\beta}^{\rho} + \|v\|_{\beta}^{\rho}) \|u-v\|_{\beta},
\\ \label{eq:Fintro2} \|F(t,u)\|_{\cV^*} &\leq C_{n,T} (1+\|u\|_{\beta}^{\rho+1}),
\end{align}
where $\|u\|_{\beta} = [\cV^*, \cV]_{\beta}$. Here $\beta\in (1/2,1)$ and $\rho\geq 0$  satisfy the (sub)criticality condition
\begin{align}\label{eq:criticalvar}
(2\beta-1)(\rho+1) \leq 1.
\end{align}
Although the estimate \eqref{eq:Fintro1} is no longer one-sided like the monotonicity condition, it does have $\|u\|_{\beta}$ on the right-hand side.
As a consequence, the following new examples were suddenly included in~\cite{agresti2022critical}: the weak setting of the Cahn--Hilliard equation, $2D$ Navier--Stokes equations and other fluid dynamics models, the strong setting of several equations (Allen--Cahn for $d\in \{1, 2, 3, 4\}$), the Swift--Hohenberg equations, etc. The nonlinearities appearing in these examples are not weakly monotone and sometimes even exhibit critical growth.  Let us emphasize that no compactness is assumed of the embedding $\cV\hookrightarrow \cH$. Therefore, well-posedness of all of the above SPDEs can also be obtained on unbounded domains.

By standard interpolation inequalities one has $\|u\|_{\beta} \leq \|u\|_{\cH}^{2-2\beta} \|u\|_{\cV}^{2\beta-1}$ and thus by \eqref{eq:criticalvar},
\begin{align*}
\|u\|_{\beta}^{\rho+1} \leq \|u\|_{\cH}^{(2-2\beta)(\rho+1)} \|u\|_{\cV}^{(2\beta-1)(\rho+1)}\leq \|u\|_{\cH}^{(2-2\beta)(\rho+1)} (1+\|u\|_{\cV}).
\end{align*}
Combining this with \eqref{eq:Fintro2} shows that $F(t,u)$ satisfies the same type of bound as in \eqref{eq:bddAintro} and explains the form of the criticality condition \eqref{eq:criticalvar}.

Of course, many of the above mentioned concrete equations have been analyzed via other methods, but it is very effective to include them in one single setting. Moreover, the variational framework provides further flexibility: it allows to consider $(\omega,t)$-dependent coefficients and it allows a noise term $B$ which could be of gradient type under a very simple but optimal joint coercivity condition. For some of the other existing methods, these two additions can be quite problematic.

\subsection{Our goal}
The goal of the current paper is to extend the setting in~\cite{agresti2022critical} in several ways:
\begin{enumerate}[(a)]
\item\label{ita:Falpha} Make the conditions \eqref{eq:Fintro1} and \eqref{eq:Fintro2} more flexible by using $\|\cdot\|_{\alpha}$ with $\alpha\in [0,1/2]$ on the left-hand side (see \eqref{eq:Flexibleintro} below).
\item\label{itb:Aalpha} Consider $A_0 = A_L+A_S$, where $A_L$ is the leading part, and $A_S$ could be singular/unbounded in time.
\item\label{itc:Levy} Use L\'evy noise instead of Gaussian noise.
\end{enumerate}
Regarding \eqref{ita:Falpha}, it turns out that the (sub)criticality condition \eqref{eq:criticalvar} can be replaced by $(2\beta-1)(\rho+1) \leq 1+ 2\alpha$, creating extra flexibility.
As a consequence of this flexibility, the Allen--Cahn equation in dimension two can now be considered in the weak setting  (see Example~\ref{ex:AC intro}). In all of the previous works on the variational setting, this critical case was excluded due to the technical fact that the Sobolev embedding $H^1\hookrightarrow L^\infty$ does not hold.

The singular part $A_S$ of \eqref{itb:Aalpha} appears for instance in the skeleton equation in the weak convergence approach to large deviations \cite{BudDup} when applied to stochastic evolution equations. This explicitly appears in Lemma 4.11 of \cite{TV24_large}, where a large deviation result was proved in the setting of \cite{agresti2022critical}. Potentially, the singular part could have other applications as well.

The motivation for \eqref{itc:Levy}  is that L\'evy noise is very natural in real-life applications. There already exists a variational setting for L\'evy noise under the above-mentioned local monotonicity condition \cite{brzezniak2014strong}, so it is very natural to try to provide a similar partial extension of it as was done for the Gaussian case in~\cite{agresti2022critical}. However, as we will see, the jumps introduced by the noise cause delicate problems which we need to overcome.

Finally,  we emphasize that as in~\cite{agresti2022critical}, we do not assume compactness of the embedding $\cV\hookrightarrow \cH$. Therefore, our results are also applicable to SPDEs on unbounded domains.

\subsection{The setting and main result}
In the rest of the paper we are concerned with the stochastic evolution equation
\begin{equation}
\label{eq:spde}
    \begin{split}
        du(t)+A(t,u(t))\,dt =&B(t,u(t))\,dW(t) + \int_{Z} C(t,u(t-),z)\,\Nt(dz,dt),\\
        u(0)=&u_0,
    \end{split}
\end{equation}
where $W$ is a $U$-cylindrical Wiener process, $U$ is a separable Hilbert space, $\Nt(dz,dt) = N(dz,dt) - \nu(dz)dt$ is a Poisson martingale measure with jump measure $N$ and characteristic measure $\nu$ which is $\sigma$-finite on the measure space $(Z,\cZ,\nu)$. In Theorems~\ref{thm:varglobal} and \ref{thm:contdep} we will prove a global existence and uniqueness result, an a-priori bound on the $L^2$-moments, and continuous dependence on the initial data.

In order to give the reader a glimpse of the results of the paper, we present a very special case of our main results below (see Section~\ref{sec:global} for the general case). Below we assume that the terms $A_0$, $B_0$ and $C_0$ are linear for simplicity. However, in the main results they are allowed to be quasi-linear.
\begin{theorem}\label{thm:mainintro}
Let $(\cV, \cH, \cV^*)$ be a Gelfand triple of Hilbert spaces. Suppose that
\begin{align*}
A &= A_0-F, &  B &= B_0+G, & C &= C_0+H,
\\
A_0 & \in \calL(\cV, \cV^*), & B_0&\in \calL(\cV,\calL_2(U,\cH)), & C_0&\in \calL(\cV,L^2(Z,\cH,\nu)),
\end{align*}
and there are constants $\kappa>0$ and $M\geq 0$ such that
\begin{equation*}
\langle A_0 v, v \rangle -\tfrac{1}{2}\|B_0(v)\|_{\calL_2(U,\cH)}^2 - \tfrac{1}{2} \|C_0(v,\cdot)\|_{L^2(Z,\cH;\nu)}^2 \geq \kappa\|v\|_\cV^2 - M \|v\|_\cH^2, \ \ v\in \cV.
\end{equation*}
Suppose $F:\cV\to \cV^*$ and that there are $C,\rho>0$ and $0\leq \alpha\leq \frac12 <\beta\leq 1$ such that
\begin{align}
\label{eq:Flexibleintro}
\|F(u) - F(v)\|_{\alpha} \leq C (1+\|u\|_{\beta}^{\rho} + \|v\|_{\beta}^{\rho}) \|u-v\|_{\beta}  \ \ \text{and} \ \ \langle F(v), v \rangle  \leq  C\|v\|_\cH^2,
\end{align}
for $u,v\in \cV$,
with the (sub)criticality condition $(2\beta-1)(\rho+1) \leq 1 + 2\alpha$.
Let $G:\cV\to \calL_2(U,\cH)$ and $H:\cV\to L^2(Z,\cH;\nu)$ be Lipschitz functions. Then for every $\cF_0$-measurable $u_0\in L^2(\Omega,\cH)$, \eqref{eq:spde} has a unique c\`adl\`ag solution $u:[0,T]\times \Omega\to \cH$ such that
\[\E \sup_{t\in [0,T]} \|u(t)\|_{\cH}^2 + \int_0^T \|u(t)\|^2_{\cV} \, dt\leq C_T(1+\E \|u_0\|_{\cH}^2). \]
Moreover, $u$ depends continuously on $u_0$ in the topology induced by convergence in probability.
\end{theorem}
Due to the flexibility on $F$, the above result is new even in the Gaussian case. In Section~\ref{sec:global} we also cover $(\omega,t)$-dependent coefficients, and $G$ and $H$ are allowed to be locally Lipschitz with a similar bound as for $F$, and thus, in particular, $G$ and $H$ do not need to have linear growth.

\begin{remark}
We remark here that it is possible to consider a more general model instead of \eqref{eq:spde}, by adding an additional jump term $\big(\int_0^t\int_{\tilde Z}D(s,u(s-),z)\,\nu(dz)ds\big)_{t\in [0,T]}$ on the right-hand side of \eqref{eq:spde}, where one integrates over a subset $\tilde Z\subset Z$ such that $\nu(\tilde Z)<\infty$. In such a case, the $D$-term models large jumps, of which there are only finitely  many, almost surely. It is clear that all the mathematical intricacies stem from the ``small'' jumps, i.e. the integral over $Z$ where $\nu(Z) = \infty$ (modeling infinitely many jumps almost surely) and that for most results, adding the $D$ only lengthens the proof without providing essential new difficulties or information. For this reason we decided to omit this term for the sake of length and presentation of the paper.
\end{remark}

Let us give an application of Theorem~\ref{thm:mainintro} to the Allen--Cahn equation in $d=2$ in the weak setting. In the existing frameworks mentioned above, only the case $d=1$ was covered in the weak setting. More general cases can be found in Theorem~\ref{thm:second} and the examples below it.
\begin{example}[Allen--Cahn for $d=2$]
\label{ex:AC intro}
Let $\Dom\subseteq \R^2$ be any open set (possibly unbounded).
\begin{equation*}
\begin{aligned}
d u &=\big[ \Delta u -u^3+u \big]\, d t
+ \sum_{n\geq 1} \big[(b_n\cdot \nabla)u + g_{n}(u)\big] \, d w_t^n \\ &  \quad +  \int_Z \big[(c(z)\cdot \nabla)u(\cdot -) + h(u(\cdot -),z)\big] \, \wt{N}(dz,dt) ,&\quad &\text{ on }\Dom,\\
\end{aligned}
\end{equation*}
with Dirichlet boundary conditions and with an $\cF_0$-measurable initial value $u_0\in L^2(\Omega;L^2(\Dom))$.
Suppose the following parabolicity condition is satisfied:
\[\theta := 1 - \frac12\|(b_n)_{n\geq 1}\|_{\ell^2}^2 - \frac{1}{2} \|c\|_{L^2(Z;\nu)}^2>0.\]
Moreover, suppose that $g:\R\to \ell^2$ and $h:\R\to L^2(Z;\nu)$ are Lipschitz functions. For simplicity in the presentation, we take $(b,c,g,h)$ to be $x$-independent, but this is not necessary.

To put this problem in the form \eqref{eq:spde}, let $\cH = L^2(\Dom)$ and $\cV=H^1_0(\Dom)$. Let $A_0 v = -\Delta v-v$, $(B_0v)_n =(b_n \cdot \nabla) v$ and $C_0v = (c(\cdot)\cdot \nabla)v$. Let $F(v) = -v^3$, $G(v) = g(v)$, $H(v,z) = h(v,z)$.
It is straightforward to check the conditions of Theorem~\ref{thm:mainintro}. The only part we need to explain in detail is the locally Lipschitz estimate for $F$. Note that $[\cV^*, \cV]_{5/6} = [\cH, \cV]_{2/3} \hookrightarrow L^6(\Dom)$, which follows by an extension-restriction argument from Sobolev embedding on the full space (cf. \cite[Lemma A.7]{AV20_NS}). Therefore
\begin{align*}
\|F(u) - F(v)\|_{\nicefrac{1}{2}} = \|u^3-v^3\|_{L^2(\Dom)}
&\leq 2(\|u\|_{L^6(\Dom)}^2+\|v\|_{L^6(\Dom)}^2)\|u-v\|_{L^{6}(\Dom)}
\\ & \leq C(\|u\|_{\nicefrac{5}{6}}^2+\|v\|_{\nicefrac{5}{6}}^2)\|u-v\|_{\nicefrac{5}{6}}
\end{align*}
for any $u,v\in H^{1}_0(\Dom)$.
Thus, the (sub)criticality condition is satisfied with $\alpha=1/2$, $\beta=5/6$ and $\rho = 2$. It is in fact critical in the sense that $(2\beta-1)(\rho+1) = 1 + 2\alpha$.
\end{example}

\subsection{Overview of the paper and the method of proof}

Below we give an overview of the different steps which will be taken in different sections.
In Section~\ref{sec:variational} we introduce a new type of coercivity condition for the triple $(A,B,C)$ which takes into account singular terms and the norm $\|\cdot\|_{\alpha}$ used on the left-hand side of \eqref{eq:Flexibleintro}. Under this coercivity condition, we are able to prove an a-priori estimate which plays a key role in several of the later sections and proofs.
One of those is a stochastic maximal $L^2$-regularity result for the linear problem associated with \eqref{eq:spde}. This is the main result of Section~\ref{sec: linear}.
In Section~\ref{sec: local well-posedness} we extend some of the ideas in~\cite{agresti2022nonlinear} to the case of L\'evy processes. We use a Banach fixed point argument applied to a suitably truncated version of \eqref{eq:spde} to obtain a local solution and extend it to a maximal solution, i.e. a solution on a maximal random time interval $[0,\sigma)$. In Section~\ref{sec:blowup} we characterize the behavior of $u$ at time $\sigma$ via blow-up criteria.  For this, the arguments from \cite{agresti2022nonlinear2, agresti2022critical} need to be put in a discontinuous framework, which leads to several technicalities related to jump processes. Fortunately, the concrete $L^2$-setting combined with the variational framework makes it possible to have effective arguments for this.
In Section~\ref{sec:global}, we combine all of the results to obtain our main theorem on global well-posedness, which in particular entails Theorem~\ref{thm:mainintro}. Ingredients in the proof are
of course the local well-posedness theorem and blow-up criteria, but also the a-priori estimate of Section~\ref{sec:variational}, the It\^o formula of Subsection~\ref{subsec:ito}, and a recent stochastic version of Gronwall's lemma. Many applications are possible with our theory.
We present a selection of them in Section~\ref{sec:applications}.

\subsection{Related literature on extensions of the variational setting}
An extension of the classical variational setting to the L\'evy setting can be found in~\cite{brzezniak2014strong} and a recent article considering singular coefficients in a Gaussian, more classical setting, with monotone operators is \cite{gyongy2025once}. Consequences on ergodicity in the case of additive noise have been obtained recently in~\cite{BT24}. Under the assumption that $\cV\hookrightarrow \cH$ is compact, an extended variational framework was also presented in~\cite{RSZ} in the Gaussian case, later extended to the L\'evy setting in~\cite{kumar2024well}. Under the same compactness condition, another class of equations with L\'evy noise was considered in~\cite{Temametal}. In each of these papers, there are smallness conditions on $B$ and $C$ which are not needed in the classical setting and our setting. At the same time, some of these papers cover important equations which fall out of our setting (e.g.\ the $p$-Laplace equation). Therefore, all frameworks appear to be of independent interest. It would be desirable to have a unifying theory, but this seems beyond reach at the moment.

\section{Preliminaries}
\label{sec:setting}

\subsection{Gelfand triple}
Let $(\cV,\cH,\cV^*)$ be a triple of spaces such that $\cV\hookrightarrow \cH \hookrightarrow \cV^*$ continuously and densely, where $\cV$ and $\cH$ are Hilbert spaces and $\cV^*$ is the dual of $\cV$.
For a Hilbert space $U$ we denote by $\cL(U,\cH)$ the space of bounded linear operators from $U$ to $\cH$ and by $\cL_2(U,\cH)$ the space of Hilbert--Schmidt operators.

For $\theta\in [0,1]$ we set $\cV_\theta :=[\cV^*,\cV]_\theta$, where the bracket denotes the complex interpolation of spaces (see \cite{BeLo} for details).
Define further
$$
\|x\|_\theta:=\|x\|_{\cV_\theta}.
$$
Note also that $\cH=[\cV^*,\cV]_{1/2}$ and thus $[\cV,\cH]_{2\beta} = \cV_{1-\beta}$ for $\beta \in [0,\nicefrac{1}{2}]$ by reiteration.
The following standard interpolation estimates will be used without further explanation:
\begin{align*}
\|x\|_\theta &\leq \|x\|_{\cV^*}^{1-\theta} \|x\|_{\cV}^{\theta}, & \theta\in [0,1],
\\ \|x\|_{\theta} &\leq \|x\|_{\cH}^{2-2\theta} \|x\|_{\cV}^{2\theta-1}, & \theta\in [1/2,1],
\end{align*}

Finally, we note that by \cite[Cor.~4.5.2]{BeLo}, $\cV_{\theta}^* = [\cV^{**}, \cV^*]_\theta = [\cV^*, \cV^{**}]_{1-\theta} = \cV_{1-\theta}$. As a consequence one has
\begin{align}\label{eq:dualityVtheta}
|\langle u, v\rangle|\leq \|u\|_{\theta} \|v\|_{1-\theta},
\end{align}
where $\langle \cdot, \cdot\rangle$ is the unique extension of $(\cdot, \cdot)_\cH$.
We employ the convention $\nicefrac{1}{0} \coloneqq \infty$ throughout.

\subsection{Stochastic calculus}

Throughout this paper, we work on a filtered probability space $(\Omega,(\cF_t)_{t\geq 0},P)$. For brevity we often write $\cF_t$ instead of $(\cF_t)_{t\geq 0}$ when referring to the filtration. We impose the \textit{usual conditions} on $\cF_t$: it is right-continuous and $\cF_0$ is complete. For a topological space $X$ we let $\cB(X)$ denote its Borel $\sigma$-algebra.
For a Hilbert space $\cH$, we call $X$ an \emph{$\cH$-valued random variable} if it is a strongly measurable mapping $X:\Omega\to \cH$. We call $f$ a \emph{process} if $f:\Omega\times\bR_+\to \cH$ is a strongly measurable function. Moreover, we say that $f$ is \emph{progressively measurable} if for any $t\geq 0$ the process $f\1_{[0,t]}$ is strongly $\cB([0,t])\otimes \cF_t$-measurable. We denote by $\cP$ the $\sigma$-algebra generated by all the progressively measurable processes and by $\cP^-$ the $\sigma$-algebra generated by all left-continuous adapted processes.

When we speak of a \cl function $h$, we understand it as $\cH$-valued function of $(\omega,t,w)\in\Omega\times \bR_+\times \cW$ which is \cl in $t$ almost surely for all $w\in\cW$, where the set $\cW$ will always be clear from the context. In this case, we denote the left-limit process by $h(t-,w)$ and the jump process by $\Delta h(t,w):= h(t,w) - h(t-,w)$. We denote the space of \cl functions $f:[0,T]\mapsto\cH$ by $D([0,T],\cH)$.

Next, we introduce the noise processes used in this article. Though their definitions are standard and can be found in many textbooks, we include them here for the sake of completeness.

\begin{definition}[cylindrical Wiener process]
    Let $U$ be a separable Hilbert space and consider a mapping $W\in \cL(L^2(\bR_+,U),L^2(\Omega))$. Then $W$ is a cylindrical $\cF_t$-Wiener process if for all $f,g\in L^2(\bR_+,U)$ and $t>0$,
    \begin{enumerate}
        \item $Wf$ is normally distributed with mean zero, and $\E WfWg = (f,g)_{L^2(\bR_+,U)}$,
        \item $Wf$ is $\cF_t$-measurable if $\supp(f)\subset [0,t]$,
        \item $Wf$ is independent of $\cF_t$ if $\supp(f)\subset [t,\infty)$.
    \end{enumerate}
    Since we work with the same filtration $\cF_t$ throughout the paper, we simply call $W$ a Wiener process for brevity.
\end{definition}
To make the notation for integrals against a Wiener process more concise, define for $h \in \cH$, $T \in \cL_2(U,\cH)$ and $u \in U$ the pairing $\langle h, T \rangle$ by $\langle h, T \rangle u \coloneqq (h,T u)_\cH = \lb T^* h,u\rb$.

For examples of cylindrical Wiener processes and the theory of its stochastic calculus we refer to \cite{da2014stochastic} or any other textbook on stochastic analysis in infinite dimensions.

The following definitions and conventions, as well as more general definitions of random measures, can be found in~\cite{ikeda2014stochastic,jacod2013limit,he2019semimartingale}. We also refer the reader to these works for more details on the theory of jump processes and general random measures. Here we introduce only the Poisson random measure, the Poisson martingale measure and selected properties of integrals against them.

\begin{definition}
    Let $(\tilde Z,\tilde\cZ,\tilde\nu)$ be a $\sigma$-finite measure space. A family of $\bN\cup\{\infty\}$-valued random variables $(N(A))_{A\in\tilde\cZ}$ is called a Poisson random measure with characteristic measure $\tilde\nu$ if
    \begin{enumerate}
    \item for each $A\in\tilde\cZ$ the random variable $N(A)$ has a Poisson distribution with intensity $\tilde\nu(A)$,
        \item for all $\omega\in\Omega$ the measure $N(\cdot)(\omega)$ is a $\sigma$-finite measure on $(\tilde Z,\tilde\cZ)$,
        \item for any $A_1,A_2\in\tilde\cZ$ such that $A_1\cap A_2 = \emptyset$ the random variables $N(A_1)$ and $N(A_2)$ are independent.
    \end{enumerate}
\end{definition}

Throughout this paper, we will work in the setting $(\tilde Z,\tilde\cZ,\tilde\nu)=(\bR_+\times Z,\cB(\bR_+)\otimes\cZ,dt\otimes\nu(dz))$ for a $\sigma$-finite measure space $(Z,\cZ,\nu)$. In this case, we call $\nu$ the characteristic measure of $N$ and often simply refer to $N$ as the Poisson random measure on $(Z,\cZ,\nu)$. The random measure $\Nt(dz,dt):=N(dz,dt)-\nu(dz)dt$ is referred to as the Poisson martingale measure or as the compensated Poisson measure.

The following proposition is a collection of properties and we refer to \cite[Ch. 3]{ikeda2014stochastic} for a proof. For the general theory of integration against (Poisson) random measures we refer the reader to any of the works \cite{ikeda2014stochastic,jacod2013limit, he2019semimartingale,Temametal}, as well as \cite{yaroslavtsev2019martingales} for a recent extension to more general infinite-dimensional spaces.

\begin{proposition}
	\label{prop: prelim poisson}
    Let $N$ be a Poisson random measure  on $(Z,\cZ,\nu)$ and let $h:\Omega\times\bR_+\times Z\to \cH$ be $\cP^-\otimes\cZ$-measurable and $T>0$. Then the following assertions hold.
    \begin{enumerate}
        \item If $\E\int_0^T\int_Z \|h(s,z)\|_\cH\,\nu(dz)ds<\infty$, then
        $$
        \E\int_0^t\int_Z h(s,z)\, \,N(dz,ds) = \E\int_0^t\int_Z h(s,z)\,\nu(dz)ds
        $$
        for $t\in [0,T]$ and in particular
        \begin{align}
            \E\int_0^T\int_Z \|h(s,z)\|_{\cH}\, \,N(dz,ds) = \E\int_0^T\int_Z \|h(s,z)\|_{\cH}\,\nu(dz)ds.
        \end{align}
        Moreover,
        \begin{align}
        \label{eq: equality Nt}
        \int_0^t\int_Z h(s,z)\,\Nt(dz,ds) := \int_0^t\int_Z h(s,z)\,N(dz,ds) - \int_0^t\int_Z h(s,z)\,\nu(dz)ds,
        \end{align}
        for $t\in [0,T]$, is an $\cF_t$-martingale.
        \item If $\E\int_0^T\int_Z \|h(s,z)\|_\cH^p\,\nu(dz)ds<\infty$ with both $p = 1,2$, then
        $$
        \Big[ \int_0^{\cdot}\int_Z h(s,z)\,\Nt(dz,ds) \Big](t) = \int_0^t\int_Z \|h(s,z)\|_\cH^2\,N(dz,ds),\quad t\in [0,T]
        $$
        where $[ \cdot ]$ denotes the quadratic variation process, and in particular
        $$
        \E\Big\| \int_0^t\int_Z h(s,z)\,\Nt(dz,ds) \Big\|^2 = \E\int_0^t\int_Z \|h(s,z)\|_\cH^2\,\nu(dz)ds.
        $$
    \end{enumerate}
\end{proposition}

Further, we remark that if $\E\int_0^T\int_Z \|h(s,z)\|_\cH^2\,\nu(dz)ds<\infty$ only, then \eqref{eq: equality Nt} may no longer hold. In this case we define
\begin{align}
\label{eq: Nt integral}
\int_0^t\int_Z h(s,z)\,\Nt(dz,ds),\quad t\in [0,T],
\end{align}
as the unique limit of $\big(\int_0^t\int_Z h_n(s,z)\,\Nt(dz,ds),\, t\in [0,T]\big)_{n\geq 1}$ in the space of square integrable martingales, where $h_n(s,z):=\1_{\|h(s,z)\|_\cH<n}\1_{Z_n}(z)h(s,z)$,
$n\geq 1$, where $Z_n$ are such that $Z_n\uparrow Z$ and $\nu(Z_n)<\infty$ for all $n\geq 1$. Finally, if $h$ is such that for an increasing sequence of stopping times $\sigma_n$ we have $\E\int_0^{T\wedge \sigma_n}\int_Z \|h(s,z)\|_\cH^2\,\nu(dz)ds<\infty$, then \eqref{eq: Nt integral} is defined as the unique element $X$ in the space of locally square integrable martingales satisfying
$$
X(t\wedge\sigma_n) = \int_0^{t\wedge \sigma_n}\int_Z h(s,z)\,\Nt(dz,ds)\quad\text{a.s. for $t\in [0,T]$, $n\geq 1$.}
$$
It is well-known that $X$ has a c\'adl\'ag version and from now on we will always use that version.

It is standard that the above extension of the integral satisfies
\[\int_0^{t\wedge \sigma}\int_Z h(s,z)\,\Nt(dz,ds) = \int_0^{t}\int_Z \1_{(0,\sigma]}(s) h(s,z)\,\Nt(dz,ds),\]
where $\sigma$ is a stopping time. We will also use the convention that
\begin{align}\label{eq:conventionsigmastart}
\int_{\sigma}^{t}\int_Z h(s,z)\,\Nt(dz,ds) &:= \int_{0}^{t}\int_Z h(s,z)\,\Nt(dz,ds) - \int_0^{t\wedge \sigma} \int_Z h(s,z)\,\Nt(dz,ds)
\\ &= \int_0^{t}\int_Z \1_{(\sigma,\infty)}(s) h(s,z)\,\Nt(dz,ds).
\end{align}

We will frequently use without mention that for a càdlàg function $h$ satisfying Proposition~\ref{prop: prelim poisson}~(1)
we have
$$
\int_0^t\int_Z h(s-,z)\,\nu(dz)ds = \int_0^t\int_Z h(s,z)\,\nu(dz)ds.
$$

For stopping times $0 \leq \tau_1 \leq \tau_2 \leq \infty$ define $$\llbracket \tau_1, \tau_2 \rrbracket \coloneqq \{ (\omega, t) \in \Omega \times [0,\infty) \colon \tau_1(\omega) \leq t \leq \tau_2(\omega) \}.$$ The sets $\llbracket \tau_1, \tau_2 \rrparenthesis$, $\llparenthesis \tau_1, \tau_2 \rrbracket$, $\llparenthesis \tau_1, \tau_2 \rrparenthesis$ are defined similarly by replacing $\leq$ with $<$ in the previous definition. Further, we write $\llbracket\tau_1\rrbracket:=\llbracket \tau_1,\tau_1\rrbracket$ for the graph of $\tau_1$.

\subsection{Itô formula}
\label{subsec:ito}

We provide an Itô formula for equations with \Ly noise when the deterministic forcing terms are allowed to come from an admissible space $L^p([0,T], \cV_\theta)$, where $(p,\theta)$ is an admissible pair in the sense of the following definition.

\begin{definition}
\label{def:admissible}
    Let $p \in [1, 2]$ and $\theta \in [0,1]$. The pair $(p, \theta)$ is called \emph{admissible} if $\theta \geq 1/p - 1/2$.
\end{definition}

The following lemma shows that the pairing between a function $u$ from the maximal regularity class that we are going to use in our Itô formula and $f$ from an admissible space is meaningful.

\begin{lemma}
\label{lem:admissible}
    Fix $T>0$ and let $u \in L^\infty([0,T], \cH) \cap L^2([0,T], \cV)$. Then $u\in L^q([0,T], \cV_\mu)$ provided the pair $(q', 1-\mu)$ is admissible, along with the estimate
    \begin{align}
        \| u \|_{L^q([0,T], \cV_\mu)} \leq \| u \|_{L^2([0,T], \cV)}^{\nicefrac{2}{q}} \|u\|_{L^\infty([0,T],\cH)}^{1 - \nicefrac{2}{q}}.
    \end{align}
    In particular, if $(p,\theta)$ is an admissible pair and $f\in L^{p}([0,T], \cV_{\theta})$, then $\langle u, f \rangle$ is integrable over $[0,T]$.
    To be more precise, one has
    \begin{align}
    \label{eq:admissible lem}
        \int_0^T |\langle u(s), f(s) \rangle| ds &\leq \int_0^T \|u(s)\|_{1-\theta} \|f(s)\|_{\theta} ds
        \\ & \leq \| u \|_{L^{p'}([0,T], \cV_{1-\theta})}  \| f \|_{L^{p}([0,T], \cV_\theta)}
        \\ & \leq \| u \|_{L^2([0,T], \cV)}^{\nicefrac{2}{p'}} \| u \|_{L^\infty([0,T],\cH)}^{1 - \nicefrac{2}{p'}} \| f \|_{L^{p}([0,T], \cV_\theta)},
    \end{align}
    where $p'$ is the Hölder conjugate of $p$, and $\langle \cdot,\cdot \rangle$ denotes the pairing between $\cV$ and $\cV^*$.
\end{lemma}

\begin{proof}
    Let $u$ be as in the statement and let $q \geq 2$ and $\mu \in [0,1]$. There is $\lambda \in [0,1]$ such that $\nicefrac{1}{q} = \nicefrac{(1-\lambda)}{2}$. By interpolation and using the reiteration identity $[\cV,\cH]_{\lambda} = \cV_{1-\nicefrac{\lambda}{2}}$ we find $u \in L^q([0,T], \cV_{1-\nicefrac{\lambda}{2}})$ with the estimate
    \begin{align}
        \| u \|_{L^q([0,T], \cV_{1-\nicefrac{\lambda}{2}})} \leq \| u \|_{L^2([0,T], \cV)}^{1-\lambda} \| u(s) \|_{L^\infty([0,T],\cH)}^\lambda.
    \end{align}
    We claim that $\mu \leq 1-\nicefrac{\lambda}{2}$ if $(q', 1-\mu)$ is admissible. Indeed, by admissibility in the second step and using the definition of $\lambda$ in the last step,
    \begin{align}
        \mu = 1 - (1-\mu) \leq 1 - (\nicefrac{1}{q'} - \nicefrac{1}{2}) = \nicefrac{1}{q} + \nicefrac{1}{2} = 1 - \nicefrac{\lambda}{2}
    \end{align}
    as desired.
    So the first claim of the lemma follows from the continuous embedding $\cV_{1-\nicefrac{\lambda}{2}} \subseteq \cV_\mu$. Note that the exponents are correct by definition of $\lambda$.
    The second part follows from \eqref{eq:dualityVtheta} and Hölder's inequality. Note that we employ the first part with $q \coloneqq p'$ and $\mu = 1-\theta$.
\end{proof}

The following special case of Itô's formula will be enough for our purposes. It partly extends \cite{Par75} to the setting of general martingales. However, we only cover the case of a Gelfand triple of Hilbert spaces.
The main difficulty in the It\^o formula is the mixed smoothness and mixed integrability of all the different terms.
It could be formulated for progressively measurable $f\in L^2(\Omega;L^{1}([0,T], \cH))+ L^2(\Omega;L^{2}([0,T], \cV^*)))$,
but it is not obvious that the progressively measurable subspace of $L^2(\Omega, L^{p}([0,T], \cV_{\theta}))$ for admissible pairs $(p,\theta)$ is included in the latter sum space. In order to avoid this problem we directly work in a suitable sum of progressively measurable subspaces of $L^2(\Omega, L^{p}([0,T], \cV_{\theta}))$ for different admissible pairs.

Before we state the following two propositions, some additional comments are necessary regarding  the quadratic variation process $[M]$ associated to a square integrable martingale $M$, since different conventions are in use. Instead of distinguishing between the operator-valued quadratic variation process $[[M]]$ and its trace $\Tr[[M]]$, as in  \cite{Par75, metivier1976equation}, we directly introduce $[M]$ as the unique process such that $\|M\|_{\cH}^2-[M]$ is a martingale, and in this way $[M] = \Tr [[M]]$.
We further denote by $\langle M\rangle$ the unique process such that with $M^c$ the continuous martingale part, $\|M^c\|_{\cH}^2-\langle M\rangle$ is a martingale. \newline
For a proof of the following Burkholder-Davis-Gundy inequality we refer to \cite{marinelli2016maximal}, or to
\cite[Theorem 9.1.1]{yaroslavtsev2019martingales}. Note that the $\cH$-valued case can also be deduced from the scalar case via \cite{KalSzt}.
\begin{proposition}
\label{prop: BDG}
    Let $M:\Omega\times\bR_+\mapsto\cH$ be a local càdlàg  martingale with $M(0) = M_0$. Then for any $p\geq 1$ and any $\cF_t$-stopping time $\tau$ we have
    \begin{align}
       c\E[M]_\tau^{p/2}\leq
       \E\sup_{0\leq s\leq \tau}\|M_s\|^p_{\cH} \leq C \E [M]_\tau^{p/2}
    \end{align}
    for constants $c,C$ depending only on $p$.
\end{proposition}

The following estimates for integrating $\cH$ valued processes against $\cH$-valued martingales are well-known. For a proof of the general martingale case in the scalar setting, with condition $\E\int_0^\infty \|u(t)\|_{\cH}^2d\langle M\rangle(t)<\infty$, we refer to \cite{meyer2002cours}. The vector-valued case can be derived by using an orthonormal basis expansion, and the almost sure finiteness condition used below, stated also in~\cite{gyongy1980stochastic}, is the result of a standard approximation by stopping times. Already in the two-dimensional case, one does not have equality in both situations, since there could be cancellations in the inner products.

\begin{proposition}
    \label{prop: quad var estimate}
    Let $u$ be a progressive $\cH$-valued process, let $M$ be a square integrable $\cH$-valued càdlàg martingale  and assume that $\int_0^\infty \|u(t)\|_{\cH}^2d\langle M\rangle(t) <\infty$ almost surely. Then we have almost surely
\begin{align}
    \left[\int_0^{\cdot} u(s-)\,dM(s)\right](t)  \leq   \int_0^t \|u(s-)\|_{\cH}^2\,d [ M] (s),
\end{align}
    as well as
    \begin{align}
        \E\left(\int_0^\infty u(t-)\,dM(t) \right)^2
         \leq \E\int_0^\infty \|u(t)\|_{\cH}^2\,d\langle M\rangle (t),
    \end{align}
    where the last right-hand side is allowed to be infinite.
\end{proposition}

\begin{proposition}[Itô formula]
\label{prop ito}
    Let $u_0 \colon \Omega \to \cH$ be a strongly $\cF_0$-measurable random variable, let $f = \sum_{j=1}^m f_j$ where $m\geq 1$, $f_j \in L^2(\Omega, L^{p_j}([0,T], \cV_{\theta_j}))$ is progressively measurable for each $1\leq j \leq m$, with $(p_j, \theta_j)$ admissible for all $j$, and let $M$ be a \cl  square integrable $\cF_t$-martingale with values in $\cH$ which
    satisfies $M(0) = 0$.
    Let $$u \in L^2(\Omega, D([0,T], \cH)) \cap L^2(\Omega\times [0,T], \cV)$$ be a progressively measurable process satisfying a.s. for all $t\in [0,T]$
    \begin{align}
        u(t) = u_0 + \int_0^t f(s) ds + M(t).
    \end{align}
    Then a.s for all $t\in [0,T]$
    \begin{align}\label{eq:Itosquare}
        \| u(t) \|_\cH^2 = \| u_0 \|_\cH^2 + 2 \int_0^t \langle f(s), u(s) \rangle ds + 2 \int_0^t u(s-) dM(s)  +[M](t).
    \end{align}
\end{proposition}

\begin{proof}
    First we observe that by the conditions imposed on $u$ and $M$ we clearly have that almost surely
    $$\int_0^T\|u(t)\|_{\cH}^2\,d\langle M\rangle(t)<\infty,$$
    so that the following stochastic integrals are well-defined.
    The proof uses ideas from \cite{Par75} and \cite{Krysimple}. By \cite[Proposition 8.1.10]{TayPDE2} we can find an invertible positive self-adjoint operator $A$ on $\cV^*$ with $D(A) = \cV$. Let $R_n = n(n+A)^{-1}$. Then by interpolation one has $C_{\theta}:=\sup_{n\geq 1}\|R_n\|_{\cL(\cV_{\theta})}<\infty$ for all $\theta\in [0,1]$. It is standard to check that $R_n$ strongly converges to the identity on each $\cV_{\theta}$. Let $u^n = R_n u$ and similarly for $u_0$, $f$ and $M$. These are regularized versions of our data and all take values in $\cH$. Then for these regularized objects $(u^n, u_0^n, f^n, M^n)$, we can apply the usual $\cH$-valued It\^o formula to deduce that \eqref{eq:Itosquare} holds (for instance \cite[Proposition 3]{metivier1976equation} applies). In the above special case, one can even reduce to the scalar case by writing
    $\| u^n(t) \|_\cH^2 = \sum_{k\geq 1} |(u^n, e_k)_\cH|^2$, where $(e_k)_{k\geq 1}$ is an orthonormal basis for $\cH$, by applying the scalar valued It\^o formula for semimartingales, \cite[Thm. 4.57]{jacod2013limit}, to the scalar process $(u^n,e_k)_{\cH}$.
    Indeed, its application and straight-forward manipulations give
    \begin{align*}
    |(u^n(t), e_k)_\cH|^2 = |(u^n_0, e_k)_\cH|^2 & + 2 \int_0^t  (f^n(s),e_k)_\cH (u^n(s),e_k)_\cH ds  \\ & + 2\int_0^t (u^n(s-), e_k)_{\cH} d (M^n(s), e_k)_{\cH} + [(M^n, e_k)_\cH]_t,
    \end{align*}
    where $u^n(\cdot-)$ is to be read as the left-limit process of the regularization $u^n$.
    Summing over all $k\geq 1$, we see that
    \begin{align*}
        \| u^n(t) \|_\cH^2 = \| u_0^n \|_\cH^2 &+ 2 \sum_{k\geq 1} \int_0^t (f^n(s),e_k)_\cH (u^n(s),e_k)_\cH ds \\ & + 2\sum_{k\geq 1}\int_0^t (u^n(s-), e_k)_{\cH} d (M^n(s), e_k)_{\cH} + \sum_{k\geq 1} [(M^n, e_k)_\cH](t),
    \end{align*}
    where we still need to check the convergence of the series in probability.
    It is clear that $\sum_{k\geq 1} [(M^n, e_k)_\cH]_t = [M^n](t)$ a.s. To calculate the two integral terms let $M^{n,\ell} = \sum_{k=1}^\ell (M^n(s), e_k) e_k$ and define $u^{n,\ell}$ and $f^{n,\ell}$ similarly. By linearity, we can write
    \[\sum_{k=1}^\ell \int_0^t (f^n(s),e_k)_\cH (u^n(s),e_k)_\cH ds = \int_0^t \langle f^{n,\ell}(s), u^{n}(s) \rangle ds\to \int_0^t \langle f^{n}(s), u^{n}(s) \rangle ds \ \text{a.s.}  \]
    as $\ell \to \infty$, since $f^{n,\ell}\to f^n$ in $L^1([0,T],\cH)$ and $u^n\in D([0,T],\cH)$ almost surely.
    Similarly,
    \[\sum_{k=1}^\ell \int_0^t (u^n(s-), e_k) d (M^n(s), e_k) = \int_0^t u^{n,\ell}(s-) dM^n(s)\to \int_0^t u^{n}(s-) dM^n(s) \ \text{in $L^1(\Omega)$}.\]
    Indeed, this follows from Propositions~\ref{prop: BDG} and \ref{prop: quad var estimate} and
    \[\E \Big[\int_0^\cdot u^n(s-) - u^{n,\ell}(s-) dM^n(s)\Big]^{1/2}(t) \leq \E\Big(\int_0^t \|u^{n}(s-) - u^{n,\ell}(s-)\|^2_{\cH} d[M^n](s)\Big)^{1/2}\to 0\]
    by the Dominated Convergence Theorem, which is applicable owing to $u^{n,\ell}\to u^{n}$ in $\cH$ a.e. and $\|u^n - u^{n,\ell}\|_\cH\leq \|u^n\|_\cH\leq \|u\|_{\cH}$ and the assumption $u\in L^2(\Omega;D([0,T];\cH))$.
    It remains to let $n\to \infty$ in
    \begin{align*}
        \| u^n(t) \|_\cH^2 = \| u_0^n \|_\cH^2 + 2 \int_0^t \langle f^n(s), u^n(s) \rangle ds + 2 \int_0^t u^n(s-) dM^n(s)  + [M^n](t).
    \end{align*}
    Almost everywhere convergence for $\|u^n(t)\|_\cH^2$ and $\|u_0^n\|_\cH^2$ follow from the properties of $R_n$.

    The convergence of $\int_0^t \langle f^n(s), u^n(s) \rangle  ds$ is more cumbersome. For each $j\in \{1, \ldots, m\}$ it is clear that $\langle u^n(s), f_j^n(s)\rangle\to \langle u(s), f_j(s)\rangle$ and that
    \[|\langle f^n_j(s), u^n(s) \rangle| \leq \|f^n_j(s)\|_{\theta} \|u^n(s)\|_{1-\theta} \leq C_{\theta} \|f_j(s)\|_{\theta} \|u(s)\|_{1-\theta}\]
    pointwise in $s\in [0,T]$. Therefore, a.s.\ in $\Omega$, the required convergence follows from the dominated convergence theorem since the latter is integrable by Lemma~\ref{lem:admissible}.

    To show the convergence of $\int_0^t u^n(s-) dM^n(s)$ it suffices to show
    \[\int_0^t u(s-) - u^n(s-) dM(s) \to 0 \ \text{and} \ \int_0^t u(s-) d(M-M^n)(s) \to 0\ \ \text{in $L^1(\Omega)$}.\]
    Moreover, by Propositions~\ref{prop: BDG} it is enough to prove that the quadratic variations tend to zero in $L^{1/2}(\Omega)$. Using Proposition~\ref{prop: quad var estimate}, for the first term this follows from
    \[\E\Big(\Big[\int_0^\cdot u(s-) - u^n(s-) dM(s)\Big](t)\Big)^{1/2} \leq  \E\Big(\int_0^t \|u(s-) - u^n(s-)\|^2_{\cH} d[M](s)\Big)^{1/2}\to 0,\]
    by the Dominated Convergence Theorem.
For the second term, we see that
\begin{align}
    \E\Big[\int_0^\cdot u(s-) d(M-M^n)(s)\Big]^{1/2}(t) &\leq \E\Big(\int_0^t \|u(s-)\|^2_\cH d [M-M^n](s)\Big)^{1/2} \\
    &\leq \E \sup_{s\in [0,T]}\|u(s-)\|_\cH [M-M^n](T)^{1/2}
   \to 0 \
\end{align}
by H\"older's inequality since $u\in L^2(\Omega, D([0,T], \cH))$ and $\E[M-M^n](T) \eqsim  \E \|M - M^n\|_{\cH}^2\to 0$, where we used Propositions~\ref{prop: BDG} and the Dominated Convergence Theorem.
\end{proof}

\begin{remark}
By localization, one can remove the integrability conditions concerning $u$ in $\Omega$ in Proposition~\ref{prop ito} a posteriori.
\end{remark}

\begin{corollary}
\label{cor: ito}
    Let the assumptions of Proposition~\ref{prop ito} hold.
    Let
    \begin{equation*}
        g \in L^2 (\Omega\times [0,T],\cL_2(U,\cH)),\quad\text{and}\quad h \in L^2(\Omega\times [0,T],L^2(Z,\cH;\nu))
    \end{equation*}
    be $\cP$-measurable and $\cP^-\otimes\cZ$-measurable, respectively.
    Let $u$ be the process from Proposition~\ref{prop ito} with
    $$
    M(t) :=\int_0^t g(s)\,dW(s) + \int_0^t\int_Z h(s,z)\,\Nt (dz,ds),\quad t\in [0,T].
    $$
    Then we have almost surely, for all $t\in [0,T]$,
     \begin{align}
        \| u(t) \|_\cH^2 = \| u_0 \|_\cH^2 &+ 2\int_0^t \langle f(s), u(s)\rangle ds
        + \int_0^t \|g(s)\|_{\calL_2(U,\cH)}^2\,ds
        + 2 \int_0^t \lb u(s),g(s)\rb \,dW(s) \\
        &+ 2\int_0^t\int_{Z} (u(s-),h(s,z))_\cH  \,\Nt(dz,ds)
         + \int_0^t\int_{Z} \|h(s,z)\|^2_{\cH}\,N(dz,ds).
    \end{align}
\end{corollary}
\begin{proof}  Though it is a well-known argument how to reduce the result to Proposition~\ref{prop ito}, for completeness we include the details. It suffices to compute
    \begin{equation*}
        2\int_0^t u(s-)\,dM(s) + [M](t).
    \end{equation*}
    Clearly, almost surely for all $t\in [0,T]$,
    \begin{align}
    \label{eq: martingale integral}
    \int_0^t u(s-)\,dM(s) &= \int_0^t \lb u(s),g(s)\rb \,dW(s) + \int_0^t\int_Z (u(s-),h(s,z))_\cH\,\Nt(dz,ds).
    \end{align}
    By orthogonality we have
    \begin{align}
    [M](t)
    &= [M_1](t) + [M_2](t).
    \end{align}
    Clearly
    $$
    [ M_1](t) = \int_0^t\|g(s)\|_{\calL_2(U,\cH)}^2\,ds.
    $$
    Moreover, by Proposition~\ref{prop: prelim poisson} we have
    \begin{align}
        [M_2](t) = \int_0^t\int_{Z} \|h(s,z)\|_\cH^2\, N(dz,ds),
    \end{align}
    which finishes the proof.

\end{proof}

The following corollary is not needed in the further course of this article. We discuss its relevance below in Remark~\ref{rem:ito_cor}.

\begin{corollary}
    \label{cor: ito 2}
    Let the assumptions of Corollary~\ref{cor: ito} hold and assume additionally that
    \begin{equation}
    \label{eq: assumptions ito 2}
    \begin{split}
        \int_0^T\int_Z\big| \|u(s-)+h(s,z)\|^2_\cH - \|u(s-)\|_\cH^2 \big|^2\,\nu(dz)ds&<\infty\\
        \int_0^T\int_Z \|h(s,z)\|^4_{\cH} \,\nu(dz)ds&<\infty.
        \end{split}
    \end{equation}
    Then we have almost surely, for all $t\in [0,T]$,
     \begin{align}
     \nonumber
        \| u_t \|_\cH^2 = &\| u_0 \|_\cH^2 + 2\int_0^t \langle f(s), u(s)\rangle ds
        + \int_0^t \|g(s)\|_{\calL_2(U,\cH)}^2\,ds    + 2 \int_0^t (u(s),g(s))_\cH \,dW(s)    \\
        \label{eq: ito 2}
        & + \int_0^t\int_{Z} \|u(s-)+h(s,z)\|^2_\cH - \|u(s-)\|_\cH^2\,\Nt(dz,ds)  + \int_0^t\int_{Z} \|h(s,z)\|^2_{\cH}\,\nu(dz)ds .
    \end{align}
\end{corollary}
\begin{proof}
A standard limiting argument, using that $h\in L^2 ([0,T],L^2(Z,\cH;\nu))$, gives
\begin{align*}
    &\int_0^t\int_Z  \|h(s,z)\|^2\,N(dz,ds)
    \\
    ={} &\int_0^t\int_Z  \|h(s,z)\|^2_{\cH}\,\Nt(dz,ds)
    +\int_0^t\int_Z  \|h(s,z)\|^2 \,\nu(dz)ds\\
    ={} &\int_0^t\int_Z  \|u(s-)+h(s,z)\|^2-\|u(s-)\|_\cH^2\,\Nt(dz,ds) \\
    &\quad - 2\int_0^t\int_Z (u(s-),h(s,z))_\cH\,\Nt(dz,ds)
     +\int_0^t\int_Z  \|h(s,z)\|^2_{\cH} \,\nu(dz)ds
\end{align*}
which, together with the second term on the right-hand side of  \eqref{eq: martingale integral} finishes the proof.
\end{proof}

\begin{remark}
\label{rem:ito_cor}
    As is remarked in~\cite{gyongy2021ito}, in many publications the It\^o formula~\eqref{eq: ito 2} is stated without prior assumption of the additional conditions~\eqref{eq: assumptions ito 2}. However, without them, the integral against the Poisson martingale measure on the right-hand side of~\eqref{eq: ito 2} may fail to exist, as is demonstrated in~\cite[Example 2.1]{gyongy2021ito}.
\end{remark}

\section{The extended variational setting}
\label{sec:variational}

\noindent Fix $T>0$. In this section, we consider the variational problem
\begin{equation}
\label{eq:spde apriori}
\tag{VP}
    \begin{split}
        du(t)+\wt{A}(t,u(t)) \, dt &= \wt{B}(t,u(t)) \, dW(t) +\int_{Z} \wt{C}(t,u(t-),z) \, \Nt(dz,dt) \\
        u(0) &= u_0
    \end{split}
\end{equation}
on $[0,T]$. The properties of the operators $\wt{A}$, $\wt{B}$ and $\wt{C}$ are stated in Assumption~\ref{assumption apriori operators}.
The main result of this section will be an a-priori estimate for~\eqref{eq:spde apriori}, see Proposition~\ref{prop: nonlin apriori}.

\subsection{Setting and notion of solution}
\begin{assumption}
    \label{assumption apriori operators}
    For the operators $\wt{A}$, $\wt{B}$ and $\wt{C}$ we assume the following.
    \begin{enumerate}
        \item The mappings
        \begin{align}
            \wt A &\colon \Omega \times [0,T]\times \cV \to \cV^*, \\
            \wt B &\colon \Omega \times [0,T]\times \cV \to \cL_2(U,\cH), \\
            \wt C &\colon \Omega \times [0,T]\times \cV\times Z \to \cH
        \end{align}
        are $\cP\otimes\cB(\cV)$-measurable, $\cP\otimes \cB(\cV)$-measurable and $\cP^-\otimes\cB(\cV)\otimes \cZ$-measurable, respectively, and such that almost surely
        \begin{align}
        \label{eq:ass integrals exist}
            \int_0^T \|\wt A(t,v(t))\|_{\cV^*}\,dt & + \int_0^T\|\wt B(t,v(t))\|_{\calL_2(U,\cH)}^2\,dt \\ & + \int_0^T \|\wt C(t,v(t),\cdot)\|_{L^2(Z,\cH;\nu)}^2 dt<\infty
        \end{align}
        for all $v\in D([0,T],\cH)\cap L^2([0,T],\cV)$.
        \item There are constants $\kappa, \eta > 0$, a non-negative function $\phi\in L^1([0,T])$, a non-negative function $\psi \in L^0(\Omega, L^1([0,T]))$, and finitely many non-negative functions $\psi_i \in L^0(\Omega, L^{p_i}([0,T]))$ with $(\theta_i, p_i)$ admissible, such that a.s. for all $v \in \cV$ and almost every $t\in [0,T]$
        \begin{align}
        \label{eq: coercivity apriori}
            \<\wt A(t,v), v\> - (\frac12 + \eta) &\|\wt B(t,v) \|_{\calL_2(U,\cH)}^2 - (\frac12 +  \eta) \| \wt C(t,v,\cdot) \|_{L^2(Z,\cH;\nu)}^2 \\ & \geq \kappa \| v \|_\cV^2 - \phi(t) \| v \|_\cH^2 - \psi(t) - \sum_i \psi_i(t) \| v \|_{1-\theta_i}.
        \end{align}
    \end{enumerate}
\end{assumption}

\begin{remark}
    To the best of our knowledge, the coercivity condition~\eqref{eq: coercivity apriori} is new. If $\psi_i$ is associated with the admissible pair $(\theta_i,p_i)\coloneqq (0,2)$, then Young's inequality allows to absorb $\| v \|_{1-\theta} = \| v \|_\cV$ into $\kappa \| v \|_\cV^2$, so that we recover the traditional lower bound $\kappa\| v \|_\cV^2 - \phi(t) \| v \|_\cH^2 - \psi(t)$.
    For all other admissible pairs, such a reduction would lead to a dependence of $\phi$ on $\psi_i$, which would lead to the wrong a-priori estimate later on, compare with Proposition~\ref{prop: nonlin apriori} and Lemma~\ref{lem:lin coercive}.
\end{remark}

\noindent Two prototypical examples are the following: First, $(\wt A, \wt B, \wt C)$ can be taken as the nonlinear operators $(A, B, C)$ from the quasilinear problem~\eqref{eq:spde} subject to a nonlinear coercivity condition. See Sections~\ref{sec:blowup} and~\ref{sec:global} for details. Second, $\wt A$, $\wt B$ and $\wt C$ can be taken as linear operators perturbed by an inhomogeneity. This case will be studied in the subsequent Section~\ref{sec: linear} and will lead to a well-posedness result for linear problems with \Ly noise. Assumption~\ref{assumption apriori operators} is flexible enough to capture these two cases at the same time and to provide unified a-priori estimates for them.

\begin{definition}
Let Assumption~\ref{assumption apriori operators} hold, let $T\in (0,\infty]$, let $\sigma$ be a stopping time with values in $[0,T]$, and let $\tau$ be a stopping time with values in $[\sigma,T]$. Let $u_{\sigma}$ be $\cH$-valued and $\cF_{\sigma}$-measurable. We call $u: \llbracket \sigma,\tau\rrbracket  \rightarrow \cV$ a \emph{strong solution} to \eqref{eq:spde apriori} if
        almost surely $u\in D([\sigma,\tau],\cH)\cap L^2([\sigma,\tau],\cV)$ and it satisfies a.s. for all $\sigma \leq t \leq\tau$
        \begin{equation}
        \label{eq: strong solution}
        \begin{split}
            u(t)  =  u_{\sigma} + \int_{\sigma}^t \wt{A}(s,u(s))\,ds &+ \int_{\sigma}^t \wt{B}(s,u(s))\,dW(s) + \int_{\sigma}^t\int_{Z} \wt{C}(s,u(s-),z)\,\Nt(dz, ds),
            \end{split}
        \end{equation}
        where we recall the convention \eqref{eq:conventionsigmastart} for the random left-end point of the stochastic integrals.
        Since we are only dealing with strong solutions in this paper, we will henceforth simply refer to them as \textit{solutions}.
\end{definition}
\begin{definition}
Let Assumption~\ref{assumption apriori operators} hold, let $T\in (0,\infty]$, let $\sigma$ be a stopping time with values in $[0,T]$, and let $\tau$ be a stopping time with values in $[\sigma,T]$. Let $u_{\sigma}$ be $\cH$-valued and $\cF_{\sigma}$-measurable. Let $u: \llbracket \sigma, \tau \rrparenthesis \rightarrow \cV$.
\label{def solution}
\begin{enumerate}
     \item We call $(u,\tau)$ a \emph{local solution} to \eqref{eq:spde apriori} on $\llbracket \sigma, T\rrbracket$ if there is an increasing sequence of stopping times $\{\tau_n\}_{n=1}^\infty$, $\tau_n\uparrow\tau$, such that for each $n\geq 1$, the restriction $u|_{\llbracket\sigma,\tau_n\rrbracket}$ is a solution to \eqref{eq:spde apriori} on $\llbracket \sigma, \tau_n \rrbracket$. In this case we call $\{\tau_n\}_{n=1}^\infty$ a \emph{localising sequence} for $u$.
    \item A local solution $(u,\tau)$ on $\llbracket \sigma, T \rrbracket$ is \emph{unique} if for any other local solution $(\wt u,\wt\tau)$ on $\llbracket \sigma, T\rrbracket$ the identity $\wt u=u$ holds $P\otimes dt$-almost everywhere on $\llbracket \sigma,\tau\wedge\wt\tau \rrparenthesis$.
    \item A unique local solution $(u,\tau)$ on $\llbracket \sigma, T \rrbracket$ is \emph{maximal} if for any other local solution $(\wt u,\wt\tau)$ on $\llbracket \sigma, T \rrbracket$ we have $\wt\tau\leq\tau$ and $\wt u=u$, $P\otimes dt$-almost everywhere on $\llbracket \sigma,\wt\tau\rrparenthesis$.
    \item A local solution $(u,\tau)$ on $\llbracket \sigma, \infty \rrparenthesis$ is called {\em global} if $\tau=\infty$.
    \end{enumerate}
\end{definition}

\subsection{An a-priori estimate}

Using the Itô formula from Section~\ref{subsec:ito}, we show an a-priori estimate for the variational problem~\eqref{eq:spde apriori}. Later on, we are going to use it on the one hand to show an a-priori estimate for linear problems, see Proposition~\ref{prop: lin apriori}, and on the other hand to check the blow up criterion for our critical non-linear problem~\eqref{eq:spde} in Section~\ref{sec:global}.

\begin{proposition}
\label{prop: nonlin apriori}
    Let Assumption~\ref{assumption apriori operators} hold with an $\eta>0$.
    Suppose that $(u,T)$ is a strong solution to
    \begin{align}
    \label{eq: spde estimate lemma}
        du(t) + \wt{A}(t,u(t)) \,dt =  \wt{B}(t,u(t))\,dW(t) + \int_Z \wt{C}(t,u(t-),z) \,\Nt(dz, dt)
    \end{align}
    with $u_0 \coloneqq u(0) \colon \Omega \to \cH$ strongly $\cF_0$-measurable and $\wt{A}(t,u(t)) = \sum_{i=1}^{m_a} \wt{A}_i(t,u(t))$. Assume further for each $i$ the moment condition
    \begin{align}
    \label{eq: integrability condition ito nonlinear}
            \E \Bigl( \int_0^T \| \wt A_i(t,u(t)) \|_{\alpha_i}^{q_i} \, dt \Bigr)^{\nicefrac{2}{q_i}} & + \E \int_0^T \|\wt B(t,u(t)) \|_{\calL_2(U,\cH)}^2 \, dt \\ & + \E \int_0^T \| \wt C(t,u(t),\cdot) \|_{L^2(Z,\cH;\nu)}^2 dt < \infty,
        \end{align}
    where each $(\alpha_i, q_i)$ is an admissible pair.
    Let $0 \leq \tau_1\leq\tau_2 \leq T$ be stopping times.
    Then
    \begin{align}
    \label{eq: sup estimate nonlinear lemma}
        \E \sup_{t \in [\tau_1,\tau_2]} &\| u(t) \|_\cH^2
        +\E\int_{\tau_1}^{\tau_2}\|u(t)\|_\cV^2\,dt
        \\ & +  \E\int_{\tau_1}^{\tau_2} \|\wt B(t,u(t)) \|_{\calL_2(U,\cH)}^2\,dt
             + \E\int_{\tau_1}^{\tau_2} \|\wt C(t,u(t),\cdot)\|_{L^2(Z,\cH;\nu)}^2 dt \\
        &\leq{}  C \Bigl( \E \| u(\tau_1) \|_\cH^2 + \E \| \psi \|_{L^1([\tau_1,\tau_2])} + \sum_{i=1}^{m}\E \| \psi_i \|_{L^{p_i}([\tau_1,\tau_2])}^2 \Bigr) ,
    \end{align}
    where $C = C(T,\eta, \kappa,\|\phi\|_{L^1([0,T])})$ is a constant and $(\theta_i, p_i)$ are the admissible pairs corresponding to $\psi_i$ from Assumption~\ref{assumption apriori operators}.
\end{proposition}

\begin{proof}
    Throughout the proof, let $C = C(T,\eta, \kappa,\|\phi\|_{L^1([0,T])})$ be a constant that can change its value from line to line.
    To ease notation, we only treat the case $\tau_1 = 0$, $\tau_2 = T$, but the more general case follows verbatim with the same proof.

    Assume in a first instance that
    \begin{equation}
    \label{eq: moment u}
        \E\sup_{t\in [0,T]}\|u(t)\|_\cH^2 + \int_0^T\|u(s)\|_\cV^2\,ds<\infty.
    \end{equation}
    By It\^o's formula (Corollary~\ref{cor: ito}) we have a.s. for all $t\in [0,T]$ that
    \begin{align}
    \label{eq: apriori ito nonlinear}
    \begin{split}
        \| u(t) \|_\cH^2 ={} &\| u_0 \|_\cH^2 - 2 \int_0^t   \langle \wt A(s, u(s)), u(s) \rangle \,ds + \int_0^t \|\wt B(s,u(s)) \|_{\calL_2(U,\cH)}^2\, ds \\
        &+ 2 \int_0^t \langle \wt B(s,u(s)), u(s) \rangle \,dW(s) \\
        &+ 2\int_0^t\int_Z (\wt C(s,u(s-),z),u(s-))_\cH\,\Nt(dz,ds)
        \\
        &+\int_0^t\int_Z  \| \wt C(s,u(s-),z)\|^2_{\cH} \,N(dz,ds).
    \end{split}
    \end{align}
    Note that~\eqref{eq: integrability condition ito nonlinear} ensures the integrability condition in Corollary~\ref{cor: ito}. In particular, the integrals on the right-hand side of \eqref{eq: apriori ito nonlinear} are well-defined martingales with vanishing expectations.

    \textbf{Step 1}: preliminary estimates.
    Taking the expectation and using the coercivity condition \eqref{eq: coercivity apriori} in conjunction with standard properties of the martingales we immediately get
    \begin{align}
    \label{eq: bound sup E u nonlinear}
        &\E\| u(t) \|_\cH^2 + 2\eta\E\int_0^t \|\wt B(s,u(s)) \|_{\calL_2(U,\cH)}^2\,ds \\ &  \qquad \qquad + 2\eta\E\int_0^t\int_Z\|\wt C(s,u(s),z)\|^2_\cH\,\nu(dz)ds + 2\kappa\E\int_0^t \|u(s)\|_\cV^2\,ds \\
        \leq{} &\E\| u_0 \|_\cH^2 +2\E\int_0^t \phi(s)\|u(s)\|_\cH^2\,ds + 2\E\int_0^t\psi(s)\,ds + \sum_{i=1}^m 2\E\int_0^t \psi_i(s) \| u(s) \|_{1-\theta_i}\,ds.
    \end{align}
    Since $\phi$ is positive and non-random, we calculate with Fubini--Tonelli
    \[\E \int_0^t \phi(s) \| u(s) \|_\cH^2 \,ds = \int_0^t \phi(s) \E \| u(s) \|^2_\cH ds.\]
    Therefore, we can apply Gr\"onwall's inequality to $\E \| u(s) \|^2_{\cH}$ to obtain
    \begin{gather}
    \label{eq: apriori sup1}
        \E\| u(t) \|_\cH^2
        \leq C\Bigl( \E\|u_0\|^2_\cH + \E\|\psi\|_{L^1([0,T])} + \sum_{i=1}^m \E \int_0^t \psi_i(s) \| u(s) \|_{1-\theta_i} \,ds \Bigr),
    \end{gather}
    where we import a dependence on $T$ and $\|\phi\|_{L^1([0,T])}$ into the constant $C$.
    The last two displayed equations also imply that
    \begin{align}
    \label{eq: apriori M u term}
        &\E \int_0^t \phi(s) \| u(s) \|_\cH^2 \,ds \\ & \leq C \| \phi \|_{L^1([0,T])} \Bigl( \E\|u_0\|^2_\cH + \E\|\psi\|_{L^1([0,T])} + \sum_{i=1}^m \E \int_0^t \psi_i(s) \| u(s) \|_{1-\theta_i} \,ds \Bigr).
    \end{align}
    Plugging this back into~\eqref{eq: bound sup E u nonlinear} yields
    \begin{align}
    \label{eq: apriori sup2}
        &\E\| u(t) \|_\cH^2 + 2\eta\E\int_0^t \|\wt B(s,u(s)) \|_{\calL_2(U,\cH)}^2\,ds \\ &\qquad + 2\eta\E\int_0^t\int_Z\|\wt C(s,u(s),z)\|^2_\cH\,\nu(dz)ds + 2\kappa\E\int_0^t \|u(s)\|_\cV^2\,ds \\
        & \leq  C\Bigl( \E\|u_0\|^2_\cH + \E\|\psi\|_{L^1([0,T])} + \sum_{i=1}^m \E \int_0^t \psi_i(s) \| u(s) \|_{1-\theta_i} \,ds \Bigr).
    \end{align}

    Fix $i$ in the sum on the right-hand side of~\eqref{eq: apriori sup2}. We want to further estimate the term $\E \int_0^t \psi_i(s) \| u(s) \|_{1-\theta_i} \,ds$.
    Recall that $\psi_i$ is associated with the admissible pair $(p_i, \theta_i)$.
    Using Hölder's inequality and Young's inequality, we
 deduce
    \begin{align}
        \E \int_0^t \psi_i(s) \| u(s) \|_{1 - \theta_i} \,ds &\leq \Bigl( \E \Bigl( \int_0^t \psi_i(s)^{p_i} \,ds \Bigr)^\frac{2}{p_i} \Bigr)^\frac{1}{2} \Bigl( \E \Bigl( \int_0^t \| u(s) \|_{1 - \theta_i}^{p_i'} \,ds \Bigr)^\frac{2}{p_i'} \Bigr)^\frac{1}{2} \\
        &\leq C_\eps \E \| \psi_i \|_{L^{p_i}([0,T])}^2 + \eps \E \| u \|_{L^{p_i'}([0,t], \cV_{1 - \theta_i})}^2.
    \end{align}
    By admissibility of $(p_i, \theta_i)$ and using Lemma~\ref{lem:admissible} one has the bound
    \begin{align}
        \E \| u \|_{L^{p_i'}([0,t], \cV_{1 - \theta_i})}^2 \leq C \bigl(\E \| u \|_{L^2([0,t], \cV)}^2 + \E\| u \|_{L^\infty([0,t], \cH)}^2 \bigr),
    \end{align}
    for a numerical constant $C > 0$.
    In summary, this gives
    \begin{align}
    \label{eq: apriori psi_j estimate}
        &\sum_{i=1}^m \E \int_0^t \psi_i(s) \| u(s) \|_{1 - \theta_i} \,ds \\
        \leq{} & \eps \E \| u \|_{L^2([0,t], \cV)}^2 + \eps \E \| u \|_{L^\infty([0,t], \cH)}^2 + \sum_{i=1}^m C_\eps \E \| \psi_i \|_{L^{p_i}([0,T])}^2.
    \end{align}
    We plug this back into~\eqref{eq: apriori sup2}. The squared $L^2(\Omega \times [0,t], \cV)$-norm appearing in~\eqref{eq: apriori psi_j estimate} is finite by~\eqref{eq: moment u}, so we can absorb it into the left-hand side of~\eqref{eq: apriori sup2} by choosing $\eps$ sufficiently small. The resulting estimate then reads
    \begin{align}
    \label{eq: apriori sup3}
        &\E\| u(t) \|_\cH^2 + 2\eta\E\int_0^t \|\wt B(s,u(s)) \|_{\calL_2(U,\cH)}^2\,ds \\ & + 2\eta\E\int_0^t\int_Z\|\wt C(s,u(s),z)\|^2_\cH\,\nu(dz)ds + \kappa\E\int_0^t \|u(s)\|_\cV^2\,ds \\
        & \leq C\Bigl( \E\|u_0\|^2_\cH + \E\|\psi\|_{L^1([0,T])} + C_\eps \sum_{i=1}^m \E \| \psi_i \|_{L^{p_i}([0,T])}^2 + \eps \E \sup_{s \leq t} \| u(s) \|_\cH^2 \Bigr).
    \end{align}
    For the moment we are not yet in the position to absorb also the term $\eps \E \sup_{s \leq t} \| u(s) \|_\cH^2$ into the left-hand side due to the wrong order of supremum and expectation for the term on the left-hand side. This will be our next task in the course of this proof.

    \textbf{Step 2}: estimate for $\E \sup_{s \leq t} \| u(s) \|_\cH^2$.
    To obtain an estimate for $\E\sup_{s\leq t}\|u(s)\|_\cH^2$ we have to bound the supremum of the martingale term. This uses a common technique based on the Burkholder--Davis--Gundy inequality. This technique is quite classical and the difference in estimating the Brownian integral to estimating the compensated Poisson integral
    is only the additional integration over $Z$. Hence we only present the treatment of the Poisson martingale term in detail.
    In this sense, by virtue of the Burkholder--Davis--Gundy inequality and Hölder's inequality we get for $\delta>0$
    \begin{align}
        &\E\sup_{s\leq t}\Big|\int_0^s\int_Z \big( \wt C(r,u(r-),z),u(r-)\big)_\cH\,\Nt(dz,dr) \Big| \\
        \leq{} &C \E\Big(\int_0^t\|( \wt C(s,u(s),\cdot),u(s))_\cH\|^2_{L^2(Z;\nu)} ds  \Big)^{1/2}\\
        \leq{} &C\E\sup_{s\leq t}\|u(s)\|_\cH\Big(\int_0^t\| \wt C(s,u(s),\cdot)\|_{L^2(Z,\cH;\nu)}^2ds\Big)^{1/2}\\
        \leq{} &\delta\E\sup_{s\leq t}\|u(s)\|_\cH^2 + C_{\delta} \E\int_0^t\| \wt C(s,u(s),\cdot)\|_{L^2(Z,\cH;\nu)}^2 ds.
    \end{align}
    Now we plug~\eqref{eq: apriori sup2} into the previous bound to get
    \begin{align}
    \label{eq: burkholder N tilde}
        &\E\sup_{s\leq t}\Big|\int_0^s\int_Z \big( \wt C(r,u(r-),z),u(r-)\big)_\cH\,\Nt(dz,dr) \Big| \\
        \leq{} &\delta\E\sup_{s\leq t}\|u(s)\|_\cH^2 + C_\delta C \Big(\E\|u_0\|^2_\cH + \E\|\psi\|_{L^1([0,T])} + \sum_{i=1}^m C_\eps \E \| \psi_i \|_{L^{p_i}([0,t])}^2 + \eps \E \sup_{s \leq t} \| u(s) \|_\cH^2 \Big) \\
        ={} &(\delta + C_\delta C \eps) \E\sup_{s\leq t}\|u(s)\|_\cH^2 + C_\delta C \Big(\E\|u_0\|^2_\cH + \E\|\psi\|_{L^1([0,T])} + \sum_{i=1}^m C_\eps \E \| \psi_i \|_{L^{p_i}([0,t])}^2 \Big).
    \end{align}
    Note that $C$ is now additionally depending on $\eta$.
In the same way we also obtain
\begin{align}
  \label{eq: bound dW term apriori nonlinear}
 \E \sup_{s \leq t} \Big| \int_0^s \langle u(r), &\wt{B}(r)u(r) \rangle \, dW(r) \Big| \\
        &\leq{} (\delta + C_\delta C \eps) \E\sup_{s\leq t}\|u(s)\|_\cH^2 \\
        \,&+ C_\delta C \Big(\E\|u_0\|^2_\cH + \E\|\psi\|_{L^1([0,T])} + \sum_{i=1}^m C_\eps \E \| \psi_{i} \|_{L^{p_i}([0,t])}^2 \Big).
\end{align}
 Next, by virtue of the coercivity condition \eqref{eq: coercivity apriori}, followed by~\eqref{eq: apriori M u term},~\eqref{eq: apriori psi_j estimate} and~\eqref{eq: apriori sup3}, we obtain
\begin{align}
\label{eq: apriori coercivity terms}
    &\E\sup_{s\leq t} \Big| - 2 \int_0^s   \langle \wt A(r) u(r), u(r) \rangle \, dr + \int_0^s \|\wt B(r) u(r)\|_{\calL_2(U,\cH)}^2\, dr \Big| \\
    \leq{} &2\E\int_0^t \phi(s)\|u(s)\|_\cH^2\,ds + 2\E\int_0^t\psi(s)\,ds + \sum_{i=1}^m 2\E \int_0^t \psi_i(s) \| u(s) \|_{1-\theta_i} \,ds \\
    \leq{} &C\E\|u_0\|^2_\cH + C\E\|\psi\|_{L^1([0,T])} + \sum_{i=1}^m C\E \int_0^t \psi_i(s) \| u(s) \|_{1-\theta_i} \,ds \\
    \leq{} &C\E\|u_0\|^2_\cH + C\E\|\psi\|_{L^1([0,T])} + \sum_{i=1}^m CC_\eps \E \| \psi_i \|_{L^{p_i}([0,t])}^2 + \eps C \E \| u \|_{L^\infty([0,t], \cH)}^2 + C \E \| u \|_{L^2([0,t], \cV)}^2 \\
    \leq{} &C\E\|u_0\|^2_\cH + C\E\|\psi\|_{L^1([0,T])} + \sum_{i=1}^m CC_\eps \E \| \psi_i \|_{L^{p_i}([0,t])}^2 + \eps C \| u \|_{L^\infty([0,t], \cH)}^2,
\end{align}
where the last step imports a dependence on $\kappa$ for $C$.
Finally, using that $N(dz,dr)$ is a non-negative measure, followed by Proposition~\ref{prop: prelim poisson} and~\eqref{eq: apriori sup2}, we deduce
\begin{align}
\label{eq: apriori final C term}
&\E\sup_{s\leq t} \Bigl| \int_0^s\int_Z \|\wt C(r,u(r-),z)\|_\cH^2\,N(dz,dr) \Bigr| \\
={} &\E\int_0^t\int_Z \|\wt C(r,u(r-),z)\|_\cH^2\,N(dz,dr) \\
={} &\E\int_0^t\int_Z \|\wt C(r,u(r),z)\|_\cH^2\,\nu(dz)ds \\
\leq{} &C\Bigl( \E\|u_0\|^2_\cH + \E\|\psi\|_{L^1([0,T])} + \sum_{i=1}^m C_\eps \E \| \psi_i \|_{L^{p_i}([0,t])}^2 + \eps \E \sup_{s \leq t} \| u(s) \|_\cH^2 \Bigr).
\end{align}
Thus, taking first the supremum in time and then the expectation on both sides of \eqref{eq: apriori ito nonlinear}, and using
\eqref{eq: burkholder N tilde},~\eqref{eq: bound dW term apriori nonlinear},~\eqref{eq: apriori coercivity terms} and~\eqref{eq: apriori final C term}, we derive
\begin{align}
\label{eq: bound E sup u nonlinear apriori 1}
    &\E\sup_{s\leq t}\|u(s)\|_\cH^2 \\
    \leq{} &C C_\delta \E\|u_0\|^2_\cH + (\delta + CC_\delta \eps) \E\sup_{s\leq t}\|u(s)\|_\cH^2 + C C_\delta \E\|\psi\|_{L^1([0,T])} + \sum_{i=1}^m C C_\eps C_\delta \E \| \psi_i \|_{L^{p_i}([0,t])}^2.
\end{align}
Now choose first $\delta$ and afterwards $\eps$ (relative to $C C_\delta$) sufficiently small to absorb the term $\E\sup_{s\leq t}\|u(s)\|_\cH^2$, which is finite by~\eqref{eq: moment u}, into the left-hand side, to give the first desired estimate
\begin{align}
    \E\sup_{s\in [0,T]}\|u(s)\|_\cH^2 & \leq C \Bigl( \E\|u_0\|^2_\cH + \E\|\psi\|_{L^1([0,T])} + \sum_{i=1}^m \E \| \psi_i \|_{L^{p_i}([0,T])}^2 \Bigr),
\end{align}
however still under the assumption \eqref{eq: moment u}.
By inserting the last bound into~\eqref{eq: apriori sup3} we also find
\begin{align}
        &\E\int_0^T \|u(s)\|_\cV^2\,ds + \E\int_0^T \|\wt B(s,u(s)) \|_{\calL_2(U,\cH)}^2\,ds + \E\int_0^T\int_Z\|\wt C(s,u(s),z)\|^2_\cH\,\nu(dz)ds \\
        \leq{} &C \Bigl( \E\|u_0\|^2_\cH + \E\|\psi\|_{L^1([0,T])} + \sum_{i=1}^m \E \| \psi_i \|_{L^{p_i}([0,T])}^2 \Bigr).
    \end{align}

\textbf{Step 3}: elimination of condition~\eqref{eq: moment u}.
For the general case we first recall that a strong solution $(u,T)$ to \eqref{eq: spde estimate lemma} satisfies $u\in D([0,T],\cH)\cap L^2([0,T],\cV)$ almost surely, so that the stopping times
\begin{align}
    \tau_n:=\inf \bigl\{t\in [0,T]: \sup_{s\in [0,t]}\|u(s)\|_\cH^2 + \int_0^t\|u(s)\|_\cV^2\,ds\geq n \bigr\}\wedge T
\end{align}
verify $\tau_n\to T$ almost surely as $n\to\infty$.
To see that also at $\tau_n$ we have integrability of $\|u(\tau_n)\|_{\cH}^2$ we first observe that
\begin{align}
\Delta u(\tau_n) &= \Delta \int_0^{\tau_n}\int_Z \wt{C}(t,u(t-),z)\,\wt{N}(dz,dt) \\
&= \int_0^{\tau_n}\int_Z \wt{C}(t,u(t-),z)\,\wt{N}(dz,dt)
- \int_0^{\tau_n-}\int_Z \wt{C}(t,u(t-),z)\,\wt{N}(dz,dt)
\end{align}
almost surely. Hence, by  a similar calculation as in \eqref{eq: apriori final C term}, we get
\begin{align}
    \E\|\Delta u(\tau_n)\|_{\cH}^2
    &\leq 4\E\int_0^{\tau_n}\int_Z \|\wt{C}(t,u(t-),z)\|_{\cH}^2\,\nu(dz)dt<\infty
\end{align}
since $\|u(t-)\|_{\cH}^2<n$ on $\llbracket 0,\tau_n\rrbracket$ by definition of $\tau_n$. Hence we have $\E\sup_{s\leq \tau_n}\|u(s)\|_{\cH}^2<\infty$, so that on $[0,\tau_n]$ the integrability condition \eqref{eq: moment u} holds.
Applying the result from Step~2 to the problem \eqref{eq: spde estimate lemma} with both sides stopped after $\tau_n$
yields for all $n\geq 1$,
\begin{gather*}
\label{eq: sup estimate with stopping times}
        \E \sup_{t \in [0,\tau_n]} \| u(t) \|_\cH^2
        +\E\int_0^{\tau_n}\|u(s)\|_\cV^2\,ds
        +  \E\int_0^{\tau_n} \|\wt B(s,u(s)) \|_{\calL_2(U,\cH)}^2\,ds \\
        + \E\int_0^{\tau_n}\int_Z \|\wt C(s,u(s),z)\|^2_\cH\,\nu(dz)ds
        \leq C \E \| u_0 \|_\cH^2 + C\E \| \psi \|_{L^1([0, T])} + \sum_{i=1}^m C \E \| \psi_i \|_{L^{p_i}([0,T])}^2,
    \end{gather*}
for all $n\in\bN$, where we recall that the constant $C$ does not depend on $n$.
Hence, by Fatou's lemma, we obtain
\begin{align}
    \E \sup_{s \in [0,T]}\|u(s)\|_\cH^2 & = \E\sup_{s \in [0,T]}\liminf_{n\to\infty} \|u(s\wedge\tau_n)\|_\cH^2\\
    &\leq \E\liminf_{n\to\infty} \sup_{s\leq \tau_n} \|u(s\wedge\tau_n)\|_\cH^2\\
    &\leq \liminf_{n\to\infty}\E \sup_{s\leq \tau_n} \|u(s\wedge\tau_n)\|_\cH^2\\
    &\leq C \E \| u_0 \|_\cH^2 + C\E \| \psi \|_{L^1([0,T])} + \sum_{i=1}^m C\E \| \psi_i \|_{L^{p_i}([0,T])}^2.
\end{align}
A similar application of Fatou's lemma to the remaining terms finishes the proof.
\end{proof}

\begin{remark}
    Proposition~\ref{prop: nonlin apriori} explains the significance of the different terms in the coercivity condition~\eqref{eq: coercivity apriori}. As a first guideline, the multiplicative constant in~\eqref{eq: sup estimate nonlinear lemma} depends on $\phi$, whereas the $\psi$-terms give additive terms on the right-hand side of the estimate. The proof highlights that the treatment of the $\psi_i$-terms is more challenging than the more traditional $\psi$-term. The $\psi_i$-term correspond to non-standard forcing terms that do not belong to $L^2([0,T], \cV^*)$, compare with Lemma~\ref{lem:lin coercive} and Proposition~\ref{prop: lin apriori}.
\end{remark}

\section{Stochastic maximal $L^2$-regularity}
\label{sec: linear}

\noindent Fix $T>0$. In this section we consider the linear problem
\begin{equation}
\label{eq:spde lin}
\tag{LP}
    \begin{split}
        du(t)+\wt{A}_0(t) u(t) \,dt &= f(t) \,dt + \bigl(\wt{B}_0(t)u(t) + g(t)\bigr) \, dW(t)\\
        &\qquad+\int_{Z} \big(\wt{C}_0(t,z) u(t-) + h(t,z)\big) \, \Nt(dz,dt) \\
        u(0) &= u_0
    \end{split}
\end{equation}
on $[0,T]$.
Here $\wt{A}_0 = \wt{A}_L + \wt{A}_S$, where
\begin{align*}
\wt{A}_L:\Omega\times[0,T]&\to \cL(\cV, \cV^*) \ \ \text{is $\cP$-measurable,}
\\ \wt{A}_S:\Omega\times[0,T]&\to \cL(\cV_{\beta_{A}}, \cV_{\alpha_{A}}) \ \ \text{is $\cP$-measurable,}
\\ \wt{B}_0:\Omega\times[0,T]&\to \cL(\cV_{\beta_{B}}, \cL_2(U, \cH)) \ \ \text{is $\cP$-measurable,}
\\ \wt{C}_0:\Omega\times[0,T]&\to \cL(\cV_{\beta_{C}}, L^2(Z, \cH;\nu)) \ \ \text{is $\cP^-$-measurable,}
\end{align*}
where  the parameters satisfy $0 \leq \alpha_{A} \leq \nicefrac{1}{2}$ and $\beta_{A},\beta_{B},\beta_{C} \in [\nicefrac{1}{2}, 1]$.
Additionally, we introduce the following assumption.
\begin{assumption}
    \label{assumption linear operators}
    For the operators $\wt{A}_0$, $\wt{B}_0$ and $\wt{C}_0$ we assume the following.
    \begin{enumerate}
        \item There are a constant $C_A\geq 0$ and non-negative functions $K_{A} \in L^{r_A}([0,T])$, $K_{B} \in L^{r_B}([0,T])$ and $K_{C} \in L^{r_C}([0,T])$ such that a.s. for all $v \in \cV$ and $t\in [0,T]$
        \begin{align}
        \label{eq: bounded lin wellpos}
            \| \wt{A}_L(t) v \|_{\cV^*} &\leq C_A \| v \|_{\cV},
            \\
            \| \wt{A}_S(t) v \|_{\alpha_{A}} &\leq K_{A}(t) \| v \|_{\beta_{A}}, \\
            \|\wt{B}_0(t) v\|_{\calL_2(U,\cH)} &\leq K_{B}(t) \| v \|_{\beta_{B}}, \\
             \| \wt{C}_0(t,\cdot) v \|_{L^2(Z,\cH;\nu)} &\leq K_{C}(t) \| v \|_{\beta_{C}},
        \end{align}
        where $r_A = {(1 + \alpha_{A} - \beta_{A})^{-1}}$, $r_B = (1 - \beta_{B})^{-1}$ and $r_C = (1 - \beta_{C})^{-1}$ with the convention that $0^{-1} =\infty$.

        \item There are a constant $\kappa > 0$ and a non-negative function $\phi \in L^1([0,T])$ such that a.s. for all $v \in \cV$ and almost every $t\in [0,T]$,
        \begin{align}
        \label{eq: coercivity lin wellpos}
            \< v, \wt{A}_0(t)v \> - \frac12 \|\wt{B}_0(t)v \|_{\calL_2(U,\cH)}^2 - \frac12  \| \wt{C}_0(t,&\cdot)v \|_{L^2(Z,\cH;\nu)}^2 \\
            &\geq \kappa \| v \|_\cV^2 - \phi(t) \| v \|_\cH^2.
        \end{align}
    \end{enumerate}
\end{assumption}
\begin{remark}
    We could also allow $\wt{A}_S = \sum_{i=1}^{m_a} \wt{A}_S^i$, where $\wt{A}_S^i$ is subject to the above assumptions with $\alpha_A$ and $\beta_A$ depending on $i$, for every $i$. To simplify notation, we stick to the case $m_a = 1$. The same holds for $\wt{B}_0$ and $\wt{C}_0$.
\end{remark}

A prototypical example for Assumption~\ref{assumption linear operators} will be $\wt{A}_0(t) = A_0(t, u(t))$, $\wt{B}_0(t) = B_0(t, u(t))$ and $\wt{C}_0(t,z) = C_0(t,u(t),z)$, where the operators $A_0$, $B_0$ and $C_0$ are as in our main result on local existence and uniqueness (Theorem~\ref{thm wellposedness local}) and $u$ is some suitable given process.

\subsection{Compatibility with the variational setting}
\label{subsec:compatibility_var}

In the first step, we show that if~\eqref{eq:spde lin} is subject to Assumption~\ref{assumption linear operators}, then it can be captured within the framework presented in Section~\ref{sec:variational}. As a consequence, the (non-)linear a-priori estimate from Proposition~\ref{prop: nonlin apriori} translates to the current setting. We will state it in the next subsection and conclude the existence and uniqueness  of~\eqref{eq:spde lin} with it.

The verification of Assumption~\ref{assumption apriori operators} will be done in the next two lemmas, which are relatively straightforward.

\begin{lemma}
\label{lem: operator bounds}
    Suppose Assumption~\ref{assumption linear operators}.
    Set $p_A = (\alpha_{A} + \nicefrac{1}{2})^{-1}$.
    Then $(p_A, \alpha_{A})$ is admissible and for all $v \in L^\infty([0,T], \cH) \cap L^2([0,T], \cV)$ it holds a.s.
    \begin{align}
        \| \wt{A}_S v \|_{L^{p_A}([0,T], \cV_{\alpha_{A}})} &\leq \| K_{A} \|_{L^{r_A}([0,T])} \| v \|_{L^2([0,T], \cV)}^{2\beta_{A}-1} \| v \|_{L^\infty([0,T], \cH)}^{2-2\beta_{A}}.
\\       \| \wt{B}_0 v \|_{L^{2}([0,T], \cL_2(U,\cH))} &\leq \| K_{B} \|_{L^{r_B}([0,T])} \| v \|_{L^2([0,T], \cV)}^{2\beta_{B}-1} \| v \|_{L^\infty([0,T], \cH)}^{2-2\beta_{B}}, \\
        \| \wt{C}_0 v \|_{L^2([0,T],L^2(Z,\cH;\nu))} &\leq \| K_{C} \|_{L^{r_C}([0,T])} \| v \|_{L^2([0,T], \cV)}^{2\beta_{C}-1} \| v \|_{L^\infty([0,T], \cH)}^{2-2\beta_{C}}.
    \end{align}
\end{lemma}

\begin{proof}
    The calculations for $\wt{B}_0$ and $\wt{C}_0$ are similar, so we concentrate on $\wt{A}_S$.
    Set $\nicefrac{1}{q_A} = \beta_{A} - \nicefrac{1}{2}$ and write $\nicefrac{1}{p_A} = (1 + \alpha_{A} - \beta_{A}) + (\beta_{A} - \nicefrac{1}{2}) = \nicefrac{1}{r_A} + \nicefrac{1}{q_A}$. Then Assumption~\ref{assumption linear operators} and Hölder's inequality yield a.s.
    \begin{align}
    \label{eq:wt A Hoelder}
        \| \wt{A}_S v \|_{L^{p_A}([0,T], \cV_{\alpha_{A}})} \leq \left( \int_0^T (K_{A}(t) \| v(t) \|_{\beta_{A}})^{p_A} dt \right)^\frac{1}{p_A} \leq \| K_{\wt{A}_0} \|_{L^{r_A}([0,T])} \| v \|_{L^{q_A}([0,T], \cV_{\beta_{A}})}.
    \end{align}
    To conclude, we want to estimate $\| v \|_{L^{q_A}([0,T], \cV_{\beta_{A}})}$ using Lemma~\ref{lem:admissible}. This requires that $(q'_A,1-\beta_{A})$ is an admissible pair. Indeed, we have $$1-\beta_{A} = 1 - \frac{1}{2} - \frac{1}{q_A} = \frac{1}{q'_A} - \frac{1}{2},$$
    where we used the definition of $q_A$ in the first step. This completes the proof.
\end{proof}

For the Lemma below, recall that $\wt{A}_0 = \wt{A}_L + \wt{A}_S$ in Assumption~\ref{assumption linear operators}.

\begin{lemma}
    \label{lem:lin coercive}
    Suppose that the operators $\wt{A}_0$, $\wt{B}_0$ and $\wt{C}_0$ satisfy Assumption~\ref{assumption linear operators}.
    Let $f_i\in L^0(\Omega,L^{p_i}([0,T],\cV_{\theta_i}))$, $g\in L^0(\Omega,L^2([0,T],\cL_2(U,\cH)))$ and $h\in L^0(\Omega,L^2([0,T], L^2(Z,\cH;\nu)))$ be $\cP$-measurable, $\cP$-measurable and $\cP^-$-measurable, respectively, where $(p_i,\theta_i)$ is an admissible pair for each $i\in \{1, \ldots, m_f\}$. Set $f = \sum_{i=1}^{m_f} f_i$ and define  for $v\in \cV$ and $t\in [0,T]$
    \begin{align}
        \wt{A}(t,v) = \wt{A}_L(t)v+\wt{A}_S(t)v - f, \quad \wt{B}(t,v) = \wt{B}_0(t)v + g(t), \quad \wt{C}(t,v,z) = \wt{C}_0(t,z)v + h(t,z).
    \end{align}
    Then the triple $(\wt{A}, \wt{B}, \wt{C})$ satisfies
    Assumption~\ref{assumption apriori operators} as well as the moment condition~\eqref{eq: integrability condition ito nonlinear}. More precisely, the coercivity condition~\eqref{eq: coercivity apriori} holds for some $\eta > 0$  depending on $\kappa$ with $$\psi(t) = C\Bigl( \|g(t)\|_{\calL_2(U,\cH)}^2 + \| h(t,\cdot) \|_{L^2(Z,\cH;\nu)}^2 \Bigr), \qquad \psi_i(t) = \| f_i(t) \|_{\theta_i},$$ almost surely, where $C$ is a numerical constant. The $L^1$-norm of $\phi$ depends now additionally on $\kappa$ and the norms of $K_{B}$ and $K_{C}$ from Assumption~\ref{assumption linear operators}.
\end{lemma}

\begin{proof}
    First we check part two of Assumption~\ref{assumption apriori operators}.
    An expansion of the left-hand side of~\eqref{eq: coercivity apriori} leads almost surely for all $v\in \cV$ and almost every $t\in [0,T]$ to
    \begin{align}
    \label{eq: expand coercivity linear}
        &\langle v, \wt{A}_0(t) v \rangle - \frac{1}{2} \bigl\|\wt{B}_0(t) v \bigr\|_{\calL_2(U,\cH)}^2 - \frac{1}{2} \| \wt{C}_0(t,\cdot) v \|_{L^2(Z,\cH;\nu)}^2 \\
        -\, &\langle v, f(t) \rangle - (\nicefrac{1}{2} + \eta) \|g(t)\|_{\calL_2(U,\cH)}^2 - (\nicefrac{1}{2} + \eta) \| h(t,\cdot) \|_{L^2(Z,\cH;\nu)}^2 \\
        -\, &(1 + 2\eta) \langle \wt{B}_0(t) v, g(t) \rangle -  (1 + 2\eta) \int_Z \bigl(\wt{C}_0(t,z) v, h(t,z) \bigr)_{\cH} \, \nu(dz) \\
        -\, &\eta \bigl\|\wt{B}_0(t) v \bigr\|_{\calL_2(U,\cH)}^2 - \eta \| \wt{C}_0(t,\cdot) v \|_{L^2(Z,\cH;\nu)}^2 \\
        &\eqqcolon I - II - III - IV.
    \end{align}
    By Assumption~\ref{assumption linear operators}, $$I \geq \kappa\| v \|_\cV^2 - \phi(t) \| v \|_\cH^2,$$ with $\kappa$ and $\phi$ taken from that assumption. Concerning the remaining terms, we  either absorb them into the lower bound for $I$ or capture them in the additional $\psi$-terms appearing in the coercivity condition~\eqref{eq: coercivity apriori}. The absorption turns out to be possible provided $\eta \leq 1$ is chosen small enough. We only present the terms $\langle v, f_i(t) \rangle$, $(1 + 2\eta) \langle \wt{B}_0(t) v, g(t) \rangle$ and $\eta \|\wt{B}_0(t) v \|_{\calL_2(U,\cH)}^2$ in detail, the remaining terms follow by similar arguments.

    We start with the term $\eta \|\wt{B}_0(t) v \|_{\calL_2(U,\cH)}^2$. Using the growth condition for $\wt{B}_0$ and the interpolation inequality we find
    \begin{align}
        \bigl\|\wt{B}_0(t) v \bigr\|_{\calL_2(U,\cH)} \leq K_{B}(t) \| v \|_{\beta_{B}} \leq K_{B}(t) \| v\|_\cH^{2-2\beta_{B}} \| v \|_\cV^{2\beta_{B} - 1}.
    \end{align}
    Hence, with Young's inequality,
    \begin{align}
        \eta \bigl\|\wt{B}_0(t) v \bigr\|_{\calL_2(U,\cH)}^2 &\leq \eta K_{B}(t)^2 (\| v \|_\cH^2)^{2-2\beta_{B}} (\| v \|_\cV^2)^{2\beta_{B} - 1} \\
        &\leq \eta(2-2\beta_{B}) K_{B}(t)^\frac{2}{2-2\beta_{B}} \| v \|_\cH^2 + \eta (2\beta_{B} - 1) \| v \|_\cV^2.
    \end{align}
    Note that $K_{B}(t)^\frac{2}{2-2\beta_{B}} = K_{B}(t)^{r_B}$ is integrable by assumption. Recall that $\beta_{B}\geq \nicefrac{1}{2}$. Hence, if we choose $\eta$ small enough (depending on $\kappa$ and the number of terms), we can indeed absorb the term $\eta \|\wt{B}_0(t) v \|_{\calL_2(U,\cH)}^2$ into the lower bound for $I$. The new function $\phi$ will then incorporate the term $K_{B}(t)^{r_B}$ as well.

    Let us continue with the term $(1 + 2\eta) \langle \wt{B}_0(t) v, g(t) \rangle$. We can ignore the prefactor since it is bounded by $3$.
    Re-using the calculations from the first term and applying Young's inequality give
    \begin{align}
    \label{eq:coercivity B and g}
        |\langle \wt{B}_0(t) v, g(t) \rangle| &\leq K_{B}(t) \| v \|_\cH^{2-2\beta_{B}} \| v \|_\cV^{2\beta_{B} - 1} \|g\|_{\calL_2(U,\cH)} \\
        &\leq \frac{1}{2} K_{B}(t)^2 (\| v \|_\cH^2)^{2-2\beta_{B}} (\| v \|_\cV^2)^{2\beta_{B} - 1} + \frac{1}{2}\|g\|_{\calL_2(U,\cH)}^2 \\
        &\leq \eps \| v \|_\cV^2 + C_\eps K_{B}(t)^{r_B} \| v \|_\cH^2 + \frac{1}{2}\|g\|_{\calL_2(U,\cH)}^2.
    \end{align}
    Choosing $\eps$ small enough, we can absorb the first term of the right-hand side of the last inequality of~\eqref{eq:coercivity B and g} into $\kappa \| v \|_\cV^2$. Up to a multiplicative constant, the second term on the right-hand side of ~\eqref{eq:coercivity B and g} was already added to $\phi(t)$ in the first step of the proof. Finally, the term $\frac{1}{2}\|g \|_{\calL_2(U,\cH)}^2$ is integrable and can thus be captured by the $\psi$-term in the coercivity condition~\eqref{eq: coercivity apriori}. This concludes the treatment of the term $-(1 + 2\eta) \langle \wt{B}_0(t) v, g(t) \rangle$.

    We conclude the proof with the term $\langle v, f_i(t) \rangle$ for some fixed $i$. Using the interpolation inequality we obtain
    \begin{align}
        \label{eq: f coercive}
        |\langle v, f_i(t) \rangle| &\leq \| f_i(t) \|_{\theta_i} \| v \|_{1-\theta_i}.
    \end{align}
    Since $(\theta_i,p_i)$ is an admissible pair, the right-hand side of~\eqref{eq: f coercive} can be treated by putting $\psi_i = \| f_i(t) \|_{\theta_i}$.

    Finally, we check the first part of Assumption~\ref{assumption apriori operators} as well as~\eqref{eq: integrability condition ito nonlinear}. The measurability property in Assumption~\ref{assumption apriori operators} follows by assumption.
    Next, we check~\eqref{eq: integrability condition ito nonlinear}.
    To this end, let $v \in L^\infty([0,T], \cH) \cap L^2([0,T], \cV)$. We only present the case for $\wt{A}_S$, the other terms are similar.
    Use Lemma~\ref{lem: operator bounds} to find a.s.
    \begin{align}
        \label{eq:growth_A0}
        \| \wt{A}_S(t) v(t) \|_{L^{p_A}([0,T], \cV_{\alpha_{A}})} \leq \| K_{A} \|_{L^{r_A}([0,T])} \| v \|_{L^2([0,T],\cV)}^{2\beta_{A}-1} \| v \|_{L^\infty([0,T],\cH)}^{2-2\beta_{A}},
    \end{align}
    where we employ the notation from Lemma~\ref{lem: operator bounds}.
    Use Young's inequality and take the expectation to finish the verification of~\eqref{eq: integrability condition ito nonlinear}.
    It remains to check~\eqref{eq:ass integrals exist}.
    Due to the embedding $L^{p_i}([0,T],\cV_{\theta_i}) \subseteq L^1([0,T], \cV^*)$, the integrability conditions for $f$ and $\wt{A}_L$ are clear. For the remaining term $\wt{A}_S$, we use the same embedding in conjunction with~\eqref{eq:growth_A0} to conclude. This finishes this proof.
\end{proof}

\subsection{Wellposedness for the linear problem}

In the main result of this section, Theorem~\ref{thm L2 linear}, we show well-posedness of the linear problem~\eqref{eq:spde lin}. This is based on the a-priori estimate from Section~\ref{sec:variational}. We verified in the last subsection that~\eqref{eq:spde lin} can be captured by the framework developed in Section~\ref{sec:variational}. To summarize, the a-priori estimate in the current setting then reads as follows and is immediate from Proposition~\ref{prop: nonlin apriori} and Lemma~\ref{lem:lin coercive}.

\begin{proposition}[linear a-priori estimate]
\label{prop: lin apriori}
    Let Assumption~\ref{assumption linear operators} hold.
    Let $f = \sum_{i=1}^m f_i$, where $f_i \in L^2(\Omega, L^{p_i}([0,T], \cV_{\theta_i}))$ is $\cP$-measurable with $(p_i, \theta_i)$ an admissible pair, $g\in L^2(\Omega \times [0,T], \cL_2(U, \cH))$ is $\cP$-measurable and $h\in L^2(\Omega \times [0,T], L^2(Z, \cH; \nu))$ is $\cP^-$-measurable.
    Suppose that $u \in L^2(\Omega, L^2([0,T], \cV)) \cap L^2(\Omega, D([0,T], \cH))$ is a strong solution on $[0,T]$ to
    \begin{align}
    \label{eq: spde linear estimate lemma}
        du(t) + \wt{A}_0(t) u(t) \,dt = f(t) \,dt &+ (\wt{B}_0(t)u(t) + g(t)) \,dW(t) \\
        &+ \int_Z \bigl(\wt{C}_0(t,z) u(t-) + h(t,z)\bigr) \,\Nt(dz, dt),
    \end{align}
    where $u_0 \coloneqq u(0) \colon \Omega \to \cH$ is strongly $\cF_0$-measurable.
    Moreover, let $0 \leq \tau_1 \leq \tau_2 \leq T$ be stopping times.
    Then
    \begin{align}
    \label{eq: sup estimate linear lemma}
        &\E \sup_{t \in [\tau_1,\tau_2]} \| u(t) \|_\cH^2 + \int_{\tau_1}^{\tau_2} \E \| u(s) \|_\cV^2 \,ds \\
        \leq{} &C\Bigl( \E \| u(\tau_1) \|_\cH^2 + \sum_{i=1}^m \E \| f_i \|_{L^{p_i}([\tau_1,\tau_2], \cV_{\theta_i})}^2 + \E \| g \|_{L^2([\tau_1,\tau_2], \cL_2(U, \cH))}^2 + \E \| h \|_{L^2([\tau_1,\tau_2], L^2(Z, \cH; \nu))}^2 \Bigr),
    \end{align}
    where the constant $C>0$ depends on $T$, $\kappa$, $\|\phi\|_{L^1([0,T])}$, $\|K_{B}\|_{L^{r_B}([\tau_1, \tau_2])}$ and $\|K_C\|_{L^{r_C}([\tau_1, \tau_2])}$.
\end{proposition}

Next, we use the a-priori estimate of Proposition~\ref{prop: lin apriori} and the method of continuity to obtain the existence and uniqueness of the linear problem \eqref{eq:spde lin}.
\begin{theorem}[linear $L^2$-theory]
\label{thm L2 linear}
    Suppose Assumption~\ref{assumption linear operators}.
    Let $f = \sum_{i=1}^m f_i$, where $f_i \in L^2(\Omega, L^{p_i}([0,T], \cV_{\theta_i}))$ is $\cP$-measurable and $(p_i,\theta_i)$ is admissible for every $i$, $g\in L^2(\Omega \times [0,T], \cL_2(U, \cH))$ is $\cP$-measurable and $h\in L^2(\Omega \times [0,T], L^2(Z, \cH; \nu))$ is $\cP^-$-measurable.
    Moreover, let $\lambda$ be a stopping time with values in $[0,T]$ and $u_\lambda \in L^2_{\cF_\lambda}(\Omega, \cH)$. Then the problem
    \begin{align}
    \label{eq: linear spde L2 theory}
        \begin{split}
            du+\wt{A}_0(t)u(t) \,dt &= f(t) \,dt + (\wt{B}_0(t)u(t)+g(t)) \,dW(t) \\
            &\quad\quad\quad\quad + \int_{Z} \big(\wt{C}_0(t,z)u(t-) + h(t,z)\big)\,\Nt(dz,dt), \\
            u(\lambda) &= u_\lambda,
        \end{split}
    \end{align}
    has a unique solution
    \begin{align}
        u \in L^2(\Omega, D([\lambda, T], \cH)) \cap L^2(\Omega, L^2([\lambda, T], \cV)),
    \end{align}
    satisfying
    \begin{align}
        &\E \sup_{t \in [\lambda,T]} \| u(t) \|_\cH^2 + \E \int_\lambda^T \| u(s) \|_\cV^2 \, ds \\
        \leq{} &C \E \Bigl( \| u_\lambda \|_\cH^2 + \sum_{i=1}^m \| f_i \|_{L^{p_i}([\lambda,T], \cV_{\theta_i})}^2 + \| g \|_{L^2([\lambda,T], \cL_2(U, \cH)}^2 + \| h \|_{L^2([\lambda,T], L^2(Z, \cH; \nu))}^2 \Bigr).
    \end{align}
\end{theorem}

\begin{proof}
    We subdivide the proof into several steps.

    \textbf{Step 1}: First we consider the existence and uniqueness for the following deterministic problem. Let $\overline{A}_0 \in \calL(\cV, \cV^*)$ be any self-adjoint positive operator which is invertible (see \cite[Proposition 8.1.10]{TayPDE2}). Pointwise in $\Omega$ consider
    \begin{align}
    \label{linear1}
        \begin{split}
            du+ \overline{A}_0 u(t) \,dt &= f(t) \,dt, \\
            u(0) &= 0.
        \end{split}
    \end{align}
    Then~\cite[Thm.~II.2.1]{Par75} provides a unique solution $v_1 \in L^2([0,T], \cV)\cap C([0,T],\cH)$ of \eqref{linear1} if $f \in L^1([0,T],H)$ or $f\in L^2([0,T], V^*)$. The case for general $f$ follows by interpolation~\cite[Thm.~2.2.6 \& Thm.~C.2.6]{hytonen2016analysis}. Moreover, progressive measurability of $v_1$ can be deduced by an approximation argument, or from Fubini's theorem and a variation of constants formula.

    Additionally, we get existence and uniqueness of a solution $v_2\in L^2(\Omega,L^2(0,T,\cV)\cap D([0,T],\cH))$ to the problem
    \begin{align}
    \label{linear2.5}
        \begin{split}
            du+\overline{A}_0u(t) \,dt &= g(t) \,dW(t) + \int_{Z} h(t,z) \,\Nt(dz,dt), \\
            u(0) &= 0,
        \end{split}
    \end{align}
    from~\cite[Thm.~1.2]{brzezniak2014strong}. Observe that the smallness conditions of that theorem disappear in the linear case as is discussed in~\cite[Rem.~1.4]{brzezniak2014strong}. Adding up the solutions to~\eqref{linear1} and~\eqref{linear2.5} we obtain existence of a solution to the full problem
    \begin{align}
    \label{linear2}
        \begin{split}
            du+\overline{A}_0 u(t)\,dt &= f(t)\,dt + g(t)\,dW(t) + \int_{Z} h(t,z)\,\Nt(dz,dt), \\
            u(0) &= 0.
        \end{split}
    \end{align}
    Moreover, from the a-priori estimate of Proposition~\ref{prop: lin apriori} we obtain uniqueness.

    \textbf{Step 2}: $\lambda = 0$ and $u_\lambda = 0$.
    For $r \in [0,1]$ put $\wt{A}^r_0 = (1-r)\overline{A}_0 +  r \wt A_0$, $\wt{B}^r_0 = r\wt B_0$, and $\wt{C}^r_0 = r\wt C_0$. Consider the family of problems
    \begin{align}
    \label{linear3}
        \begin{split}
            du+\wt{A}^r_0(t)u(t)\,dt &= f(t)\,dt + (\wt{B}^r_0(t)u(t)+g(t))\,dW(t) \\
            &\quad\quad\quad\quad+ \int_{Z} \wt{C}^r_0(t,z)u(t-) + h(t,z)\,\Nt(dz,dt), \\
            u(0) &= 0.
        \end{split}
    \end{align}
    The goal of this step is to show existence and uniqueness in the case $r = 1$. To this end, we appeal to the stochastic method of continuity. In the Gaussian case, when $\wt{C}^r_0 = 0$ and $h=0$, the stochastic method of continuity was presented in all detail in~\cite[Prop.~3.10]{portal_veraar}, and the argument extends verbatim to the case of \Ly noise. The hypotheses of the method of continuity are twofold. First, the existence and uniqueness in the case $r = 0$ have to hold. Indeed, this was the content of Step~1 above. Second, there has to be a constant $C > 0$ such that for any solution $v$ of~\eqref{linear3} there is the a priori estimate
    \begin{align}
    \label{eq: lin wellpos apriori}
        &\E \sup_{t \in [0,T]} \| v(t) \|_\cH^2 + \E \int_0^T \| v(s) \|_\cV^2 \, ds \\
        \leq{} &C \E \left( \sum_{i=1}^m \| f_i \|_{L^{p_i}([0,T], \cV_{\theta_i})}^2 + \| g \|_{L^2([0,T], \cL_2(U, \cH)}^2 +\| h \|_{L^2([0,T], L^2(Z, \cH
        ; \nu))}^2 \right).
    \end{align}
    Since the operators $\wt{A}^r_0$, $\wt{B}^r_0$ and $\wt{C}^r_0$ satisfy the coercivity condition~\eqref{eq: coercivity lin wellpos} with $\kappa$ and $\phi$ uniform in $r$, the a-priori estimate~\eqref{eq: lin wellpos apriori} follows from Proposition~\ref{prop: lin apriori}. Thus, the stochastic method of continuity indeed yields the existence and uniqueness of~\eqref{linear3} when $r = 1$.

    \textbf{Step 3}: non-trivial initial value $u_\lambda$ at a random initial time $\lambda$.
    When $\lambda$ is a non-trivial initial time but still with $u_\lambda = 0$, then existence and uniqueness of
    \begin{align}
    \label{linear4}
        \begin{split}
            du+\wt{A}_0(t)u(t)\,dt &= f(t) \, dt + (\wt{B}_0(t)u(t)+g(t))\,dW(t) \\
            &\quad\quad\quad\quad + \int_{Z} \wt{C}_0(t,z)u(t-) + h(t,z)\,\Nt(dz,dt), \\
            u(\lambda) &= 0,
        \end{split}
    \end{align}
    on $\llbracket\lambda, T \rrbracket$ follows from causality, which is a consequence of the linearity of the problem, compare also with~\cite[Prop.~3.10]{agresti2022nonlinear2}. Then the existence and uniqueness of
    \begin{align}
    \label{linear5}
        \begin{split}
            du+\wt{A}_0(t)u(t)\,dt &= f(t)\,dt + (\wt{B}_0(t)u(t)+g(t))\,dW(t) \\
            &\quad\quad\quad\quad+ \int_{Z} \wt{C}_0(t,z)u(t-) + h(t,z)\,\Nt(dz,dt), \\
            u(\lambda) &= u_\lambda,
        \end{split}
    \end{align}
    follows from the homogeneous case~\eqref{linear4} by virtue of~\cite[Prop.~3.12]{agresti2022nonlinear2}.

    \textbf{Step 4}: the a-priori estimate. The a-priori estimate follows from an application of Proposition~\ref{prop: lin apriori}.
\end{proof}

\subsection{No jumps of $u$ at predictable times}

Here, we show that the solution $u$ constructed in Theorem~\ref{thm L2 linear} does not experience jumps at predictable times owing to the total inaccessibility (see Definition~\ref{def: totally inaccessible}) of the jump times of a \Ly process. 
In the proof of Proposition~\ref{prop: jumps u}, we use some well-known properties regarding the jump times of the associated Poisson process.
For further background, we refer the reader to~\cite{jacod2013limit}.\newline
While Proposition~\ref{prop: jumps u} below is not needed in the present paper, we include it as information that may be useful in future work on the subject.

\begin{definition}
\label{def:pred time}
    A stopping time $\tau$ is called \emph{$\cF_t$-predictable}, or simply \emph{predictable}, if there exists a sequence of increasing stopping times $(\tau_n)_{n\geq 1}$, such that $\tau_n<\tau$ on $\{\tau>0\}$ and $\tau_n\to\tau$ almost surely as $n\to\infty$.
\end{definition}

\begin{definition}
\label{def: totally inaccessible}
    A stopping time $\tau$ is called totally inaccessible if $P(\tau = \mu<\infty)=0$ for all predictable stopping times $\mu$.
\end{definition}

\begin{proposition}
\label{prop: jumps u}
    Let the condition of Theorem~\ref{thm L2 linear} hold and let $u$ be the solution of the corresponding linear SPDE \eqref{eq: linear spde L2 theory} on $[0, T]$.
    Let $\mu\in [0,T]$ be a predictable stopping time. Then $\lim_{t\uparrow\mu} u(t) = u(\mu)$. In other words, $u$ does not experience jumps at $\mu$.
\end{proposition}

\begin{proof}
    Recall that on $\Omega\times\bR_+\times Z$ the Poisson random measure admits a representation of the form
    $$
    N(dz,dt) = \sum_{0<s} \delta_{p_s,s}(dz,dt)\1_{D}(t),
    $$
    where $D = \bigcup_n \llbracket \tau_n\rrbracket$ is
    exhausted by a sequence of totally inaccessible stopping times $(\tau_n)_{n=1}^\infty$, such that $\tau_n\to\infty$ as $n\to\infty$ and $(p_t)_{t\geq 0}$ is an optional jump process (see \cite[Ch.II]{jacod2013limit}). 
    Let $\Delta v$ denote the jump process associated to a càdlàg process $v$, given by $\Delta v(t) =  v(t)-v(t-)$. From the representation of the solution, \eqref{eq: strong solution}, we see almost surely
    \begin{equation}
    \label{eq: equality jumps}
    \Delta u(t) = \Delta \int_0^t\int_Z \wt{C}_0(s,z)u(s-) + h(s,z)\,\Nt(dz,ds),\quad t\in [0,T].
    \end{equation}
    Hence $N^u(dx,dt)$, the jump measure of $u$, has a representation of the form $N^u(dx,dt) = \sum_{0<s}\delta_{(\Delta u(s),s)}(dx,dt)\1_{D'}(t)$, where $D'=\bigcup_n \llbracket \tau_n'\rrbracket$
    with $(\tau_n')_{n\geq 1}$ being the jump times of $u$. By definition we have
    $$
    \Delta u = 0\quad\text{almost surely for}\quad t\in  [0,T] \setminus \bigcup_k \llbracket \tau_k'\rrbracket.
    $$
    Since for each $n\geq 1$, $\llbracket \tau_n' \rrbracket \subset \bigcup_k \llbracket \tau_k\rrbracket$, and thus
    $D'\subset D$, we get
    \begin{align}
        P(\{\tau_k' = \mu\}) &= P\Big( \{\tau_k' = \mu\}\cap \bigcup_n \{\tau_k' = \tau_n\}\Big)\\
        &\leq \sum_n P(\{\tau_k' = \mu\}\cap \{\tau_k' = \tau_n\})\\
        \label{eq: stopping times inaccessible}
        &\leq \sum_n P(\{\tau_n = \mu\}) = 0,
    \end{align}
    where in the last step we used that $\mu$ is predictable.
    Hence, almost surely $\Delta u = 0$ on $\llbracket \mu \rrbracket$, which proves the claim.
\end{proof}

\section{Local well-posedness}
\label{sec: local well-posedness}

In this section, we will prove the existence and uniqueness of a maximal solution to the following variant of the quasilinear problem \eqref{eq:spde}:
\begin{equation}\label{eq:spde2}
\tag{QLP}
    \begin{split}
        du(t)+A(t,u(t))\,dt &= B(t,u(t))\,dW(t) + \int_{Z} C(t,u(t-),z)\,\Nt(dz,dt)\\
        u(\sigma) &= u_{\sigma},
    \end{split}
\end{equation}
where $\sigma$ is a stopping time.
The nonlinearities consist of sums of a quasilinear part, a semilinear part, and an inhomogeneous part:
    \begin{align*}
        A(t,v) &= A_0(t,v)v- F(t,v) - f,
    \\  B(t,v) &= B_0(t,v)v + G(t,v) + g,
    \\  C(t,v) &= C_0(t,v)v + H(t,v) + h.
    \end{align*}
    The operator $A_0$ can be decomposed as $A_0(t,u)v = A_L(t,u)v + A_S(t,u)v$, where $A_L$ is the \emph{leading part} and $A_S$ is the \emph{singular part} of $A$.
    Here the following mapping and measurability properties are assumed
    \begin{align*}
    &A_L:\Omega\times \bR_+\times \cH\to \cL(\cV, \cV^*) \ \ \text{is $\cP\otimes \cB(\cH)$-measurable,}
    \\
    &A_S:\Omega\times\bR_+\times \cH\to \cL(\cV_{\beta_{A}}, \cV_{\alpha_{A}}) \ \ \text{is $\cP\otimes \cB(\cH)$-measurable,}
    \\ &B_0:\Omega\times\bR_+\times \cH\to \cL(\cV_{\beta_{B}}, \cL_2(U, \cH)) \ \ \text{is $\cP\otimes \cB(\cH)$-measurable},
    \\ &C_0:\Omega\times\bR_+\times \cH\to \cL(\cV_{\beta_{C}}, L^2(Z, \cH;\nu)) \ \ \text{is $\cP^-\otimes \cB(\cH)$-measurable,}
    \\ &F:\Omega\times\bR_+\times \cV_{\beta_F}\to \cV_{\alpha_F} \ \ \text{is $\cP\otimes \cB(\cV_{\beta_F})$-measurable,}
    \\ &G:\Omega\times\bR_+\times \cV_{\beta_G}\to \cL_2(U, \cH) \ \ \text{is $\cP\otimes \cB(\cV_{\beta_G})$-measurable,}
    \\ &H:\Omega\times\bR_+\times \cV_{\beta_H}\to L^2(Z,\cH;\nu) \ \ \text{is $\cP^{-}\otimes \cB(\cV_{\beta_H})$-measurable,}
    \\ &f:\Omega\times\bR_+\to \cV_{\alpha_f} \ \ \text{is $\cP$-measurable,}
    \\ &g:\Omega\times\bR_+\to \cL_2(U, \cH) \ \ \text{is $\cP$-measurable,}
    \\ &h:\Omega\times\bR_+\to L^2(Z,\cH;\nu) \ \ \text{is $\cP^{-}$-measurable,}
    \end{align*}
    where the parameters satisfy $\alpha_{A}, \alpha_{F}, \alpha_f \in [0, \nicefrac{1}{2}]$, $\beta_{A},\beta_{B},\beta_{C} \in [\nicefrac{1}{2},1]$, and $\beta_F, \beta_G, \beta_H \in (\nicefrac{1}{2}, 1]$. As before we let
    \begin{align}
    r_A = (1+\alpha_{A} - \beta_{A})^{-1}, & & r_B = (1- \beta_{B})^{-1}, & & r_C = (1- \beta_{C})^{-1}.
    \end{align}

We make the following further local Lipschitz, criticality and coercivity assumptions on the nonlinearities.
\begin{assumption}
\label{assumption operators new}
    Suppose that the following hold:
    \begin{enumerate}
        \item For each $n\geq 1$ and $T>0$ there is a constant $C_{n,T} \geq 0$ and a positive function $K_{A,n,T} \in L^{r_A}([0,T])$  such that, a.s. for all $t\in [0,T]$ and $u,v,w\in \cV$ satisfying $\|u\|_\cH,\|v\|_\cH\leq n$,
        \begin{align*}
            \|A_L(t,u)w\|_{\cV^*}&\leq C_{n,T} \|w\|_{\cV}\\
            \|A_L(t,u)w-A_L(t,v)w\|_{\cV^*}&\leq C_{n,T} \|u-v\|_\cH\|w\|_{\cV}\\
            \|A_S(t,u)w\|_{\alpha_A}&\leq K_{A,n,T}(t) \|w\|_{\beta_{A}}\\
            \|A_S(t,u)w-A_S(t,v)w\|_{\alpha_A}&\leq K_{A,n,T}(t)\|u-v\|_\cH\|w\|_{\beta_A}\\
            \|F(t,u)\|_{\alpha_F}&\leq C_{n,T}(1+\|u\|_{\beta_F}^{1+\rho_F})\\
            \|F(t,u)-F(t,v)\|_{\alpha_F}&\leq C_{n,T}(1+\|u\|_{\beta_F}^{\rho_F} + \|v\|_{\beta_F}^{\rho_F})\|u-v\|_{\beta_F},
        \end{align*}
        where  $\rho_F > 0$ and the following subcriticality condition holds
        \begin{equation}
        \label{eq: criticality F}
        (1 + \rho_F)(2\beta_{F} - 1) \leq 1 + 2 \alpha_{F}.
        \end{equation}
        \item For each $n\geq 1$ and $T>0$ there are a constant $C_{n,T} \geq 0$ and a positive function $K_{B,n,T} \in L^{r_B}([0,T])$  such that, a.s. for all $t\in [0,T]$ and $u,v,w\in \cV$ satisfying $\|u\|_\cH,\|v\|_\cH\leq n$,
        \begin{align*}
            \|B_0(t,u)w\|_{\calL_2(U,\cH)}&\leq K_{B,n,T}(t)\|w\|_{\beta_B}\\
            \|B_0(t,u)w-B_0(t,v)w\|_{\calL_2(U,\cH)}&\leq K_{B,n,T}(t)\|u-v\|_\cH\|w\|_{\beta_B}\\
             \|G(t,u)\|_{\calL_2(U,\cH)}&\leq C_{n,T} (1+\|u\|_{\beta_G}^{1+\rho_G})\\
            \|G(t,u)-G(t,v)\|_{\calL_2(U,\cH)}& \leq C_{n,T}(1+\|u\|_{\beta_G}^{\rho_G}+\|v\|_{\beta_G}^{\rho_G})\|u-v\|_{\beta_G},
        \end{align*}
        where  $\rho_G > 0$ and the following subcriticality condition holds
        \begin{equation}
        (1 + \rho_G)(2\beta_G - 1) \leq 1.
        \end{equation}
        \item For each $n\geq 1$ and $T>0$ there are a constant $C_{n,T}, C_{H,n,T}\geq 0$ and a positive function $K_{C,n,T} \in L^{r_C}([0,T])$  such that, a.s. for all $t\in [0,T]$ and $u,v,w\in \cV$ satisfying $\|u\|_\cH,\|v\|_\cH\leq n$,
        \begin{align*}
            \|C_0(t,u,\cdot)w\|_{L^2(Z,\cH;\nu)}&\leq K_{C,n,T}(t)\|w\|_{\beta_C}\\
            \|C_0(t,u,\cdot)w - C_0(t,v,\cdot)w\|_{L^2(Z,\cH;\nu)}&\leq K_{C,n,T}(t)\|u-v\|_\cH\|w\|_{\beta_C}\\
            \|H(t,u,\cdot)\|_{L^2(Z,\cH;\nu)}&\leq C_{n,T}(1+\|u\|_{\beta_H}^{1+\rho_H})\\
            \|H(t,u,\cdot) - H(t,v,\cdot)\|_{L^2(Z,\cH;\nu)}&\leq C_{n,T}(1+\|u\|_{\beta_H}^{\rho_H} + \|v\|_{\beta_H}^{\rho_H})\|u-v\|_{\beta_H},
        \end{align*}
        where  $\rho_H > 0$ and the following subcriticality condition holds
                \begin{equation}
        (1 + \rho_H)(2\beta_H - 1) \leq 1.
        \end{equation}
        \item For each $n\geq 1$ and $T>0$ there are a constant $\kappa_n>0$ and a positive function $\phi_n \in L^1([0,T])$ such that a.s.\ for all $v\in \cV$ and $u\in \cH$ satisfying $\|u\|_\cH\leq n$, and almost every $t\in [0,T]$,
        \begin{equation}
            \label{coercivity condition}
            \begin{aligned}
            \langle A_0(t,u)v, v \rangle -\tfrac{1}{2}\|B_0(t,u)v\|_{\calL_2(U,\cH)}^2 -& \tfrac{1}{2} \|C_0(t,u,\cdot)v\|_{L^2(Z,\cH;\nu)} \\ & \geq \kappa_n\|v\|_\cV^2 - \phi_n(t) \|v\|_\cH^2.
            \end{aligned}
        \end{equation}
    \end{enumerate}
\end{assumption}

\begin{remark}
	\label{rem: ass nonlinear operators}
    Some comments regarding Assumption~\ref{assumption operators new} are in order:
    \begin{enumerate}
        \item The conditions on $(A_0, B_0, C_0)$ have a lot of symmetry. The exceptions are that $A_S$ is allowed to take values in an interpolation space and that $A_L$ maps always between the endpoint spaces $V$ and $V^*$. Also the conditions on $(F, G, H)$ have a lot of symmetry with the same type of exception for the range space of $F$.
        \item One can allow a sum with different combinations $\rho_{F,j}$ and $\beta_{F,j}$ on the right-hand side, but for simplicity we have not done this here. The same applies to the nonlinearities $G$ and $H$.
        \item Usually, $A_S = 0$, $\beta_B = \beta_C = 1$. One could actually consider a sum of $A_{S,i}$ with different numbers $\alpha_{A,i}$ and $\beta_{A,i}$. The same applies to $B_0$ and $C_0$, as well as $F$, $G$ and $H$. \label{it:sum assump}
        \item If $\beta < 1$, we can also allow $\rho = 0$. Indeed, then $2\beta - 1 < 1$, so for $\rho$ small enough one still has $(1+\rho)(2\beta - 1) < 1$. Here, $(\beta, \rho)$ can be any of the pairs $(\beta_F, \rho_F)$, $(\beta_G, \rho_G)$ or $(\beta_H, \rho_H)$.
        \item The criticality condition~\eqref{eq: criticality F} for $F$ recovers the classical notion of criticality from~\cite{agresti2022critical} when $\alpha_F = 0$. However, taking $\alpha_F$ different to zero is one of the core observations of this paper and allows to treat a broader class of critical equations. See Section~\ref{sec:applications} for examples.
        \item The coercivity condition \eqref{coercivity condition} is an assumption on the linearized part of the equation and is usually easy to verify. A similar condition for $(A,B,C)$ will appear in Section~\ref{sec:global}, where we show that it implies global well-posedness for \eqref{eq:spde2}.
    \end{enumerate}
\end{remark}

The following two Lemmas are useful for proving local well-posedness of the problem \eqref{eq:spde2}. We show that, using the correct truncations and the $\cM$-norm defined below, the terms in \eqref{eq:spde2} locally exhibit Lipschitz behaviour, which allows us to employ a fixed point argument later on.

To abbreviate the notation in the text below we introduce the short-hand notation $\mathcal{M}(a,b) = D([a,b],\cH)\cap L^2([a,b],\cV)$ with the norm
\begin{align}\label{eq:cMspaces}
\|u\|_{\cM(a,b)}:= \sup_{t\in [a,b]}\|u(t)\|_\cH + \Big(\int_{a}^{b} \|u(t)\|_\cV^2\,dt\Big)^{1/2}.
\end{align}
We will also write $\cM(b)$ for $\cM(0,b)$.
The estimate of Lemma~\ref{lem:admissible} implies that for every $\beta\in (1/2, 1]$ and $u\in \cM(a,b)$ one has
    \begin{align}\label{eq:bH embedding}
        \| u \|_{L^{2/(2\beta-1)}([a,b],\cV_\beta)} \leq \| u \|_{L^{2}([a,b],\cV)}^{2\beta-1} \|u\|_{L^\infty([a,b];\cH)}^{2-2\beta} \leq \| u \|_{\cM(a,b)}.
    \end{align}
Let
\[\cX(a,b) = \bigcap L^{2/(2\beta-1)}([a,b],\cV_\beta),\]
where the intersection is taken over all $\beta\in \{1, \beta_A, \beta_B, \beta_C, \beta_F, \beta_G, \beta_H\}$.
As for $\cM$, we use the short-hand notation $\cX(b) \coloneqq \cX(0,b)$.
Let a smooth function $\xi \colon \bR \to [0,1]$ be such that $\supp \xi\subset [-2,2]$ and $\xi=1$ on $[-1,1]$, and define for $\lambda>0$ the dilated function $\xi_\lambda:=\xi(\tfrac{\cdot}{\lambda})$.
Given $0\leq a<b<\infty$, for $(t,x,u)\in [a,b]\times \cH\times \cM(a,b))$ define the truncation
\begin{align}
    \Theta_{\lambda} (t,x,v) &\coloneqq \xi_\lambda \Bigl(\|v\|_{\mathcal{X}(a,t)} + \sup_{s\in [a,t]}\|v(s-)-x\|_\cH\Bigr),
\end{align}
Using the truncation $\Theta_{\lambda}$, we define truncations of the non-linearities $F$, $G$, $H$ in the following lemma and obtain local growth and Lipschitz estimates for them.
\begin{lemma}
\label{lemma truncation 1}
    Fix $T>0$.
    Let $F,G$ and $H$ be as in Assumption~\ref{assumption operators new}, let $\lambda\in (0,1)$ and $\ell>0$. Put $p_F \coloneqq 2/(2\alpha_F + 1) \in [1, 2]$, where $\alpha_F$ is the parameter from Assumption~\ref{assumption operators new}. Then a.s.\ for $0 \leq a<b \leq T$ such that $|b-a|\leq \ell$ the functions
    \begin{align*}
    F_\lambda &\colon \cH\times \cM(a,b)\rightarrow L^{p_F}([a,b],\cV_{\alpha_F}),\\
    G_\lambda &\colon \cH\times \cM(a,b)\rightarrow L^2([a,b],\cL_2(U,\cH)),\\
    H_\lambda &\colon \cH\times \cM(a,b)\rightarrow
    L^2([a,b],L^2(Z,\cH;\nu),
    \end{align*}
    defined pathwise by
    \begin{align*}
        \bigl[F_\lambda (x,v) \bigr](t) &\coloneqq \Theta_\lambda (t,x,v)(F(t,v(t))-F(t,0)),\\
        \bigl[G_\lambda (x,v)\bigr](t) &\coloneqq \Theta_\lambda (t,x,v)(G(t,v(t))-G(t,0)),\\
        \bigl[H_\lambda (x,v)\bigr](t,z) &\coloneqq \Theta_\lambda (t,x,v)(H(t,v(t-),z)-H(t,0,z)),
    \end{align*}
    satisfy for $n\geq 1$ and $\|x\|_\cH\leq n$, $u,v\in \cM(a,b)$
    \begin{align*}
        \|F_\lambda (x,u)\|_{L^{p_F}([a,b],\cV_\alpha)} + \|G_\lambda (x,u)\|_{L^2([a,b],\cL_2(U,\cH))} + \|H_\lambda (x,u,\cdot)\|_{L^2([a,b],L^2(Z,\cH;\nu))} \leq C_{\lambda,\ell},
     \end{align*}
       \begin{align*}
        \|F_\lambda (x,u)-F_\lambda (x,v)&\|_{L^{p_F}([a,b],\cV_\alpha)} + \|G_\lambda (x,u)-G_\lambda(x,v)\|_{L^2([a,b],\cL_2(U,\cH))} \\ & \qquad + \|H_\lambda (x,u,\cdot)-H_\lambda(x,v,\cdot)\|_{L^2([a,b],L^2(Z,\cH;\nu))} \leq C_{\lambda,\ell} \|u-v\|_{\cM(a,b)},
    \end{align*}
    where $C_{\lambda,\ell}$ depends on $n,T,\lambda,\ell$ and the parameters quantified in Assumption~\ref{assumption operators new} (applied with $n$ and $T$). Moreover, for every $\varepsilon>0$ there exist $\ell^*$ and $\lambda^*$ such that $C_{\lambda,\ell}<\varepsilon$ for all $\ell\leq \ell^*$ and $\lambda \leq \lambda^*$.
\end{lemma}
\begin{proof}
    We present the case of $F$ in detail, the terms concerning $G$ and $H$ can be calculated in the same way.
    To ease notation, we just write $(\alpha,\beta,\rho,p)$ instead of $(\alpha_F,\beta_F, \rho_F,p_F)$ for the parameters from Assumption~\ref{assumption operators new}.
    By translation, it suffices to consider $a = 0$ and $b = \ell$ and we assume without loss of generality $T = \ell$.
    Define
    \begin{equation}
    \label{eq: def tau_u}
    \tau_u:=\inf\{ t\in [0,T]:
    \|u\|_{\cX(t)} + \sup_{s\leq t}\|u(s)-x\|_\cH\geq 2\lambda\}\wedge T.
    \end{equation}
    Then, using that we have $\|u\|_\cH\leq \| x \|_\cH +2\lambda \leq n + 2$ on $[0,\tau_u)$, H\"older's inequality in conjunction with Assumption~\ref{assumption operators new} yield
    \begin{align}
    \label{eq: truncation F bound}
        \| F(\cdot, u) - F(\cdot, 0) &\|_{L^p([0,\tau_u], \cV_\alpha)} \\
        &\leq C \bigl\| (1 + \| u \|_\beta^\rho) \| u \|_\beta \bigr\|_{L^p([0,\tau_u])} \\
        &\leq C \bigl(T^{\nicefrac{\rho}{(p(\rho + 1))}} + \| u \|_{L^{p(\rho+1)}([0,\tau_u], \cV_\beta)}^\rho \bigr) \| u \|_{L^{p(\rho+1)}([0,\tau_u], \cV_\beta)}
    \end{align}
    almost surely.
    Here, the constant $C$ depends on the parameters quantified in Assumption~\ref{assumption operators new} and $n$.
    Using the definition of $p$ and the criticality condition~\eqref{eq: criticality F} we obtain
    \begin{align}
        p(\rho+1) = \frac{2(\rho+1)}{1+2\alpha} \leq \frac{2}{2\beta-1}.
    \end{align}
    Therefore, estimating the $L^{p(\rho+1)}([0,\tau_u], \cV_\beta)$ norms by $C_T$-times the $L^{2/(2\beta-1)}([0,\tau_u], \cV_\beta)$ norm, we can control the right-hand side of~\eqref{eq: truncation F bound} by using
    \[\|u\|_{L^{2/(2\beta-1)}([0,\tau_u], \cV_\beta)}\leq 2\lambda.\]
    Indeed, we get
    \begin{align}
        \| F(\cdot, u) - F(\cdot, 0) \|_{L^p([0,\tau_u], \cV_\alpha)}
        \leq C (T^\frac{\rho}{p(\rho + 1)} + (2\lambda)^{\rho}) \lambda  \label{eq: F lambda u bound} \eqcolon C_{\lambda,T}.
    \end{align}
    Observe that if $\lambda,T\rightarrow 0$ then also  $C_{\lambda,T}\rightarrow 0$.
    Also, note that if $t > \tau_u$ then $\Theta_{\lambda}(t,x,u)=0$. Therefore,
    \begin{align}
         \| F_\lambda(x, u) \|_{L^p([0,T], \cV_\alpha)} \leq \| F(\cdot, u) - F(\cdot, 0) \|_{L^p([0,\tau_u], \cV_\alpha)} \leq C_{\lambda,T}.
    \end{align}
    Next, we compute
    $$
    \| F_\lambda(x,u)- F_\lambda(x,v)\|_{L^p([0,T],\cV_\alpha)}\leq R_1 + R_2,
    $$
    where we define
    \begin{align*}
        R_1&\coloneqq \left(\int_0^T\|\Theta_\lambda (s,x,u)\left(  F(s,u(s))- F(s,v(s))\right)\|_{\cV_\alpha}^p\,ds \right)^\frac{1}{p},\\
        R_2&\coloneqq \left(\int_0^T \|\left(\Theta_\lambda (s,x,u)-\Theta_\lambda(s,x,v\right) (F(s,v(s))-F(s,0)) \|_{\cV_\alpha}^p\,ds \right)^\frac{1}{p}.
    \end{align*}
    We also define $\tau_v$ by the same expression as in~\eqref{eq: def tau_u} but with $u$ replaced by $v$. Since the argument is pathwise, we may assume without loss of generality that $\tau_u\leq\tau_v$.
    Since $F$ satisfies Assumption~\ref{assumption operators new} we can estimate the terms similarly to above (where we had $v = 0$) to obtain
    \begin{align*}
    R_1&\leq C \left(\int_0^{\tau_u} (1+\|u(s)\|_\beta^\rho + \|v(s)\|_\beta^\rho)^p\|u(s)-v(s)\|_\beta^p\,ds \right)^\frac{1}{p}\\
    &\leq C\left(\int_0^{\tau_u}(1+\|u(s)\|_\beta^{p(1+\rho)} + \|v(s)\|_\beta^{p(1+\rho)})\,ds  \right)^\frac{\rho}{p(1+\rho)} \left(\int_0^{\tau_u}\|u(s)-v(s)\|_\beta^{p(1+\rho)}\,ds\right)^\frac{1}{p(1+\rho)}\\
    &\leq C\big(\tau_u^\frac{1}{p(1+\rho)} + 2(2\lambda)^{\rho} \big)\|u-v\|_{\cM(T)}\\
    &\eqcolon C_{\lambda,T}\|u-v\|_{\cM(T)}
    \end{align*}
    for a constant $C=C(n,T,\beta,\rho)$, and the constant $C_{\lambda,T}$ vanishes as $\lambda,T\rightarrow 0$.
    To estimate $R_2$ note first that for $C_\xi:=\sup_{r\in\bR}|\tfrac{d}{dr} \xi(r)|$ we have, by the reverse triangle inequality and \eqref{eq:bH embedding},
    \begin{align}
    \label{eq: Theta lipschitz}
    |\Theta_\lambda(s,x,u)-\Theta_\lambda(s,x,v)|&\leq \frac{C_\xi}{\lambda}\Bigl(\|u-v\|_{\cX(s)} + \sup_{r\leq s}\|u(r)-v(r)\|_{\cH}\Bigr)\\
    &\leq \frac{C}{\lambda} \|u-v\|_{\cM(s)}
    \end{align}
    for a constant $C=C(C_\xi,\beta)$.
    Hence we get, using \eqref{eq: F lambda u bound} with $u$ replaced by $v$,
    \begin{align}
        R_2
        &\leq \left( \int_0^{\tau_v} |\Theta_\lambda(s,x,u)-\Theta_\lambda(s,x,v)|^p \| F(s,v(s)) - F(s,0)\|_{\cV_\alpha}^p\,ds \right)^\frac{1}{p}\\
        &\leq \frac{C}{\lambda}\|u-v\|_{\cM( \tau_v)} \left(\int_0^{\tau_v} \| F(s,v(s)) - F(s,0)\|_{\cV_\alpha}^p\,ds \right)^\frac{1}{p}\\
        &\leq \frac{C}{\lambda} (T^{\nicefrac{1}{(p(1+\rho))}} + (2\lambda)^\rho)\lambda \|u-v\|_{\cM(T)} \\
        &\eqqcolon C_{\lambda,T} \|u-v\|_{\cM(T)},
    \end{align}
    where $C=C(C_\xi,\beta,n,T,\rho)$, and $C_{\lambda,T}\rightarrow 0$ when $T,\lambda\rightarrow 0$.
\end{proof}
The next lemma is again a truncation lemma, but this time for the quasilinear operators $A$, $B$, $B$.
\begin{lemma}
\label{lemma truncation 2}
    Fix $T>0$.
    Let $A_L, A_S$, $B_0$ and $C_0$ be as in Assumption~\ref{assumption operators new}, and let $\lambda\in (0,1)$ and $\ell > 0$. Put $p_A \coloneqq 2/(2\alpha_A + 1) \in [1, 2]$, where $\alpha_A$ is the parameter from Assumption~\ref{assumption operators new}. Then a.s.\ for $0 \leq a<b \leq T$ such that $|b-a|\leq \ell$ the functions
    \begin{align*}
        F_{A_L,\lambda} &\colon \Omega \times \cH\times \cM(a,b)\rightarrow L^2([a,b], \cV^*),
\\       F_{A_S,\lambda} &\colon \Omega \times \cH\times \cM(a,b)\rightarrow L^{p_A}([a,b], \cV_{\alpha_A}), \\
        G_{B_0,\lambda} &\colon \Omega \times \cH\times \cM(a,b)\rightarrow L^2([a,b], \cL_2(U,\cH)), \\
        H_{C_0,\lambda} &\colon \Omega \times \cH\times \cM(a,b)\rightarrow L^2([a,b], L^2(Z,\cH; \nu)),
   \end{align*}
    defined pathwise by
    \begin{align*}
        \bigl[F_{A_L,\lambda}(x,u)\bigr](t) &\coloneqq \Theta_\lambda(t,x,u)(A_L(t,x)-A_L(t,u(t)))u(t), \\
        \bigl[F_{A_S,\lambda}(x,u)\bigr](t) &\coloneqq \Theta_\lambda(t,x,u)(A_S(t,x)- A_S(t,u(t)))u(t), \\
        \bigl[G_{B_0,\lambda}(x,u)\bigr](t) &\coloneqq \Theta_\lambda(t,x,u)(B_0(t,x)-B_0(t,u(t)))u(t), \\
        \bigl[H_{C_0,\lambda}(x,u)\bigr](t,z) &\coloneqq \Theta_\lambda(t,x,u)(C_0(t,x,z)-C_0(t,u(t-),z))u(t-).
    \end{align*}
    satisfy for any $n\geq 1$ and all $\|x\|_\cH\leq n$, $u,v\in \cM(a,b)$
    \begin{align*}
        & \|F_{\lambda,A_L}(x,u)\|_{L^2([a,b],\cV^*)}+
        \|F_{\lambda,A_S}(x,u)\|_{L^{p_A}([a,b],\cV_{\alpha_A})} \\ & + \|G_{B_0,\lambda}(x,u)\|_{L^2([a,b],\cL^2(U,\cH))} + \|H_{C_0,\lambda}(x,u,\cdot)\|_{L^2([a,b],L^2(Z,\cH;\nu))} \leq C_{\lambda,\ell},
    \end{align*}
    \begin{align*}
        \|F_{\lambda,A_L}(x,u) &- F_{\lambda,A_L}(x,v)\|_{L^2([a,b],\cV^*)}  \\
        &+ \|F_{\lambda,A_S}(x,u) - F_{\lambda,A_S}(x,v)\|_{L^{p_A}([a,b],\cV_{\alpha_A})}
        \\ & + \|G_{B_0,\lambda}(x,u)-G_{B_0,\lambda}(x,v)\|_{L^2([a,b],\cL^2(U,\cH))}
         \\
         &+ \|H_{C_0,\lambda}(x,u,\cdot)- H_{C_0,\lambda}(x,v,\cdot)\|_{L^2([a,b],L^2(Z,\cH;\nu))} 
         \leq C_{\lambda,\ell}\|u-v\|_{\cM(a,b)},
    \end{align*}
    where $C_{\lambda,\ell}$ depends on $n,T,\lambda$, $\ell$ and the parameters quantified in Assumption~\ref{assumption operators new} (with $T$ and $n$). Moreover, for every $\varepsilon>0$ there exist $\ell^*$ and $\lambda^*$ such that $C_{\lambda,\ell}<\varepsilon$ for all $\ell\leq \ell^*$ and $\lambda \leq \lambda^*$.
\end{lemma}

\begin{proof}
    We present the case of $H_{C_0,\lambda}$ in detail and comment on the necessary changes for $F_{A_L,\lambda}$, $F_{A_S,\lambda}$ and $G_{B_0,\lambda}$ afterwards.
    By translation, it suffices to consider $a = 0$ and $b = \ell$, and again we assume without loss of generality $\ell = T$.
    Recall $\beta_C$, $K_{C,n,T}$ and $r_C$ from Assumption~\ref{assumption operators new}. To simplify notation we write $(\beta,K_{n,T}, r)$ instead of $(\beta_C, K_{C,n,T}, r_C)$.
    Similar to Lemma~\ref{lemma truncation 1} we introduce the stopping times $\tau_u$ and $\tau_v$ using the formula~\eqref{eq: def tau_u}. By symmetry, we can assume $\tau_u \geq \tau_v$.

    Let us start with the bound for $H_{C_0,\lambda}$. Using that $\Theta_\lambda(t,x,u) = 0$ when $t > \tau_u$ and $\| u \|_\cH \leq n + 2$ on $\llbracket 0, \tau_u \rrparenthesis$, we compute with Assumption~\ref{assumption operators new} that
    \begin{align}
        \|H_{C_0,\lambda}(x,u,\cdot)\|_{L^2([0,T],L^2(Z,\cH;\nu))}&\leq \left( \int_0^{\tau_u} \| (C_0(s,x,\cdot) - C_0(s,u(s),\cdot))u(s)\|_{L^2(Z,\cH;\nu)}^2 ds \right)^\frac{1}{2} \\
        &\leq \left( \int_0^{\tau_u} (K_{n+2,T}(s) \| x - u(s) \|_\cH \| u(s) \|_\beta)^2 ds \right)^\frac{1}{2}.
    \end{align}
    Since $\tfrac{1}{2} = \tfrac{2\beta-1}{2} + \tfrac{1}{r}$, Hölder's inequality gives
    \begin{align}
    \label{eq:truncation2_hoelder}
        &\left( \int_0^{\tau_u} (K_{n+2,T}(s) \| x - u(s) \|_\cH \| u(s) \|_\beta)^2 ds \right)^\frac{1}{2} \\
        \leq{} &\sup_{s <\tau_u} \| u(s) - x \|_\cH \| K_{n+2,T} \|_{L^r([0,T])} \| u \|_{L^\frac{2}{2\beta-1}([0,\tau_u), \cV_\beta)}.
    \end{align}
    If $\beta > 1/2$ deduce using the definition of $\tau_u$ that
    \begin{align}
        \|H_{C_0,\lambda}(x,u,\cdot)\|_{L^2([0,T],L^2(Z,\cH;\nu))} \leq C \| K_{n+2,T} \|_{L^r([0,T])} \lambda^2 \eqqcolon C_{\lambda,T},
    \end{align}
    where $C_{\lambda,T} \to 0$ when $\lambda,T \to 0$. Otherwise, if $\beta = 1/2$, estimate similarly that
    \begin{align}
        \|H_{C_0,\lambda}(x,u,\cdot)\|_{L^2([0,T],L^2(Z,\cH;\nu))} \leq C \| K_{n+2,T} \|_{L^r([0,T])} \lambda (n+2)
        \eqqcolon &C_{\lambda,T},
    \end{align}
    where still $C_{\lambda,T} \to 0$ when $\lambda,T \to 0$.

    For the Lipschitz bound for $H_{C_0,\lambda}$, we decompose
    $$
    \|H_{C_0,\lambda} (x,u,\cdot)-H_{C_0,\lambda}(x,v,\cdot)\|_{L^2([0,T],L^2(Z,\cH;\nu))}\leq I_1 + I_2 + I_3,
    $$
    where
    \begin{align*}
        I_1&\coloneqq \left(\int_0^T\|\big(\Theta_\lambda(s,x,u) - \Theta_\lambda(s,x,v)\big)(C_0(s,x,\cdot)-C_0(s,u(s),\cdot))u(s)\|_{L^2(Z,\cH;\nu)}^2 ds \right)^\frac{1}{2}, \\
      I_2&\coloneqq \left(\int_0^T\|\Theta_\lambda(s,x,v)\big(C_0(s,u(s),\cdot)-C_0(s,v(s),\cdot)\big) u(s)\|_{L^2(Z,\cH;\nu)}^2 ds \right)^\frac{1}{2}, \\
      I_3&\coloneqq \left(\int_0^T\|\Theta_\lambda(s,x,v) (C_0(s,x,\cdot)-C_0(s,v(s),\cdot))(u(s)-v(s))\|_{L^2(Z,\cH;\nu)}^2 ds \right)^\frac{1}{2}. \\
    \end{align*}
    Recall from~\eqref{eq: Theta lipschitz} the Lipschitz bound
    $$
    |\Theta_\lambda (s,x,u)-\Theta_\lambda (s,x,v)|\leq \frac{C_\xi}{\lambda}\|u-v\|_{\cM(s)}.
    $$
    When $s > \tau_u$ and consequently also $s > \tau_v$, the integrand of $I_1$ vanishes. Otherwise, we have $\| u \|_\cH \leq n + 2$ on $\llbracket 0, \tau_u \rrparenthesis$. Using Assumption~\ref{assumption operators new} and Hölder's inequality estimate
    \begin{align*}
        I_1&\leq \frac{C_\xi}{\lambda} \left(\int_0^{\tau_u} \|u-v\|_{\cM(s)}^2 \|(C_0(s,x,\cdot)-C_0(s,u(s),\cdot))u(s)\|_{L^2(Z,\cH;\nu)}^2 ds \right)^\frac{1}{2}\\
        &\leq \frac{C}{\lambda} \|u-v\|_{\cM(T)} \left(\int_0^{\tau_u} \bigl(K_{n+2,T}(s) \| u(s) - x \|_\cH \| u(s) \|_\beta \bigr)^2 \,ds \right)^\frac{1}{2} \\
        &\leq{} \frac{C}{\lambda} \|u-v\|_{\cM(T)} \sup_{s\in [0,\tau_u)}\| u(s) -x\|_{\cH} \| K_{n+2,T} \|_{L^r([0,T])} \| u \|_{L^\frac{2}{2\beta-1}([0,\tau_u), \cV_\beta)}.
    \end{align*}
    If $\beta > 1/2$, the definition of $\tau_u$ as above yields
    \begin{align}
        I_1 \leq C \lambda \| K_{n+2,T} \|_{L^r([0,T])} \|u-v\|_{\cM(T)}
        \eqqcolon{} C_{\lambda,T} \|u-v\|_{\cM(T)},
    \end{align}
    where $C_{\lambda,T}$ has the desired property.
    If $\beta = 1/2$ then $r=2$, so calculate
    \begin{align}
        I_1 \leq C \| K_{n+2,T} \|_{L^2([0,T])} (n+2) \|u-v\|_{\cM(T)} \coloneqq C_{\lambda,T} \|u-v\|_{\cM(T)}.
    \end{align}
    Observe that $C_{\lambda,T} \to 0$ when $T \to 0$ as this is the case for $\| K_{n+2,T} \|_{L^2([0,T])}$, independent of $\lambda$.
    This completes the treatment of $I_1$.

    To derive a bound for $I_2$ we can argue as for the bound for $H_{C_0,\lambda}(x,u,\cdot)$ with two remarks. First, we can restrict integration to $s \leq \tau_v$,
    but since $\tau_v \leq \tau_u$ we obtain $\| u \|_\cH \vee \| v \|_\cH \leq n + 2$ on $\llbracket 0, \tau_v \rrparenthesis$. Second, instead of $\sup_{s <\tau_u} \| u(s) - x \|_\cH$ we obtain $\sup_{s <\tau_v} \| u(s) - v(s) \|_\cH \leq \| u - v \|_{\cM(T)}$.

    Finally, for $I_3$ calculate similarly
    \begin{align}
        I_3 &\leq \left( \int_0^{\tau_v} \|(C_0(s,x,\cdot)-C_0(s,v(s),\cdot))(u(s)-v(s))\|_{L^2(Z,\cH;\nu)}^2ds \right)^\frac{1}{2} \\
        &\leq \left( \int_0^{\tau_v} (K_{n+2,T}(s) \| v(s) - x \|_\cH \| u(s) - v(s) \|_\beta)^2 \, ds \right)^\frac{1}{2} \\
        &\leq \sup_{s <\tau_v} \| v(s) - x \|_\cH \| K_{n+2,T} \|_{L^r([0,T])} \| u - v \|_{L^\frac{2}{2\beta-1}([0,T], \cV_\beta)} \\
        &\leq C \lambda \| K_{n+2,T} \|_{L^r([0,T])} \| u - v \|_{\cM(T)} \\
        &\eqqcolon C_{\lambda,T} \| u - v \|_{\cM(T)},
    \end{align}
    where we used \eqref{eq:bH embedding} in the penultimate step.

    The proof for $F_{A_L}$, $F_{A_S}$ and $G_{B_0}$ follows by analogous arguments. In the case of $F_{A_S}$ we have to define $r=r_A$, so that $\nicefrac{1}{p_A} = \nicefrac{(2\beta_A - 1)}{2} + \nicefrac{1}{r}$ lets us apply Hölder's inequality in the same way as in~\eqref{eq:truncation2_hoelder}.
\end{proof}

With the truncation lemmas at hand, we present a local well-posedness result in the case of integrable inhomogeneities. It is based on a fixed-point argument using Banach's fixed point theorem. To define the fixed-point map, we are going to apply the linear theory from Section~\ref{sec: linear} to a suitable auxiliary problem which incorporates the truncated non-linearities.
\begin{theorem}[Local well-posedness]
\label{thm wellposedness local}
    Let $T>0$ and fix a stopping time $\sigma \in [0,T]$.
    Suppose Assumption~\ref{assumption operators new}. Let
    \begin{align}
        f_i\in L^2(\Omega,L^{p_{i}}([0,T],\cV_{\theta_{i}})),\;\; g\in L^2(\Omega\times [0,T],\cL_2(U,\cH)),\;\; h\in L^2(\Omega\times [0,T],L^2(Z,\cH;\nu)),
    \end{align}
    be $\cP$-measurable, $\cP$-measurable, and $\cP^-$-measurable respectively, where the pairs $(p_{i},\theta_{i})$ are admissible in the sense of Definition~\ref{def:admissible}. Put $f = \sum_{i=1}^{m_f} f_i$.
    Let $u_\sigma$ be a bounded $\cF_\sigma$-measurable random variable taking values in $\cH$. The following holds:
    \begin{enumerate}
        \item (Existence and uniqueness) There exists a unique maximal solution $(u, \tau)$ on $\llbracket \sigma, T \rrbracket$ to~\eqref{eq:spde2} such that $\tau > \sigma$ almost surely on $\{\sigma < T\}$ and $\tau$ is a stopping time.
        \item (Continuous dependence on initial data) There are constants $C,\eta > 0$ such that for another bounded and $\cF_\sigma$-measurable initial value $v_\sigma$ with maximal solution $(v,\wt \tau)$ on $\llbracket \sigma, T \rrbracket$ and almost surely satisfying the closeness condition $\| u_\sigma - v_\sigma \|_\cH < \eta$, there exists a stopping time $\nu$ with $\nu\in (\sigma, \tau \wedge \wt\tau]$ on $\{\sigma < T\}$ such that
        \begin{align}
            \left( \E \| u - v \|_{\cM(\sigma,\nu)}^2 \right)^\frac{1}{2} \leq C \left( \E \| u_\sigma - v_\sigma \|_\cH^2 \right)^\frac{1}{2}.
        \end{align}
        \item (Localization property) Given another bounded and $\cF_\sigma$-measurable initial value $v_\sigma$ with corresponding maximal solution $(v,\wt \tau)$ on $\llbracket \sigma, T \rrbracket$ one has with $\Gamma \coloneqq \{u_\sigma = v_\sigma\}$ the identities
        \begin{align}
            \tau\1_{\Gamma} = \wt \tau \1_\Gamma \quad \mathrm{and} \quad u\1_\Gamma = v \1_\Gamma.
        \end{align}
    \end{enumerate}
\end{theorem}
\begin{proof}
    Fix $u_0$ as in the statement.

    \textbf{Step 1}: the existence of a (local) solution.
    For technical reasons in the proof, we start with a bounded initial value $w_\sigma$ potentially different to $u_\sigma$ that is $\cH$-valued and $\cF_\sigma$-measurable.
    For $T^* > 0$ define the stopping time $\mu_{T^*} \coloneqq (\sigma + T^*) \wedge T$. If no confusion is to be expected, we simply write $\mu$ instead of $\mu_{T^*}$.
    Let $v \in L^2(\Omega, \cM(\sigma,\mu))$, where $T^* > 0$ will be chosen small enough later on in the proof.
    Now, for $\lambda > 0$ let the maps $F_\lambda,G_\lambda, H_\lambda$ and $F_{A_L,\lambda}, F_{A_S,\lambda}, G_{B_0,\lambda}, H_{C_0,\lambda}$ be defined as in Lemmas~\ref{lemma truncation 1} and \ref{lemma truncation 2}. Recall $A_0 = A_L + A_S$ and let
    \begin{equation*}
        \tilde F_\lambda := (F_\lambda + F_{A_{L},\lambda}+F_{A_S,\lambda})\1_{\llbracket \sigma, \mu \rrbracket},\quad \tilde G_\lambda := (G_\lambda + G_{B_0,\lambda})\1_{\llbracket \sigma, \mu \rrbracket},\quad \tilde H_\lambda:=(H_\lambda + H_{C_0,\lambda})\1_{\llbracket \sigma, \mu \rrbracket},
    \end{equation*}
    and consider the linear SPDE
    \begin{equation}
    \label{eq:aux SPDE 1}
        \begin{split}
            du+A_0(t,u_{\sigma})u \, dt &= (\tilde F_\lambda (u_{\sigma},v)(t) + \tilde f(t)) \, dt + (B_0(t,u_{\sigma})u+ \tilde G_\lambda(u_{\sigma},v)(t) + \tilde g(t)) \, dW(t) \\
            &\quad\quad\quad\quad + \int_Z\big(C_0(t,u_{\sigma},z)u(\cdot\, -) + \tilde H_\lambda(u_{\sigma},v)(t,z) + \tilde h(t,z)\big)\,\Nt(dz,dt)\\
            u(\sigma) &= w_\sigma
        \end{split}
    \end{equation}
    on $\llbracket \sigma,T \rrbracket$, where
    \begin{align*}
        \tilde f(t):= f(t) + F(t,0),\quad \tilde g(t):=g(t) + G(t,0),\quad \tilde h(t,z):=h(t,z) + H(t,0,z).
    \end{align*}
    We follow a similar procedure as in~\cite{agresti2022nonlinear}. First, note that by Assumption~\ref{assumption operators new} in conjunction with boundedness of $u_\sigma$ the operators $(A_0(\cdot,u_\sigma),B_0(\cdot,u_\sigma),C_0(\cdot,u_\sigma,\cdot))$ satisfy the conditions of Theorem~\ref{thm L2 linear}. Hence, we may define a solution map $\cR$ associated with the operators $(A_0(\cdot,u_\sigma),B_0(\cdot,u_\sigma),C_0(\cdot,u_\sigma,\cdot))$ for the initial value $w_\sigma$. Also, the right-hand sides $\tilde F_\lambda(v) + \tilde f$, $\tilde G_\lambda (v) + \tilde g$ and $\tilde H_\lambda(v)+\tilde h$ of~\eqref{eq:aux SPDE 1} are well-defined and satisfy the integrability conditions of Theorem~\ref{thm L2 linear} by virtue of Lemmas~\ref{lemma truncation 1} and \ref{lemma truncation 2}. It follows by Theorem~\ref{thm L2 linear} that
    \begin{align}
    \Pi_{w_\sigma}(v):&=\cR(w_\sigma, \tilde F_\lambda(u_{\sigma}, v) + \tilde f, \tilde G_\lambda (u_{\sigma}, v) + \tilde g, \tilde H_\lambda(u_{\sigma}, v,\cdot)+\tilde h)\,\in L^2(\Omega,\cM(\sigma,T))
    \end{align}
    is well-defined and gives the unique solution of~\eqref{eq:aux SPDE 1} on $\llbracket \sigma, T \rrbracket$.
    \newline
    We show now that for $\lambda$ and $T^*$ sufficiently small, $\Pi_{w_\sigma}$ becomes a contraction on $L^2(\Omega,\cM(\sigma,\mu))$. Let $v,w\in L^2(\Omega,\cM(\sigma,\mu))$.
    Write $(p_F, \alpha_F)$ for the admissible pair defined in Lemma~\ref{lemma truncation 1} and $(p_A, \alpha_A)$ for the admissible pair defined in Lemma~\ref{lemma truncation 2}.
    Then, by Proposition~\ref{prop: lin apriori},
    \begin{align}
        \label{eq:contraction}
       \|\Pi_{w_\sigma}(v)-\Pi_{w_\sigma}(w)\|_{L^2(\Omega,\cM(\sigma,\mu))}
       &\leq C\| F_\lambda (u_{\sigma},v)- F_\lambda(u_{\sigma},w)\|_{L^2(\Omega,L^{p_F}([\sigma,\mu]
       ,\cV_{\alpha_F}))}
       \\ & \ + C\| F_{\lambda,A_{L}} (u_{\sigma},v)- F_{\lambda,A_L}(u_{\sigma},w)\|_{L^2(\Omega,L^{2}([\sigma,\mu],\cV^*))}
       \\ & \ + C\| F_{\lambda,A_S} (u_{\sigma},v)- F_{\lambda,A_S}(u_{\sigma},w)\|_{L^2(\Omega,L^{p_A}([\sigma,\mu],\cV_{\alpha_A}))}  \\
        &\ + C\|\tilde G_\lambda(u_{\sigma},v)-\tilde G_\lambda(u_{\sigma},w)\|_{L^2(\Omega,L^{2}([\sigma,\mu],\calL_2(U,\cH)))} \\
        &\ + C\|\tilde H_\lambda(u_{\sigma},v,\cdot)-\tilde H_\lambda(u_{\sigma},w,\cdot)\|_{L^2(\Omega,L^{2}([\sigma,\mu],L^2(Z,\cH;\nu)))} \\
        &\leq CC_{\lambda,T^*} \|v-w\|_{L^2(\Omega,\cM(\sigma,\mu))},
    \end{align}
    where for the second estimate we used Lemmas~\ref{lemma truncation 1} and \ref{lemma truncation 2}. We remark that the constant $C$ is independent of $T^*$ and $\lambda$ and that $C_{\lambda,T^*}$ can be made arbitrarily small by choosing $\lambda, T^*$ sufficiently close to $0$. Hence, by choosing $\lambda\in (0,1)$ and $T^*>0$ sufficiently small, meaning that $CC_{\lambda,T^*}<1$, we obtain that $\Pi_{w_\sigma}: L^2(\Omega, \cM(\sigma,\mu)) \rightarrow L^2(\Omega,\cM(\sigma,\mu))$ is a contraction. By Banach's fixed point theorem, there is a unique solution $u\in L^2(\Omega,\cM(\sigma,\mu))$ to the problem
    \begin{equation}
    \label{eq:aux SPDE 2}
        \begin{split}
            du+A_0(t,u_\sigma)u \, dt &= \bigl(\tilde F_\lambda (u_\sigma,u)(t) + \tilde f(t)\bigr) \, dt + \bigl(B_0(t,u_\sigma)u+\tilde G_\lambda(u_\sigma,u)(t) + \tilde g(t)\bigr) \, dW(t) \\
            &\quad + \int_Z\bigl(C_0(t,u_\sigma,z)u(\cdot\, -) + \tilde H_\lambda(u_\sigma,u)(t,z) + \tilde h(t,z)\bigr)\,\Nt(dz,dt)\\
            u(\sigma) &= w_\sigma.
        \end{split}
    \end{equation}
     We claim that if $w_\sigma$ is sufficiently close to $u_\sigma$ then there exists a stopping time $\tau$ such that $\tau > \sigma$ on $\{\sigma < T\}$ and such that $\Theta_\lambda(\cdot,u_\sigma,u) = 1$ on $\llbracket \sigma, \tau \rrbracket$. In this case, $(u, \tau)$ is a solution of~\eqref{eq:spde2} on $\llbracket \sigma, \tau \rrbracket$
     with initial value $w_\sigma$ by construction of~\eqref{eq:aux SPDE 2}.
     Indeed, put
    \begin{align}
    \label{eq:contraction stopping time}
        \tau \coloneqq \inf \bigl\{ t\in [\sigma,\mu] \colon \| u \|_{\cX(\sigma,t)} + \sup_{\sigma \leq s<t} \| u(s) - w_\sigma \|_{\cH} \geq \lambda/2 \bigr\} \wedge \mu.
    \end{align}
    Observe that $\sigma < \mu$ on $\{\sigma < T\}$.
    Now first, $\tau > \sigma$ on $\{\sigma < T\}$ since $u \in \cM(\sigma,\mu)$ with $u(\sigma) = w_\sigma$ almost surely, where we used the right-continuity of $u$ at $\sigma$.

    Second, if $\| u_\sigma - w_\sigma \|_\cH < \lambda/2$ almost surely,
    then with $t \in [\sigma, \mu]$ one has $\sup_{\sigma \leq s<t} \| u(s) - u_\sigma \|_\cH < \sup_{\sigma \leq s< t} \| u(s) - w_\sigma \|_\cH + \lambda/2$ almost surely, so that for $t\in [\sigma, \tau]$ almost surely  $\| u \|_{\cX(\sigma,t)} + \sup_{\sigma \leq s \leq t} \| u(s-) - u_\sigma \|_\cH < \lambda$.
    Consequently, $\Theta_\lambda(\cdot,u_\sigma,u) = 1$ on $\llbracket \sigma, \tau \rrbracket$ follows by definition, and hence $(u,\tau)$ is indeed a solution for~\eqref{eq:spde2} on $\llbracket \sigma, \tau \rrbracket$ with initial value $w_\sigma$.

    Note that, in particular, we have $u = \Pi_{w_\sigma}(u)$ on $\llbracket\sigma,\tau\rrbracket$ (we use this relation in later steps of this proof).
    Finally, we may take $w_\sigma \coloneqq u_\sigma$ to complete this step.

    \textbf{Step 2}: uniqueness of the local solution.
    Let $(u, \tau)$ be the local solution on $\llbracket \sigma, T \rrbracket$ constructed in Step~1 for $w_\sigma \coloneqq u_\sigma$ and let $(v, \wt\tau)$ be another local solution of~\eqref{eq:spde2} with $v(\sigma) = u_\sigma$. By definition, there is a sequence of stopping times $(\wt\tau_n)_{n \geq 1}$ such that $\wt\tau_n \uparrow \wt\tau$ and such that $v \1_{\llbracket \sigma, \wt\tau_n \rrbracket}$ is a solution of~\eqref{eq:spde2} on $\llbracket \sigma, \wt\tau_n \rrbracket$. For every $n$ define a new stopping time
    \begin{align}
        \tau_n \coloneqq \inf \bigl\{ t \in [\sigma, \tau \wedge \wt\tau_n] \colon \| v \|_{\cX(\sigma,t)} + \sup_{\sigma \leq s \leq t} \| v(s-) - u_\sigma \|_{\cH} \geq \lambda/2 \bigr\}\wedge \tau \wedge \wt\tau_n.
    \end{align}
    Note that, in particular, $\tau_n \leq \tau \leq T$. By the construction in Step~1 we have $u = \Pi_{u_\sigma}(u)$ on $\llbracket \sigma, \tau_n \rrbracket$. Also, since $\Theta_\lambda(t, u_\sigma, v) = 1$ on $\llbracket \sigma, \tau_n \rrbracket$, it follows $v = \Pi_{u_\sigma}(v)$ on $\llbracket \sigma, \tau_n \rrbracket$. Thus, arguing similar as in~\eqref{eq:contraction} we obtain
    \begin{align}
        \left( \E \| u-v \|_{\cM(\sigma, \tau_n)}^2 \right)^\frac{1}{2} \leq C C_{\lambda,T} \left( \E \| u - v \|_{\cM(\sigma,\tau_n)}^2 \right)^\frac{1}{2},
    \end{align}
    where $C C_{\lambda,T} < 1$. Consequently, $u = v$ on $\llbracket \sigma, \tau_n \rrbracket$. Therefore, by definition of $\tau$ in Step 1, we conclude furthermore $\tau_n = \tau \wedge \wt \tau_n$. Eventually, since $\tau_n = \tau \wedge \wt \tau_n \uparrow \tau \wedge \wt\tau$, we deduce $u = v$ on $\llbracket \sigma, \tau \wedge \wt\tau \rrparenthesis$. This shows uniqueness.

    \textbf{Step 3}: existence of a maximal solution.
    We show the existence of a maximal solution to \eqref{eq:spde2}. The argument is taken from \cite{hornung2019quasilinear}. Let $\cS$ be the set of stopping times $\tau:\Omega\rightarrow [0,T]$ such that there exists a unique local solution $u^\tau$ on $\llbracket \sigma, \tau\rrparenthesis$ to \eqref{eq:spde2} with initial value $u_\sigma$. From Step~2 we know that $\cS$ is non-empty. Next, we show that $\cS$ is closed under the maximum operation $\vee$, that is to say, that for $\tau_1,\tau_2\in\cS$ we have $\tau_1\vee\tau_2\in\cS$. For that purpose, observe first that by uniqueness of the local solutions we have $u^{\tau_1} = u^{\tau_2}$ on $\llbracket \sigma,\tau_1\wedge\tau_2\rrparenthesis$. Hence, it follows by standard stopping time arguments that
    $$
    \bar u(t) := u^{\tau_1}(t\wedge \tau_1) + u^{\tau_2}(t\wedge\tau_2) - u^{\tau_1}(t\wedge\tau_1\wedge\tau_2)
    $$
    is a local solution of~\eqref{eq:spde2} on $\llbracket \sigma,\tau_1\vee\tau_2\rrparenthesis$. Its uniqueness follows from the uniqueness of $u^{\tau_1}$ or $u^{\tau_2}$, respectively. In summary, this means that $\tau_1\vee\tau_2\in\cS$. By \cite[Thm.~A.3]{karatzas2014brownian} we hence know that $\tau:=\esssup \cS$ is also a stopping time and there exists a sequence of stopping times $(\tau_n)_{n\geq 1}\subset\cS$ such that $\tau_n\uparrow\tau$ as $n\to\infty$ and there exists (by compatibility of unique local solutions) a function $u$ such that for each $n\geq 1$, $(u\1_{\llbracket \sigma,\tau_n\rrbracket},\tau_n)$ is a solution to~\eqref{eq:spde2} on $\llbracket \sigma,\tau_n\rrbracket$. Hence $(u,\tau)$ is the (unique) maximal solution on $\llbracket \sigma, T \rrbracket$ with localizing sequence given by $(\tau_n)_{n\in\bN}$.

    \textbf{Step 4}: continuous dependence on initial data.
    In addition to $u_\sigma$, let $v_\sigma$ be another $\cH$-valued and $\cF_\sigma$-measurable bounded initial value satisfying the closeness condition $\| u_\sigma - v_\sigma \|_\cH < \lambda/2$ almost surely introduced in Step~1. They have corresponding unique maximal solutions $(u, \tau_1)$ for $u_\sigma$ and $(v, \tau_2)$ for $v_\sigma$ on $\llbracket \sigma, T \rrbracket$ as constructed in Step~1-3. Due to Step~1, we have  that $\tau_1>\mu$ and $\tau_2>\mu$, and $u = \Pi_{u_\sigma}(u)$ and $v = \Pi_{v_\sigma}(v)$ on $\llbracket \sigma, \tilde \tau_1 \wedge \tilde \tau_2 \rrbracket$, where the positive stopping times $\tilde \tau_1$ and $\tilde \tau_2$ are defined analogously to~\eqref{eq:contraction stopping time} as 
        \begin{align}
    \label{eq:contraction stopping time2}
        \tilde \tau_1 &\coloneqq \inf \bigl\{ t\in [\sigma,\mu] \colon \| u \|_{\cX(\sigma,t)} + \sup_{\sigma \leq s<t} \| u(s) - u_\sigma \|_{\cH} \geq \lambda/2 \bigr\} \wedge \mu,
        \\ \tilde \tau_2 &\coloneqq \inf \bigl\{ t\in [\sigma,\mu] \colon \| v \|_{\cX(\sigma,t)} + \sup_{\sigma \leq s<t} \| v(s) - v_\sigma \|_{\cH} \geq \lambda/2 \bigr\} \wedge \mu.
    \end{align}
     Put $\nu \coloneqq \tilde \tau_1 \wedge \tilde \tau_2$. Now, we can write $\Pi_{u_\sigma}(u) = \cR(u_\sigma,0,0,0) + \Pi_0(u)$ and $\Pi_{v_\sigma}(v) = \cR(v_\sigma,0,0,0) + \Pi_0(v)$. Hence, arguing similar as in Step~2,
    \begin{align}
        \| u - v \|_{L^2(\Omega,\cM(\sigma, \nu))}&\leq \| \cR(u_\sigma-v_\sigma,0,0,0)\|_{L^2(\Omega,\cM(\sigma,\nu))}+ \| \Pi_0(u) - \Pi_0(v) \|_{L^2(\Omega,\cM(\sigma, \nu))}\\
        &\leq C \| u_\sigma - v_\sigma \|_{L^2(\Omega,\cH)} + C C_{\lambda,T^*} \| u - v \|_{L^2(\Omega,\cM(\sigma, \nu))}.
    \end{align}
    Since $C C_{\lambda,T^*} < 1$, we can absorb the second term of the right-hand side into the left-hand side to obtain
    \begin{align}
        \| u - v \|_{L^2(\Omega,\cM(\sigma, \nu))} \leq C \| u_\sigma - v_\sigma \|_{L^2(\Omega,\cH)},
    \end{align}
    where $C$ now also depends on $\lambda$ and $T^*$.
    Therefore, with $\eta \coloneqq \lambda/2$ and a stopping time $\nu$ with $\nu \in (\sigma, \wt\tau_1 \wedge \wt\tau_2]$ on $[\sigma < T]$ the claim follows.

    \textbf{Step 5}: localization property.
    Let $(u,\tau)$ be the maximal solution on $\llbracket \sigma, T \rrbracket$ with initial condition $u_\sigma$ and let $(v,\wt\tau)$ be the maximal solution on $\llbracket \sigma, T \rrbracket$ with initial condition $v_\sigma$, and put $\Gamma \coloneqq \{u_\sigma = v_\sigma\}$ as in the statement.  Set $\mu \coloneqq \tau \1_\Gamma + \wt\tau \1_{\Omega\setminus \Gamma}$ and $w \coloneqq u\1_\Gamma \1_{[\sigma,\tau)} + v \1_{\Omega\setminus \Gamma} \1_{[\sigma,\wt\tau)}$.
    Using causality and the results from above it follows that $(w,\mu)$ is a local solution to \eqref{eq:spde2} on $\llbracket \sigma, T \rrbracket$ with initial value $v_\sigma$. Since $(v,\wt\tau)$ is maximal, we conclude that $\tau = \mu \leq \wt\tau$ on $\Gamma$ and
    $$
    u = w = v\quad \text{on}\quad \Gamma\times [\sigma,\tau).
    $$
    Indeed, to see this it suffices to multiply both sides of the SPDE \eqref{eq:spde2} by $1_\Gamma$. Using that $\Gamma\in\cF_\sigma$ together with the locality property of stochastic integrals we get a.s. for  $t\in \llbracket\sigma,T\rrbracket$,
    \begin{align}
        \1_\Gamma\int_\sigma^t B(s,u(s))\,dW(s) = \int_\sigma^t 1_\Gamma B(s,u(s))\,dW(s) = \int_\sigma^t 1_\Gamma B(s,w(s))\,dW(s), 
    \end{align}
    and analogously for the $\tilde N$ integral.
    Swapping the roles of $(u,\tau)$ and $(v,\wt\tau)$ we likewise obtain $\wt\tau\leq\tau$ on $\Gamma$ and $u=v$ on $\Gamma\times [\sigma,\wt\tau)$. Therefore, $\tau = \wt\tau$ on $\Gamma$ with $u \1_{\Gamma\times [\sigma,\tau)} = v \1_{\Gamma\times [\sigma,\wt\tau)}$. This finishes the proof.
\end{proof}

The next result is similar to Theorem~\ref{thm wellposedness local}, but we consider now the case of initial data and inhomogeneities without moments. The proof is based on localization arguments in conjunction with Theorem~\ref{thm wellposedness local}.

\begin{theorem}
\label{thm well-posedness unbounded}
    Let $T>0$ and
    suppose Assumption~\ref{assumption operators new}.
    Let
    \begin{align}
        f_i\in L^0(\Omega,L^{p_{i}}([0,T],\cV_{\theta_{i}})),\;\; g\in L^0(\Omega\times [0,T],\cL_2(U,\cH)),\;\; h\in L^0(\Omega\times [0,T],L^2(Z,\cH;\nu)),
    \end{align}
    be $\cP$-measurable, $\cP$-measurable, and $\cP^-$-measurable respectively, where the pairs $(p_{i},\theta_{i})$ are admissible in the sense of Definition~\ref{def:admissible}. Put $f = \sum_{i=1}^{m_f} f_i$.
    Also, let $u_0$ be an $\cH$-valued and $\cF_0$-measurable initial value.
    Then the following holds:
    \begin{enumerate}
        \item (Existence and uniqueness) 
        There exists a unique maximal solution $(u,\tau)$ to \eqref{eq:spde2} on $\llbracket 0, T \rrbracket$ such that almost surely $\tau>0$.
        \item (Continuous dependence on initial data) Fix $n \geq 1$ and define $\Gamma_n \coloneqq\{\| u_0 \|_\cH \leq n\}$. Then there are constants $C,\eta > 0$ such that for another $\cF_0$-measurable initial value $v_0$ with maximal solution $(v,\wt \tau)$ on $\llbracket 0, T \rrbracket$ and satisfying the closeness condition $\| u_0 \1_{\Gamma_n} - v_0 \1_{\Gamma_n} \|_\cH < \eta$ almost surely there exists a stopping time $\nu$ with $\nu\in (0, \tau \wedge \wt\tau]$ such that
        \begin{align}
        	\left( \E \| (u - v) \1_{\Gamma_n} \|_{\cM(\nu)}^2 \right)^\frac{1}{2} \leq C \left( \E \| (u_0 - v_0) \1_{\Gamma_n} \|_\cH^2 \right)^\frac{1}{2}.
        \end{align}
        \item (Localization property) Given another $\cF_0$-measurable initial value $v_0$ with corresponding maximal solution $(v,\wt \tau)$ one has with $\Gamma \coloneqq \{u_0 = v_0\}$ the identities
        \begin{align}
            \tau\1_{\Gamma} = \wt \tau \1_\Gamma \quad \mathrm{and} \quad u\1_\Gamma = v \1_\Gamma.
        \end{align}
    \end{enumerate}
\end{theorem}

\begin{proof}
    We only have to show the existence of a local solution. The proofs for the remaining assertions do not rely on the moment condition and can hence be taken from the proof of Theorem~\ref{thm wellposedness local}, compare also with~\cite[Thm.~4.7]{agresti2022nonlinear}.

    Now we construct a local solution for~\eqref{eq:spde2}. For $n\geq 1$ define the truncated initial value $u_{0,n} \coloneqq u_0 \1_{\|u_0\|_\cH\leq n}$. Also, define the stopping time
    $$
    \mu =\inf \bigl\{ t \in [0,T]: \max_i \|f_i\|_{L^{p_i}([0,t],\cV_{\theta_i})}+\|g\|_{L^2([0,t],\cL^2(U,\cH))} + \|h\|_{L^2([0,t],L^2(Z,\cH;\nu))} \geq 1 \bigr\} \wedge T,
    $$
    and note that $\mu>0$ almost surely, as the processes in its argument are continuous and starting at $0$, almost surely.
    By Theorem~\ref{thm wellposedness local} and using the localizing sequence for a maximal solution, for each $n\geq 1$ there exists a solution $u_n$ on $\llbracket 0, \tau_n \rrbracket$ to \eqref{eq:spde2} with initial condition $u_{0,n}$ and right-hand sides given by
    \begin{align}
    \label{eq: right-hand side truncated}
    f\1_{\llbracket 0, \mu \rrbracket},\quad
    g\1_{\llbracket 0, \mu \rrbracket},\quad
    h\1_{\llbracket 0, \mu \rrbracket}.
    \end{align}
    We now construct a solution $u$ on $\llbracket 0, \tau \rrbracket$ to \eqref{eq:spde2} with initial condition $u_0$ and right-hand sides given again by~\eqref{eq: right-hand side truncated}.
    Define $\Lambda^1 \coloneqq \Gamma^1$, $\Lambda^n \coloneqq \Gamma^n\setminus \Gamma^{n-1}$ for $n \geq 2$, and put
    \begin{equation}
    \label{eq: right-hand side 2}
    u\coloneqq \sum_{n\geq 1} u_n\1_{\Lambda^n},\quad \tau' \coloneqq \sum_{n\geq 1} \tau_n\1_{\Lambda^n}.
    \end{equation}
    Almost surely we have $\tau'>0$ since this is the case for each individual stopping time $\tau_n$.
    By causality and the fact that $u_n \1_{\Gamma^n}$ is a solution on $\llbracket 0, \tau_n \rrbracket$ to~\eqref{eq:spde2} with initial value $u_{0,n} \1_{\Gamma^n} = u_0 \1_{\Gamma^n}$ and right-hand sides as in~\eqref{eq: right-hand side truncated}, it follows that $u$ is a solution on $\llbracket 0, \tau' \rrbracket$ to~\eqref{eq:spde2} with initial condition $u_0$ and right-hand sides as in~\eqref{eq: right-hand side truncated}.
    Finally, if we put $\tau \coloneqq \tau' \wedge \mu$, then $u$ is a solution to~\eqref{eq:spde2} on $\llbracket 0, \tau \rrbracket$ with initial value $u_0$ and right-hand sides $(f,g,h)$.
\end{proof}

\begin{remark}
In Assumption~\ref{assumption operators new} the functions $K_{A,n,T}$, $K_{B,n,T}$, $K_{C,n,T}$ can also be allowed to depend on $\omega$. Theorem~\ref{thm well-posedness unbounded} remains true in this setting, and this follows from a standard localization argument.

Moreover, in case the nonlinearities $F$, $G$, $H$ are subcritical it should also be possible to replace the constant in the locally Lipschitz estimate by a function/process that is singular in $t$.
\end{remark}

\section{Blow up criteria}
\label{sec:blowup}

\noindent
The main result of the last section (Theorem~\ref{thm well-posedness unbounded}) ensures existence of a maximal solution $(u,\tau)$ to the problem~\eqref{eq:spde2}. An important question is if this maximal solution is in fact a global solution. In this section,the  we provide blowup criteria that enable us to exclude a blowup in finite time. More precisely, if we fix a finite time $T>0$ and we know that the solution $u$ stays bounded almost surely when approaching $\tau \wedge T$, then we can conclude that $\tau \geq T$ with probability one. The precise result will be given in Theorem~\ref{thm: blowup}

Fix a finite time $T>0$ throughout this section.
If $f_i\in L^0(\Omega,L^{p_{i}}([0,T],\cV_{\theta_{i}}))$, $g\in L^0(\Omega\times [0,T],\cL_2(U,\cH))$, and $h\in L^0(\Omega\times [0,T],L^2(Z,\cH;\nu))$ are $\cP$-measurable, $\cP$-measurable, and $\cP^-$-measurable, respectively, and the pairs $(p_{i},\theta_{i})$ are admissible, then define for $n\geq 1$ the stopping time
$$
\tau_n:=\inf\{t\in [0,T]: \max_i\|f_i\|_{L^{p_i}([0,t],\cV_{\theta_i})} + \|g\|_{L^2([0,t],\cL_2(U,\cH))} + \|h\|_{L^2([0,t],L^2(Z,\cH;\nu))}\geq n\}\wedge T.
$$
Put $f = \sum_{i=1}^{m_f} f_i$ as usual.
Introduce the truncated processes
\begin{align}
\label{eq: def truncated data}
f_n:= f \1_{[0,\tau_n]},\quad g_n(t):=g \1_{[0,\tau_n]},\quad h_n:=h \1_{[0,\tau_n]},\quad n\geq 1.
\end{align}
As before, we let further
\begin{align}
\label{eq: def truncated initial value}
    u_{0,n} \coloneqq u_0\1_{\|u_0\|\leq n},\quad n\geq 1.
\end{align}
The reader should also recall the space $\cM(a,b)$ from \eqref{eq:cMspaces}. We start out with a first reduction.

\begin{lemma}
\label{lemma event localization}
    Let Assumption~\ref{assumption operators new} hold.
    Let
    \begin{align}
        f_i\in L^0(\Omega,L^{p_{i}}([0,T],\cV_{\theta_{i}})),\;\; g\in L^0(\Omega\times [0,T],\cL_2(U,\cH)),\;\; h\in L^0(\Omega\times [0,T],L^2(Z,\cH;\nu)),
    \end{align}
    be $\cP$-measurable, $\cP$-measurable, and $\cP^-$-measurable respectively, where the pairs $(p_{i},\theta_{i})$ are admissible. Put $f = \sum_{i=1}^{m_f} f_i$.
    Also, let $u_0$ be $\cH$-valued and $\cF_0$-measurable.
    Let the truncated data $(u_{0,n}, f_n, g_n, h_n)$ be defined as before the lemma.
    Consider maximal solutions $(u,\sigma)$ and $(u_n,\sigma_n)$, $n\geq 1$, where
    \begin{gather}
        (u,\sigma) \quad\quad \text{solves \eqref{eq:spde2} with original}\quad (u_0,f,g,h),\\
        (u_n,\sigma_n)\quad \text{solves \eqref{eq:spde2} with the truncations}\quad (u_{0,n},f_n,g_n,h_n)\text{ in place of }(u_0,f,g,h).
    \end{gather}
    Consider a Borel measurable mapping
    $$
    \cO: \cM(T) \times\bR_+\mapsto \bR\cup\{\infty\}.
    $$
    Then, if 
    \begin{equation*}
        \liminf_{n\to \infty} P\Big(\sigma_n<T,\quad \cO(u_n,\sigma_n)<\infty \Big) = 0,
    \end{equation*}
    we also have
    \begin{equation*}
        P\Big(\sigma<T,\quad \cO(u,\sigma)<\infty \Big) = 0.
    \end{equation*}
\end{lemma}
\begin{proof}
    Throughout the proof, all solutions refer to problem~\eqref{eq:spde2} on $\llbracket 0, T \rrbracket$ without explicitly mentioning it.
    Observe also that the maximal solutions $(u,\sigma)$ and $(u_n,\sigma_n)$ exist owing to Theorem~\ref{thm well-posedness unbounded} and using Assumption~\ref{assumption operators new}.

    Put $\Gamma_n \coloneqq \{\|u_0\|_\cH\leq n\}$.
    On the one hand, by construction $\Gamma_n \uparrow \Omega$ as $n \to \infty$, and on the other hand, $\{\tau_n = T\} \uparrow \Omega$ as $n \to \infty$ by the integrability condition for the data. Therefore,
    \begin{align}
    \label{eq: localization set convergence}
        \Gamma_n\cap \{\tau_n = T\} \uparrow \Omega \quad \text{as } n \to \infty.
    \end{align}
    Owing to Theorem~\ref{thm well-posedness unbounded} there are maximal solutions $(v_n,\wt\tau_n)$ with initial value $u_0$ and truncated data $(f_n,g_n,h_n)$.
    First, we compare $u$ with $v_n$. To this end, observe first that $(u, \sigma \wedge \tau_n)$ is a local solution with initial condition $u_0$ and truncated data $(f_n,g_n,h_n)$. Thus, by maximality, $\sigma \wedge \tau_n \leq \wt\tau_n$ and $u = v_n$ on $\llbracket 0, \sigma \wedge \tau_n \rrparenthesis$.
    Next, we compare $u_n$ with $v_n$.
    Due to the localization property in Theorem~\ref{thm well-posedness unbounded}, also $\sigma_n = \wt\tau_n$ as well as $u_n = v_n$ on $\Gamma_n$.
    Hence, combining both facts, we find on $\{\tau_n = T\} \cap \Gamma_n$ that $\sigma = \sigma \wedge \tau_n \leq \wt \tau_n = \sigma_n$ with $u = v_n = u_n$ on $[0, \sigma)\times \{\tau_n = T\} \cap \Gamma_n$.
    Similarly, since $(v_n, \wt\tau_n \wedge \tau_n)$ is a local solution with initial value $u_0$ and data $(f,g,h)$, it follows from maximality of $u$ that $\wt \tau_n \wedge \tau_n \leq \sigma$, which gives $\sigma_n = \wt \tau_n \leq \sigma$ on $\{\tau_n = T\} \cap \Gamma_n$.
    In summary, we conclude
    \begin{align}
    \label{eq: localization lemma solutions and times coincide}
        \sigma = \sigma_n \quad\text{and}\quad u = u_n \quad\text{on}\quad \Gamma_n \cap \{\tau_n = T\}.
    \end{align}
    Then, due to~\eqref{eq: localization set convergence} and \eqref{eq: localization lemma solutions and times coincide}, we get
    \begin{align*}
        P(\{\sigma<T\}\cap\{\cO(u,\sigma)<\infty\})& = \lim_{n\to\infty} P\Big( \{\sigma<T\}\cap\{\cO(u,\sigma)<\infty\}\cap\{\tau_n = T\}\cap\Gamma_n \Big)\\
        &  =  \lim_{n\to\infty} P\Big( \{\sigma_n<T\}\cap\{\cO(u_n,\sigma_n)<\infty\}\cap\{\tau_n = T\}\cap\Gamma_n \Big)\\
        & \leq \liminf_{n\to\infty} P\Big( \{\sigma_n<T\}\cap\{\cO(u_n,\sigma_n)<\infty\}\Big)=0. \qedhere
    \end{align*}
\end{proof}

Next, we prove a first blowup criterion. Its assumption, the existence of the limit at the maximal existence time of the solution, will be replaced by the finiteness of a supremum in Theorem~\ref{thm: blowup}. The idea of its proof is simple: if limits at the final time $\sigma$ exist, then we can restart the equation with initial time $\sigma$ using the local wellposedness result of Theorem~\ref{thm wellposedness local} to deduce a contradiction to the maximality of the solution.

The procedure is similar to that in~\cite[Lem. 5.4]{agresti2022nonlinear2}, but the additional jump part presents further challenges. These will be solved by using an extension of the stochastic integral to progressive integrands developed in~\cite[Thm. 3.3.2]{zhu2010study}. One of the reasons we need this is that, a priori, we do not know that $\sigma$ is predictable. Therefore, we cannot exclude the possibility of jumps at $\sigma$ in the first place.
However, in Lemma~\ref{lem: sigma predictable} below, we show that $\sigma$ is indeed predictable, which then implies that the \Ly noise does not introduce jumps at $\sigma$.

\begin{proposition}
	\label{prop: P W_lim = 0} Let Assumption~\ref{assumption operators new} be satisfied.
    Let
    \begin{align}
        f_i\in L^0(\Omega,L^{p_{i}}([0,T],\cV_{\theta_{i}})),\;\; g\in L^0(\Omega\times [0,T],\cL_2(U,\cH)),\;\; h\in L^0(\Omega\times [0,T],L^2(Z,\cH;\nu)),
    \end{align}
    be $\cP$-measurable, $\cP$-measurable, and $\cP^-$-measurable respectively, where the pairs $(p_{i},\theta_{i})$ are admissible in the sense of Definition~\ref{def:admissible}. Put $f = \sum_{i=1}^{m_f} f_i$.
    Also, let $u_0$ be $\cH$-valued, $\cF_0$-measurable, and bounded.
	If $(u,\sigma)$ is the maximal solution to \eqref{eq:spde2} on $\llbracket 0, T \rrbracket$ provided by Theorem~\ref{thm wellposedness local},
	then
	\begin{equation}
		\label{eq: blow up limit}
		P\Bigl( \sigma<T,\, \lim_{t\uparrow\sigma} u(t)\text{ exists in $\cH$ and }\|u\|_{L^2([0,\sigma),\cV)}<\infty\Bigr) = 0.
	\end{equation}
\end{proposition}
\begin{proof}
    By Lemma~\ref{lemma event localization}, we may assume that the data is already truncated.
	Define
	$$
	\cW_{\lim} \coloneqq \big\{ \sigma<T,\, \lim_{t\uparrow\sigma} u(t)\text{ exists in $\cH$ and }\|u\|_{L^2([0,\sigma),\cV)}<
	\infty\big\},
	$$
	and assume for the sake of contradiction that $P(\cW_{\lim})>0$.
    By the \cl property, $\sup_{s \in [0, \sigma)} \| u(s) \|_H$ is finite on $\cW_{\lim}$.
	Hence, by an exhaustion argument, there exists $M > 0$ and a $\cF_\sigma$-measurable subset $\cW\subset \cW_{\lim}$ such that $P(\cW)>0$ and
	$$
	\sup_{s\in [0,\sigma)}\|u(s)\|_\cH + \|u\|_{L^2([0,\sigma),\cV)} \leq M\quad\text{for all $\omega\in\cW$.}
	$$
    This also implies $\|u(\sigma-)\|_{\cH}\leq M$ on the set $\cW$.
    Define the mappings
    \begin{align}
        \bar C(t,x,z):&= 
        C(t,x,z)\1_{\|\cdot\|_{\cH}\leq M}(x)\big(\1_{\llbracket 0,\sigma\rrparenthesis}(t)  + \1_{\llbracket\sigma\rrbracket}(t)\1_{\cW}\big),\\
        J_{\bar C}(t):&= \int_0^t\int_Z \bar C(s,u(s-),z)
        \,\wt{N}(dz,ds),\quad t\in [0,T],
    \end{align}
    We first observe that with this choice of $\bar C$ we have $\bar C(t,u(t-),z) = C(t,u(t-),z)$ for almost all $\omega\in\cW$ on $\llbracket 0,\sigma\rrparenthesis$.
    We note, moreover that $\1_{\llbracket \sigma\rrbracket}\1_{\cW}$ is adapted  and that on $\cW$, the left limit $u(\sigma-)$ is well-defined and bounded by $M$, so that also on $\llbracket\sigma\rrbracket$ we have $\bar C(\sigma,u(\sigma-),z) = C(\sigma,u(\sigma-),z)$ for almost all $\omega \in\cW$. Hence by \cite[Thm. 3.3.2]{zhu2010study}
     $J_{\bar C}$ is a well-defined martingale and admits a \cl modification which we henceforth refer to by $J_{\bar C}$ and there exists $M_J>0$ and a set $\cW_J\subset \cW$, $\cW_J\in\cF_\sigma$, of positive probability such that $\|\Delta J_{\bar C}(\sigma)\|_{\cH}\leq M_J$.
	Define
    an initial condition
    $$
    v_\sigma \coloneqq \1_{\cW_J}\big(\lim_{t\uparrow\sigma} u(t) + \Delta J_{\bar C}(\sigma)\big).
    $$
    By construction, $v_\sigma$ is $\cF_\sigma$-measurable and $\|v_\sigma\|_\cH\leq M+M_J$ almost surely. Therefore, using Theorem~\ref{thm wellposedness local} for the initial time $\sigma$, there exists a maximal solution $(\wt u,\wt\sigma)$ on $\llbracket \sigma, T \rrbracket$ to the SPDE
	\begin{equation*}
		\begin{split}
			du(t)+A(t,u(t)) \,dt ={} &B(t,u(t)) \,dW(t) +\1_{\llparenthesis \sigma,T\rrbracket}(t)\int_Z C(t,u(t-),z)\,\Nt(dz,dt)\\
			u(\sigma)={} &v_\sigma.
		\end{split}
	\end{equation*}
	Clearly $\wt\sigma>\sigma$ on $\cW_J$ as $\cW_J \subset \cW_{\lim} \subset \{\sigma < T\}$. Define
	$$
	\bar u \coloneqq u\1_{\llbracket 0, \sigma \rrparenthesis} + \wt u\1_{\llbracket\sigma,\wt{\sigma} \rrparenthesis}\1_{\cW_J},
	$$
    as well as the square integrable martingale
    \begin{align}
        J_{C}(t) :&= \int_0^t\int_Z C(s,\bar u(s-),z)\1_{\|\cdot\|\leq M}(\bar u(s-))
        \,\wt{N}(dz,ds),\quad t\in [0,\sigma].
    \end{align}
	By definition of the initial condition we have
    \begin{align}
        \Delta\bar u(\sigma)\1_{\cW_J}& = \big(\bar u(\sigma) - \lim_{t\uparrow\sigma}\bar u(t)\big)\1_{\cW_J}\\
        & = \big(\lim_{t\uparrow\sigma} u(t)  +\Delta J_{\bar C}(\sigma) - \lim_{t\uparrow\sigma} u(t)\big) \1_{\cW_J}\\
        &= \Delta J_{\bar C}(\sigma)\1_{\cW_J}.
    \end{align}
   Further, by \cite[Thm. 3.3.4]{zhu2010study}, the martingales $J_C$ and $J_{\bar C}$ are indistinguishable on $\llbracket 0,\sigma\rrbracket$. Indeed, recall again that $\bar C(s,u(s-),z) = \bar C(s,\bar u(s-),z)$ on $\llbracket 0,\sigma \rrparenthesis$ and that we always work with the \cl modification, so that this follows from
   \begin{align}
   \E\int_0^\sigma \int_Z
\|C(s,\bar u(s-),z)\1_{\|\cdot\|\leq M}(\bar u(s-)) - \bar C(s,u(s-),z)\|_{\cH}^2\,\nu(dz)ds   = 0.
\end{align}
In particular,
    \begin{align}
        \Delta J_{\bar C}(\sigma)\1_{\cW_J} &=\Delta J_C(\sigma)\1_{\cW_J} \\
        &=  \Delta \int_0^\sigma \int_Z C(s,\bar u(s-),z)\,\wt{N}(dz,ds)\1_{\cW_J}
    \end{align}
    almost surely.  Therefore, with the stopping time
    $$
    \bar\sigma \coloneqq \sigma\1_{\Omega\setminus\cW_J} + \wt\sigma\1_{\cW_J}.
    $$
    it is straightforward to see that $(\bar u,\bar\sigma)$
	is a local solution to \eqref{eq:spde2} on $\llbracket 0, T \rrbracket$ with initial condition $u_0$. Since $\bar\sigma > \sigma$ on a set of positive measure, this is a contradiction to the maximality of $(u,\sigma)$.
\end{proof}

With this we can see that the blow up times are predictable.
While this is not of immediate importance here, we include it, together with a short proof, as valuable information for future articles on this subject.

\begin{lemma}
\label{lem: sigma predictable}
    Let the conditions of Proposition~\ref{prop: P W_lim = 0} hold. Then $\sigma$ is predictable.
\end{lemma}
\begin{proof}
Consider a localizing sequence $(\sigma_n)_{n\geq 1}$ for $u$, verifying $\sigma_n\to\sigma$ as $n\to\infty$, so that by definition $u\in \cM(\sigma_n)$ a.s.\ is a solution to \eqref{eq:spde2} on $\llbracket 0,\sigma_n\rrbracket$. We claim that for the sequence $(\sigma_n\wedge (T-\tfrac{1}{n}))_{n\geq 1}$ we have $\sigma_n\wedge (T-\tfrac{1}{n})\to\sigma$ as $n\to\infty$, while $\sigma_n\wedge (T-\tfrac{1}{n})<\sigma$ for all $n\geq 1$. The convergence holds by definition of $(\sigma_n)_{n\geq 1}$, so that it suffices to prove the strict inequality.
Indeed, from the path regularity of $u$ on $[0,\sigma_n]$ and by Proposition~\ref{prop: P W_lim = 0}, we see
    \begin{align}
        P(\sigma_n = \sigma<T)& = P(\sigma_n = \sigma<T,\  \lim_{t\uparrow\sigma_n} u(t)\ \text{exists in $\cH$ and}\ \|u\|_{L^2([0,\sigma_n],\cV)}<\infty)\\
        & = P(\sigma_n = \sigma<T,\ \lim_{t\uparrow\sigma} u(t)\ \text{exists in $\cH$ and}\ \|u\|_{L^2([0,\sigma],\cV)}<\infty) = 0,
    \end{align}
which proves $\sigma_n<\sigma$ on $\{\sigma<T\}$. Clearly also $\sigma\wedge (T-\tfrac{1}{n})<\sigma = T$ on $\{\sigma = T\}$. This finishes the proof.
\end{proof}

Finally, we have the tools to prove the main theorem of this section, a blowup criterion based on finiteness of the supremum of the maximal solution. Using a stopping time argument and suitable linear auxiliary problems, this case can be reduced to the blowup criterion in Proposition~\ref{prop: P W_lim = 0}.
\begin{theorem}
\label{thm: blowup}
    Let $T>0$.
    Suppose Assumption~\ref{assumption operators new}.
    Let
    \begin{align}
        f_i\in L^0(\Omega,L^{p_{i}}([0,T],\cV_{\theta_{i}})),\;\; g\in L^0(\Omega\times [0,T],\cL_2(U,\cH)),\;\; h\in L^0(\Omega\times [0,T],L^2(Z,\cH;\nu)),
    \end{align}
    be $\cP$-measurable, $\cP$-measurable, and $\cP^-$-measurable respectively, where the pairs $(p_{i},\theta_{i})$ are admissible in the sense of Definition~\ref{def:admissible}. Put $f = \sum_{i=1}^{m_f} f_i$.
    Also, let $u_0$ be $\cH$-valued and $\cF_0$-measurable.
    Let $(u,\sigma)$ be the maximal solution to \eqref{eq:spde2} on $[0, T]$.
    Then
    \begin{align}
        P\Bigl( \sigma<T,\quad \sup_{t\in [0,\sigma)} \|u(t)\|_\cH + \int_0^\sigma \|u(t)\|_\cV^2\,dt<\infty \Bigr) &= 0. \label{eq: blowup criterion 1}
    \end{align}
\end{theorem}
\begin{proof}
    As in the proof of Proposition~\ref{prop: P W_lim = 0} we may assume the data is truncated already.  Let
\begin{align}
		\cW_{\sup} &\coloneqq \big\{\sigma<T, \quad \sup_{t\in [0,\sigma)}\|u(t)\|_\cH + \int_0^\sigma \|u(t)\|_\cV^2\,dt <\infty\big\}, \\
		\cW_{\lim} &\coloneqq\big\{\sigma<T, \quad \lim_{t\uparrow\sigma} u(t)\in \cH \quad\text{and}\quad \int_0^\sigma \|u(t)\|_\cV^2\,dt <\infty\big\}.
\end{align}
    Define
    \begin{align}
        \tau_n \coloneqq \inf\bigl\{t \in [0,\sigma): \| u \|_{\cM(t)} \geq n \bigr\}\wedge \sigma.
    \end{align}
Then $\sup_{t\in [0,\tau_n)} \|u(t)\|_{\cH} + \|u\|_{L^2([0,\tau_n],\cV)}\leq n$.
    By construction of $\cW_{\sup}$ one has
    \begin{align}
    \label{eq: stopping time limit on sigma<T}
    \{\tau_n = \sigma < T \} \cap \cW_{\sup}
    = \{\tau_n = \sigma\}\cap \cW_{\sup}\uparrow  \cW_{\sup} \quad \text{as $n\to\infty$}.
    \end{align}
    Now, for each $n\geq 1$ we define the operators
    \begin{align}
        &\bar A_n(t)v \coloneqq A_L(t,u(t)\1_{\llbracket 0, \tau_n \rrparenthesis}(t))v +  A_0(t,u(t)\1_{\llbracket 0, \tau_n \rrparenthesis}(t))v,\;\;
        \\ & \bar B_n(t)v \coloneqq B_0(t,u(t)\1_{\llbracket 0, \tau_n \rrparenthesis}(t))v, \qquad \bar C_n(t,z) \coloneqq C_0 (t,u(t-)\1_{\llbracket 0, \tau_n \rrparenthesis}(t),z)
    \end{align}
    as well as the right-hand sides
    \begin{align}
        F_n(t) &\coloneqq F(t,u(t)\1_{\llbracket 0, \tau_n \rrparenthesis}) + f(t),\qquad G_n(t) \coloneqq G(t,u(t)\1_{\llbracket 0, \tau_n \rrparenthesis}) + g(t),\\
        H_n(t) &\coloneqq H(t,u(t-)\1_{\llbracket 0, \tau_n \rrparenthesis}(t),z) + h(t,z).
    \end{align}
    Consider on $[0, T]$ the linear SPDE
    \begin{align}
    \label{eq: aux lin spde}
         dv + \bar A_n(t)v\,dt &= \bar F_n(t) \,dt + \bigl(\bar B_n(t)v + G_n(t) \bigr) \,dW(t) \\
         &\quad+\int_Z \big( \bar C_n(t,z) + H_n(t,z) \big) \,\Nt(dz,dt)\\
         v(0)&=u_0.
    \end{align}
    Due to Assumption~\ref{assumption operators new} and the boundedness of $u \1_{\llbracket 0, \tau_n \rrparenthesis}$ we see that, for each $n\geq 1$, the operators $(\bar A_n,\bar B_n,\bar C_n)$ satisfy Assumption~\ref{assumption linear operators}. Moreover, using \eqref{eq:bH embedding} and a calculation similar to Lemma~\ref{lemma truncation 1}, Assumption~\ref{assumption operators new} together with the boundedness of $u\1_{\llbracket 0, \tau_n \rrparenthesis}$  ensure that $(F_n,G_n,H_n)$ satisfy the integrability conditions of Theorem~\ref{eq: linear spde L2 theory} for each $n\geq 1$. Hence, for each $n\geq 1$, there exists a unique solution $v_n$ to \eqref{eq: aux lin spde} on $\llbracket 0, T \rrbracket$ satisfying
    \begin{align}
    \label{eq: sup estimate v_n}
        \E\sup_{s\in [0,T]} \|v_n(s)\|_\cH^2 + \E\int_0^T \|v_n(s)\|_\cV^2\,ds <\infty.
    \end{align}
    Moreover, by uniqueness $u=v_n$ a.s.\ on $\llbracket 0,\tau_n\rrparenthesis$.
    Hence, since $v_n$ is \cl on $\llbracket 0,T\rrbracket$, it follows that a.s.\ on $\cW_{\sup}$,
    \begin{align}
    \label{eq: limit at tau_n = sigma}
        \lim_{t\uparrow\sigma} u(t)\1_{\{\tau_n = \sigma<T\}} = \lim_{t\uparrow\sigma} v_n(t)\1_{\{\tau_n = \sigma<T\}} \quad\text{in $\cH$.}
    \end{align}
    Hence, by~\eqref{eq: stopping time limit on sigma<T} and~\eqref{eq: limit at tau_n = sigma} we have
    \begin{align*}
    \P(\cW_{\sup}) & = \lim_{n\to \infty} \P\Big(\cW_{\sup}\cap \{\tau_n=\sigma<T\}\Big) \\ & \leq \lim_{n\to \infty} \P\Big(\lim_{t\uparrow \sigma} u(t) \ \text{exists in $\cH$}, \  \|u\|_{L^2([0,\sigma],\cV)}<\infty, \ \tau_n=\sigma<T\Big) \\ &\leq  \P(\cW_{\lim}) = 0,
    \end{align*}
     where in the last step we used Proposition~\ref{prop: P W_lim = 0}.
\end{proof}

\section{Global well-posedness under coercivity conditions}\label{sec:global}

\subsection{Existence and uniqueness results}
Under a coercivity condition on $(A,B,C)$ we prove global existence and uniqueness.
\begin{theorem}[Global well-posedness]\label{thm:varglobal}
Suppose that Assumption~\ref{assumption operators new} holds.
Let $u_0\in L^0(\Omega,\cH)$ be an $\cF_0$-measurable initial value and
let
    \begin{align}
        f_i\in L^0(\Omega,L^{p_{i}}([0,&T],\cV_{\theta_{i}})),\;\; g\in L^0(\Omega,L^2([0,T],\cL_2(U,\cH))), \\
        h&\in L^0(\Omega,L^2([0,T],L^2(Z,\cH;\nu))),
    \end{align}
be $\cP$-measurable, $\cP$-measurable, and $\cP^-$-measurable, respectively, for all $T>0$, where the pairs $(p_{i},\theta_{i})$ are admissible (see Definition~\ref{def:admissible}).
Put $f = \sum_{i=1}^{m_f} f_i$.
Suppose that for all $T>0$ there exist $\kappa,\eta>0$ such that a.s.\ for all $v\in \cV$ and almost every $t\in [0,T]$,
            \begin{equation}
            \label{nonlinear coercivity condition}
            \begin{aligned}
            \langle A(t,v), v \rangle -(\tfrac{1}{2}+\eta)\|B(t,v)&\|_{\calL_2(U,\cH)}^2 - (\tfrac{1}{2}+\eta) \|C(t,v,\cdot)\|_{L^2(Z,\cH;\nu)} \\ & \geq \kappa\|v\|_\cV^2 - \phi(t) \|v\|_\cH^2 - \psi(t) - \sum_{i=1}^{\wt{m}} \psi_i(t) \| v \|_{1-\wt{\theta}_i},
            \end{aligned}
        \end{equation}
        where $\phi\in L^1(0,T)$ is positive, $\psi\in L^0(\Omega,L^1([0,T]))$,  $\psi_i\in L^0(\Omega,L^{\wt{p}_i}([0,T]))$ with admissible pairs $(\wt{p}_{i},\wt{\theta}_{i})$.

Then there exists a unique global solution $u\in L^2_{\rm loc}([0,\infty),\cV)\cap D([0,\infty),\cH)$
to \eqref{eq:spde}. Moreover, for every $T>0$ there is a constant $C_{T}$ such that
\begin{align*}
\E \sup_{t\in [0,T]}\|u(t)\|_{\cH}^{2} + \E\int_0^T \|u(t)\|^2_{\cV} \, dt
& \leq C_T \Bigl( \E \| u_0 \|_\cH^2 + \E \| \psi \|_{L^1([0,T])} + \sum_{i=1}^{\wt{m}}\E \| \psi_i \|_{L^{\wt{p}_i}([0,T])}^2 \Bigr).
\end{align*}
\end{theorem}
\begin{proof}
Fix $T>0$. Let $(u, \tau)$ be the maximal solution of Theorem~\ref{thm well-posedness unbounded} on $[0,T]$. The idea will be to use the blowup criterion of Theorem~\ref{thm: blowup}. By Lemma~\ref{lemma event localization} we may assume that the data $(f, g, h)$ is square integrable and $u_0$ is uniformly bounded. By the same type of argument as in Lemma~\ref{lemma event localization} we may suppose that $\psi\in L^1(\Omega,L^1([0,T]))$ and $\psi_i \in L^2(\Omega,L^{\wt{p}_i}([0,T]))$.

Let $(\tau_n')_{n\geq 1}$ be a localizing sequence for $(u,\tau)$ and let
\[\tau_n  = \inf\bigl\{t\in [0,\tau_n']: \|u\|_{L^2([0,t],\cV)} + \sup_{s\in[0,t]}\|u(s)-u_0\|_{\cH}\geq n\bigr\}\wedge \tau_n'.\]
Then $u$ is a solution to \eqref{eq:spde} on $[0,\tau_n]$.
Now the plan is to use the coercivity assumption \eqref{nonlinear coercivity condition} and Proposition~\ref{prop: nonlin apriori} to find an a-priori estimate.

Let $\wt{u}\in L^2(\Omega,D([0,T],\cH))\cap L^2(\Omega,L^2([0,T],\cV))$ be the solution to the linear problem
\begin{equation*}
    \begin{split}
        dv(t)+\wt{A}_0(t)v(t)\,dt &= \wt{f}(t) dt + \wt{g}(t)\,dW(t) + \int_{Z} \wt{h}(t,z)\,\Nt(dz,dt)\\
        v(0)&=u_0,
    \end{split}
\end{equation*}
which is well-posed by the linear result of Theorem~\ref{thm L2 linear} applied with
\begin{align*}
\wt{A}_0(t)&= A_0(t, u(t)\one_{[0,\tau_n]}(t)), \ \ &\wt{f}(t) &= \one_{[0,\tau_n]}(t)[F(t,u(t)) + f(t)],
\\ \wt{g}(t) &= \one_{[0,\tau_n]}(t)B(t,u(t)),
& \wt{h}(t,z) &= \one_{[0,\tau_n]}(t)C(t,u(t-),z).
\end{align*}
Here we used \eqref{coercivity condition} and the definition of $\tau_n$ to ensure the coercivity condition. By uniqueness, $\wt{u} = u$ on $[0,\tau_n]$.

Next, we apply the coercivity condition for the nonlinear equation via Proposition~\ref{prop: nonlin apriori}. Let
\begin{align*}
\wt{A}(t,v) &= \one_{[0,\tau_n]}(t) A(t,v) + \one_{(\tau_n,T]}(t)A_0(t,0)v,
\\ \wt{B}(t,v) &= \one_{[0,\tau_n]}(t)B(t,v), \ \ \wt{C}(t,v,z) = \one_{[0,\tau_n]}(t)C(t,v,z), \  \ v\in \cV, z\in Z.
\end{align*}
Then from the coercivity assumption \eqref{nonlinear coercivity condition} we see that a.s.\ for all $v\in \cV$ and almost every $t\in [0,T]$,
\begin{align*}
\langle \wt{A}(t,v), v \rangle -(\tfrac{1}{2}+\eta)\|\wt{B}(t,v)&\|_{\calL_2(U,\cH)}^2 - (\tfrac{1}{2}+\eta) \|\wt{C}(t,v,\cdot)\|_{L^2(Z,\cH;\nu)}^2 \\ & \geq \wt{\kappa}\|v\|_\cV^2 - \wt{\phi}(t) \|v\|_\cH^2 - \psi(t) - \sum_{i=1}^{\wt{m}} \psi_i(t) \| v \|_{1-\wt{\theta}_i},
\end{align*}
where $\wt{\kappa} = \min\{\kappa, \kappa_0\}$ and $\wt{\phi} = \min\{\phi, \phi_0\}$ (recall that the constants $\kappa_0$ and $\phi_0$ stem from~\eqref{coercivity condition} applied with $n=0$). By construction, the process $\wt{u}$ satisfies
\begin{equation*}
    \begin{split}
        d\wt{u}(t)+\wt{A}(t,\wt{u}(t))\,dt &= \wt{B}(t,\wt{u}(t)) \,dW(t) + \int_{Z} \wt{C}(t,\wt{u}(t-),z) \,\Nt(dz,dt)\\
        \wt{u}(0)&=u_0
    \end{split}
\end{equation*}
on $[0,T]$. Therefore, by Proposition~\ref{prop: nonlin apriori} with \begin{align*}
A_1(t,v) &= \one_{[0,\tau_n]}(t)A_L(t,v)v + \one_{(\tau_n,T]}(t)A_L(t,0)v,
\\ A_2(t,v) &= \one_{[0,\tau_n]}(t)A_S(t,v)v +  \one_{(\tau_n,T]}(t)A_S(t,0)v,
\end{align*}
$A_3(t,v) = \one_{[0,\tau_n]}(t)F(t,v)$, and $A_{i+3} = f_i$ for $i\in \{1, \ldots, m_f\}$ (recalling the localization) we obtain
\begin{align*}
        \E \sup_{t \in [0,T]} \| \wt{u}(t) \|_\cH^2
        +\E\int_{0}^{T}\|\wt{u}(t)\|_\cV^2\,dt
        \leq{}  C \Bigl( \E \| u_0\|_\cH^2 + \E \| \psi \|_{L^1([0,T])} + \sum_{i=1}^{\wt{m}}\E \| \psi_i \|_{L^{\wt{p}_i}([0,T])}^2 \Bigr) ,
    \end{align*}
Since $\wt{u} = u$ on $[0,\tau_n]$, after letting $n\to \infty$ and applying Fatou's lemma it follows that
\begin{align}
        \E \sup_{t \in [0,\tau)} \| u(t) \|_\cH^2
        +\E\int_{0}^{\tau}\|u(t)\|_\cV^2\,dt
        \leq{}  C \Bigl( \E \| u_0\|_\cH^2 + \E \| \psi \|_{L^1([0,T])} + \sum_{i=1}^{\wt{m}}\E \| \psi_i \|_{L^{\wt{p}_i}([0,T])}^2 \Bigr).
    \end{align}
Therefore, by Theorem~\ref{thm: blowup} we can conclude that
\begin{align*}
\P(\tau<T) =   P\Bigl( \tau<T,\quad \sup_{t\in [0,\tau)} \|u(t)\|_\cH + \int_0^\tau \|u(t)\|_\cV^2\,dt<\infty \Bigr) &= 0.
\end{align*}
This implies $\tau = T$ a.s. Since $T$ was arbitrary, this completes the global existence proof by uniqueness. Moreover, the a priori estimate follows as well.

\end{proof}

\begin{remark}
It is possible to take $\eta=0$ in \eqref{nonlinear coercivity condition}. For this, one can use the stochastic Gronwall lemma, and the a-priori estimate needs to be replaced by a different estimate. For details the reader is referred to \cite[Theorem 3.5]{agresti2022critical}.
\end{remark}

As a first simplification, we state a version without inhomogeneities. In this way, also the coercivity condition is simplified.

 \begin{corollary}[Global well-posedness]\label{cor:varglobal}
Suppose that Assumption~\ref{assumption operators new} holds. Let $u_0\in L^0(\Omega,\cH)$ be an $\cF_0$-measurable initial value and let
$f = 0$, $g = 0$ and $h=0$. Suppose that for all $T>0$ there exist $\kappa,\eta>0$ such that a.s.\ for all $v\in \cV$ and almost every $t\in [0,T]$,
            \begin{equation}
            \label{nonlinear coercivity conditionsimple}
            \begin{aligned}
            \langle A(t,v), v \rangle &-(\tfrac{1}{2}+\eta)\|B(t,v)\|_{\calL_2(U,\cH)}^2 \\
            &- (\tfrac{1}{2}+\eta) \|C(t,v,\cdot)\|_{L^2(Z,\cH;\nu)} 
            \geq \kappa\|v\|_\cV^2 - \phi(t) \|v\|_\cH^2 - \psi(t),
            \end{aligned}
        \end{equation}
        where $\phi\in L^1(0,T)$, $\psi\in L^0(\Omega,L^1([0,T]))$.

Then there exists a unique global solution $u\in L^2_{\rm loc}([0,\infty),\cV)\cap D([0,\infty),\cH)$ a.s.\ to \eqref{eq:spde}. Moreover, there is a constant $C_{T}$ such that
\begin{align*}
\E \sup_{t\in [0,T]}\|u(t)\|_{\cH}^{2} + \E\int_0^T \|u(t)\|^2_{\cV} \, dt
& \leq C_T \left( \E \| u_0 \|_\cH^2 + \E \| \psi \|_{L^1([0,T])} \right).
\end{align*}
\end{corollary}

Moreover, in the special case that $G$ and $H$ have linear growth, we can further simplify the formulation and omit $\eta$ in the coercivity condition. The proof is similar to Lemma~\ref{lem:lin coercive}.
 \begin{corollary}[Global well-posedness]\label{cor:varglobal2}
Suppose that Assumption~\ref{assumption operators new} holds and there is a constant $K\geq 0$ and $\beta\in [0,1)$ such that for all $v\in \cV$ 
\begin{align}
\|G(t,v)\|_{\calL_2(U, \cH)} + \|H(t,v,\cdot)\|_{L^2(Z;\cH;\nu)} \leq K (\|v\|_{\beta}+1).
\end{align}

Let $u_0\in L^0(\Omega,\cH)$ be an $\cF_0$-measurable initial value and
assume that $f = 0$, $g = 0$ and $h=0$. Suppose that for all $T>0$ there exist $\kappa,\eta>0$ such that a.s.\ for all $v\in \cV$ and almost every $t\in [0,T]$,
            \begin{equation}
            \label{nonlinear coercivity conditionsimple2}
            \begin{aligned}
            \langle A_0(t,v), v \rangle -\tfrac{1}{2}\|B_0(t,v)&\|_{\calL_2(U,\cH)}^2 - \tfrac{1}{2} \|C_0(t,v,\cdot)\|_{L^2(Z,\cH;\nu)}^2 \geq \kappa\|v\|_\cV^2 - \phi(t) \|v\|_\cH^2 - \psi(t),
            \end{aligned}
        \end{equation}
and
\begin{equation}
            \label{nonlinear coercivity conditionsimple3}
            \begin{aligned}
            \langle F(t, v), v \rangle  \leq  \phi(t) \|v\|_\cH^2 + \psi(t),
            \end{aligned}
        \end{equation}
where $\phi\in L^1(0,T)$, $\psi\in L^0(\Omega,L^1([0,T]))$.
Then there exists a unique global solution $u\in L^2_{\rm loc}([0,\infty),\cV)\cap D([0,\infty),\cH)$ a.s.\ to \eqref{eq:spde}. Moreover, there is a constant $C_{T}$ such that
\begin{align*}
\E \sup_{t\in [0,T]}\|u(t)\|_{\cH}^{2} + \E\int_0^T \|u(t)\|^2_{\cV} \, dt
& \leq C_T \left( \E \| u_0 \|_\cH^2 + \E \| \psi \|_{L^1([0,T])} \right).
\end{align*}
\end{corollary}

\subsection{Continuous dependence on the initial data}

Next we prove continuous dependence on the initial data.
\begin{theorem}[Continuous dependence on the initial data]\label{thm:contdep}
Suppose that the conditions of Theorem~\ref{thm:varglobal}  hold and let $u$ denote the global solution to \eqref{eq:spde} provided there.
For each $n\geq 1$ let $u^n$  denote the unique global solution to \eqref{eq:spde} with strongly $\cF_0$-measurable initial data $u_{0}^n$.
Suppose that $\|u_0-u_0^n\|_\cH\to 0$ in probability, then for every $T\in (0,\infty)$,
\[\|u - u^n\|_{L^2([0,T],\cV)} + \|u - u^n\|_{D([0,T],\cH)}\to 0 \ \ \text{in probability as $n\to \infty$}.\]
If additionally $\sup_{n\geq 1}\E \|u_0^n\|^2 <\infty$, then for any $q\in (0,2)$
\[\|u - u^n\|_{L^q(\Omega,L^2([0,T],\cV))} + \|u - u^n\|_{L^q(\Omega,D([0,T],\cH))}\to 0 \ \ \text{as $n\to \infty$}.\]
\end{theorem}

The following tail estimate is the key to the proof of the continuous dependency. A key part in the proof is the usage of a stochastic Gronwall lemma.
\begin{proposition}\label{prop:estimateLpcont}
Suppose that the conditions of Theorem~\ref{thm:varglobal}  hold.
Let $u$ and $v$ denote the solution to \eqref{eq:spde} with initial values $u_0$ and $v_0$ with the property that there is an $r>0$ such that a.s. $\|u_0\|_{\cH}+\|v_0\|_{\cH}\leq r$. Then for every $T>0$ there exist $\chi_1,\chi_2:[r,\infty)\to (0,\infty)$, both independent of $u_0$ and $v_0$ and such that $\lim_{R\to \infty}\chi_2(R) = 0$, such that for all $\varepsilon>0$ and $R\geq r$ one has
\begin{align*}
\P(\|u-v\|_{\cM(T)}\geq \varepsilon) \leq \varepsilon^{-2}\chi_1(R) \E\|u_0 - v_0\|_{\cH}^2 + \chi_2(R)(1+\E\|u_0\|^2_{\cH} + \E\|v_0\|_{\cH}^2),
\end{align*}
where $\cM(T)=D([0,T],\cH)\cap L^2([0,T],\cV)$ as in Section~\ref{sec: local well-posedness}.
\end{proposition}

\begin{proof}
Let $w = u-v$. Then $w$ is the unique solution to
\begin{align*}
       \left\{\begin{array}{ll}
        dw + \wt{A}_0(t) w(t) dt &= f_w(t) dt + (\wt{B}_0(t) w(t)  + g_w(t)) dW(t) \\
        &\qquad\quad+ \int_Z (\wt C_0(t,z)v(t-) + h_{w}(t,z)) \,\wt{N}(dz,dt),
\\        w(0) &= u_0-v_0,
\end{array}\right.
\end{align*}
where $\wt{A}_0(t) w = A_L(t,u(t))w + A_S(t,u(t))w$, $\wt{B}_0(t) w = B_0(t,u(t))w$,\newline
$\wt{C}_0(t,z) w = C_0(t,u(t-),z)w$, and where $f_w = f_1+f_2+f_3$, $g_w = g_1+g_2$, $h_w = h_1+h_2$ are given by
\begin{align*}
f_1 & = ((A_L(\cdot,u) - A_L(\cdot,v))v, & f_2 & =  (A_S(\cdot,u) - A_S(\cdot,v))v,  \\ & & f_3 &= F(\cdot, u) - F(\cdot, v),
\\ g_1 &= (B_0(\cdot, u) - B_0(\cdot, v))v, & g_2  &= G(\cdot, u) - G(\cdot, v),
\\ h_1 &= (C_0(\cdot, u(\cdot\,-),z) - C_0(\cdot, v(\cdot\,-),z))v(\cdot\,-), & h_2  &= H(\cdot, u(\cdot\,-),z) - H(\cdot, v(\cdot\,-),z).
\end{align*}
In order to derive an a-priori estimate for $w$, we will apply coercivity for $(\wt{A}_0,\wt{B}_0,\wt{C}_0)$ combined with a stochastic Gronwall argument. In order to check \eqref{eq: coercivity lin wellpos} we will use \eqref{coercivity condition} and a stopping time argument to ensure $\|u\|_{\cH}+\|v\|_{\cH}\leq R$, where $R\geq r$. Let
\[\tau_R:=\inf\big\{t\in[0,T]\,:\, \|u(t)\|_{\cH}+\|v(t)\|_{\cH}\geq R\big\}\wedge T.\]
Then $\{\tau_{R}=T\}=\{\sup_{t\in [0,T)}\|u(t)\|_{\cH} + \|v(t)\|_{\cH}< R\}$. Moreover, by Markov's inequality in conjunction with Theorem~\ref{thm:varglobal},
\[\P(\tau_R<T) \leq \P(\sup_{t\in [0,T)}(\|u(t)\|_{\cH} + \|v(t)\|_{\cH})\geq R)\leq \frac{C_T}{R} (1+\E\|u_0\|_{\cH}^2 + \E\|v_0\|_{\cH}^2),\]
where $C_T$ depends on the data, but not on the initial values.
Thus by the global estimate of Theorem~\ref{thm:varglobal} deduce
 \begin{align}\label{eq:PtailwM}
\P(\|w\|_{\cM(T)}\geq \varepsilon)
& = \P(\|w\|_{\cM(\tau_R)}\geq \varepsilon, \tau_R=T) + \P(\tau_R<T)
\\ & \leq \P \big(\|w\|_{\cM(\tau_R)}\geq \varepsilon\big)+ \frac{C_T}{R} (1+\E\|u_0\|_{\cH}^2 + \E\|v_0\|_{\cH}^2).
\end{align}

It remains to estimate $\P \big(\|w\|_{\cM(\tau_R)}\geq \varepsilon\big)$ via coercivity and a stochastic Gronwall lemma.
Let $\wt{w}$ be the solution to
\begin{align}\label{eq:diffeqcont}
       \left\{\begin{array}{ll}
        d\wt{w} + \wt{A}_{0,R}(t)\wt{w}(t) \, dt &= \one_{[0,\tau_R]}f_w(t) \, dt + (\wt{B}_{0,R}(t) \wt{w}(t)  + \one_{[0,\tau_R]}g_w(t)) \, dW(t) \\
        &\qquad\qquad+ \int_Z (\wt{C}_{0,R}(t) \wt{w}(t-) + \one_{[0,\tau_R]}h_{w}(t,z)) \,\wt{N}(dz,dt),
\\        \wt{w}(0) &= u_0-v_0.
\end{array}\right.
\end{align}
Here, after setting $u_R(t) \coloneqq \one_{[0,\tau_R]}(t) u(t)$, we use the notation
\begin{align}
    \wt{A}_{0,R}(t)y &= A_L(t,u_R(t))y + \one_{[0,\tau_R]}(t)A_S(t,u_R(t))y \\
    \wt{B}_{0,R}(t)y &= \one_{[0,\tau_R]}(t)B_0(t,u_R(t))y, \\ \wt{C}_{0,R}(t,z) y &= \one_{[0,\tau_R]}(t)C_0(t,u_R(t-),z)y
\end{align}
when $y\in \cV$.
Note that $\wt{w} = w$ on $\llbracket 0, \tau_R \rrbracket$.
Next, by It\^o's formula (see Corollary~\ref{cor: ito}) and the coercivity condition \eqref{coercivity condition} we find that
\begin{align*}
\|\wt{w}(t)\|_{\cH}^2 + 2\kappa_R \|\wt{w}\|_{L^2([0,t],\cV)}^2
\leq & \|u_0-v_0\|_{\cH}^2 +  D(t) + M(t),
\end{align*}
where the deterministic term is given by
\begin{align*}
D(t) =& 2\int_0^t \lb \wt{w}(s), \one_{[0,\tau_R]}(s)f_w(s) \rb + \phi_R(s) \|\wt{w}(s)\|_{\cH}^2 \, ds
\\ & + \int_0^t 2\one_{[0,\tau_R]}(s)(\wt B_{0,R} \wt{w}, g_w(s))_{\calL(U,\cH)} + \one_{[0,\tau_R]}(s) \|g_w(s)\|^2_{\calL_2(U,\cH)} \, ds
\\ & + \int_0^t \int_Z 2(\one_{[0,\tau_R]}(s)h_w(s,z),\wt C_{0,R}(s,z)\wt{w}(s))_\cH  +\one_{[0,\tau_R]}(s)\|h_w(s,z)\|^2  \, \nu(dz) ds,
\end{align*}
and the stochastic term $M(t)$ is given by
\begin{align*}
M(t) &= 2\int_0^t \lb \wt{w}(s), B_R(s)\wt{w}(s)+\one_{[0,\tau_R]}(s)g_w(s) \rb \, d W(s)
\\ & + 2\int_0^t\int_Z  \lb \wt{w}(s-), C_{0,R}(s,z)\wt{w}(s-)+ \one_{[0,\tau_R]}(s)h_w(s,z) \rb \, \wt{N}(dz,ds)
\\ & + \int_0^t\int_Z \big\|\wt{C}_{0,R}(s,z)\wt{w}(s-) + \one_{[0,\tau_R]}(s)h_w(s,z)\big\|_{\cH}^2  \, \wt{N}(dz, ds)
\end{align*}
Observe that $M$ is a local martingale.

Next we estimate the Lebesgue integrals appearing in the It\^o formula pointwise in $\Omega$. On $[0,\tau_R)$, taking $\wt{w} = w$ into account, we have
\begin{align*}
\|f_1\|_{\cV^*} &= \|(A_L(\cdot, u) - A_L(\cdot, v))v\|_{\cV^*} \leq C_{R,T} \|\wt{w}\|_{\cH}  \|v\|_{\cV},
\\
\|f_2\|_{\alpha_A} &= \|(A_S(\cdot, u) - A_S(\cdot, v))v\|_{\alpha_A}\leq K_{A,R,T}\|\wt{w}\|_{\cH} \|v\|_{\beta_A},
\\ \|f_3\|_{\alpha_F} &=  \|F(\cdot,u) - F(\cdot, v)\|_{\alpha_F} \leq C_{R,T} (1+\|u\|_{\beta_F}^{\rho_F} + \|v\|_{\beta_F}^{\rho_F})\|\wt{w}\|_{\beta_F}.
\end{align*}
Since $|\lb \wt{w}, f_w \rb|\leq  \sum_{i=1}^3 |\lb \wt{w}, f_i \rb|$ we estimate these terms individually. Let $\delta>0$ be arbitrary. For $f_1$ on $[0,\tau_R)$ we have
\begin{align*}
|\lb \wt{w}, f_1 \rb| & \leq \|\wt{w}\|_{\cV} \|f_1\|_{\cV^*} \leq C_{R,T} \|\wt{w}\|_{\cV}    \|\wt{w}\|_{\cH}  \|v\|_{\cV}\leq \delta \|\wt{w}\|_{\cV}^2 + C_{\delta,R,T}  \|\wt{w}\|_{\cH}^2  \|v\|_{\cV}^2.
\end{align*}
Using interpolation estimates and the boundedness of $\|v\|_{\cH}$ on $[0,\tau_R)$ we obtain
\begin{align*}
|\lb \wt{w}, f_2 \rb| &  \leq \|\wt{w}\|_{1-\alpha_A} \|f_2\|_{\alpha_A}
\\ & \leq K_{A,R,T}\|\wt{w}\|_{1-\alpha_A}  \|\wt{w}\|_{\cH} \|v\|_{\beta_A}
\\ & \leq K_{A,R,T}\|\wt{w}\|_{\cH}^{2\alpha_A+1} \|\wt{w}\|_{\cV}^{1-2\alpha_A} \|v\|_{\beta_A}
\\ & \leq \delta \|\wt{w}\|_{\cV}^{2}   + C_{\delta,R} K_v \|\wt{w}\|_{\cH}^{2}
\end{align*}
where  $K_v(t) = K_{A,R,T}(t)^{\nicefrac{2}{1+2\alpha_A}} \|v\|_{\beta_A}^{{\nicefrac{2}{1+2\alpha_A}}}$. Combining the conditions on $(\alpha_A,\beta_A)$ and $v\in \cM(0,T)$ with the first estimate in \eqref{eq:bH embedding}, one can check that
\[\|K_v\|_{L^1([0,\tau_R])} \leq \|K_{A,R,T}(t)\|_{L^{r_A}([0,T])}^{\nicefrac{2}{1+2\alpha_A}}  \|v\|_{L^{\nicefrac{2}{2\beta_A-1}}([0,\tau_R],\cV_\beta)}^{\nicefrac{2}{1+2\alpha_A}} \leq C_R \|v\|_{L^2([0,T],\cV)}^{\gamma_A}\]
for some constant $C_R$ and where $\gamma_A =\nicefrac{2(2\beta_A-1)}{1+2\alpha_A}\leq 2$.
For $f_3$ we find in a similar way
\begin{align*}
|\lb \wt{w}, f_3 \rb| &  \leq \|\wt{w}\|_{{1-\alpha_F}} \|f_3\|_{\alpha_F}
\\ & \leq C_{R,T}\|\wt{w}\|_{1-\alpha_F} (1+\|u\|_{\beta_F}^{\rho_F} + \|v\|_{\beta_F}^{\rho_F})\|\wt{w}\|_{\beta_F}
\\ & \leq C_{R,T}\|\wt{w}\|_{\cH}^{2+2\alpha_F-2\beta_F} \|\wt{w}\|_{\cV}^{2\beta_F-2\alpha_F} (1+\|u\|_{\beta_F}^{\rho_F} + \|v\|_{\beta_F}^{\rho_F})
\\ & \leq \delta \|\wt{w}\|_{\cV}^{2}  +  C_{\delta,R,T}\|\wt{w}\|_{\cH}^2 K_{u,v},
\end{align*}
where $K_{u,v}(t) = (1+\|u\|_{\beta_F}^{\rho_F} + \|v\|_{\beta_F}^{\rho_F})^{\nicefrac{1}{1+\alpha_F-\beta_F}}$.
Combining the subcriticality condition \eqref{eq: criticality F} with  the first estimate in \eqref{eq:bH embedding}, one can check that $\|K_{u,v}\|_{L^1(0,\tau_R)}\leq C_R (1+\|u\|_{L^2([0,T],\cV)}^2+ \|v\|_{L^2([0,T],\cV)}^2)$ for some constant $C_R$.
The terms $\|g_w\|^2_{\calL_2(U,\cH)}$ and $\|h_w\|^2_{\calL_2(U,\cH)}$  can be estimated similarly. The mixed terms can be reduced to the previous ones by arguing as in Lemma~\ref{lem:lin coercive}.

The other deterministic terms can be estimated similarly and in total we obtain that for all $\delta>0$ there exists an $K_R\in L^1(\Omega,L^1(0,T))$ such that a.s.\ for all  $t\in [0,\tau_R)$
\begin{align*}
D(t) \leq \delta \int_0^t \|\wt{w}(s)\|_{\cV}^{2} \, ds  + \int_0^t K_R(s) \|\wt{w}(s)\|_{\cH}^{2} \, ds,
\end{align*}
where $\|K_R\|_{L^1(0,T)}\leq C_R(1+\|u\|_{L^2([0,T],\cV)}^2+\|v\|_{L^2([0,T],\cV)}^2)$.
Therefore, we can conclude that
\begin{align*}
\|\wt{w}(t)\|_{\cH}^2 + \|\wt{w}\|_{L^2([0,t],\cV)}^2
\leq & \theta_R \|u_0-v_0\|_{\cH}^2 + \int_0^t \theta_R K_R(s) \|\wt{w}(s)\|_{\cH}^{2} \, ds  + M(t),
\end{align*}
where $\theta_R = \max\{1,\kappa_R^{-1}\}$.
Now the stochastic Gronwall inequality \cite[Corollary 5.4]{geiss2024sharp} gives that for all $\varepsilon,\gamma>0$ one has
\begin{align}
\label{eq:tail estimate 1}
\P(\|\wt{w}\|_{\cM(T)}^2>\varepsilon^2) \leq  \frac{\theta_R e^\gamma }{\varepsilon^2} \E(\|u_0-v_0\|_{\cH}^2)  + \P(\theta_R\|K_R\|_{L^1(0,T)}>\gamma).
\end{align}
By Markov's inequality and Theorem~\ref{thm:varglobal} we can estimate
\begin{align*}
\P(\theta_R\|K_R\|_{L^1(0,T)}>\gamma)& \leq \frac{\theta_R\E \|K_R\|_{L^1(0,T)}}{\gamma}
\\ & \leq \frac{C_R(1+\E\|u\|_{L^2([0,T],\cV)}^2+\E\|v\|_{L^2([0,T],\cV)}^2)}{\gamma}
\\ & \leq \frac{C_R'(1+\E\|u_0\|_{\cH}^2+\E\|v\|_{\cH}^2)}{\gamma},
\end{align*}
where $C_R'$ also depends on $C_T$, $\psi$ and $\psi_i$.
Now we plug the last bound back into~\eqref{eq:tail estimate 1} and set $\gamma \coloneqq R C_R'$, $\chi_1(R) \coloneqq \theta_R e^\gamma$ to obtain
\begin{align*}
\P(\|\wt{w}\|_{\cM(T)}^2>\varepsilon^2) \leq  \frac{\chi_1(R)}{\varepsilon^2} \E(\|u_0-v_0\|_{\cH}^2)  + \frac{(1+\E\|u_0\|_{\cH}^2+\E\|v\|_{\cH}^2)}{R}.
\end{align*}
Combining the estimates with \eqref{eq:PtailwM} and using $w = \wt{w}$ on $[0,\tau_R)$ we see that
\begin{align*}
\P(\|w\|_{\cM(0,T)}\geq \varepsilon) \leq \frac{\chi_1(R)}{\varepsilon^2} \E(\|u_0-v_0\|_{\cH}^2)  + (C_T+1)\frac{(1+\E\|u_0\|_{\cH}^2+\E\|v\|_{\cH}^2)}{R}
\end{align*}
\end{proof}

\begin{proof}[Proof of Theorem~\ref{thm:contdep}]
Given Proposition~\ref{prop:estimateLpcont}, we can argue as in~\cite[Theorem 3.8]{agresti2022critical}.
\end{proof}

\section{Applications}
\label{sec:applications}

\noindent
Due to our results,  the applications in~\cite[Section 5]{agresti2022critical} can be generalized to the case of L\'evy noise. In particular, this includes Cahn--Hilliard, tamed Navier--Stokes, Allen--Cahn in the strong setting for $d\in \{1, 2, 3, 4\}$, a quasilinear equation for $d=1$, as well as many more. All of these problems can be considered on bounded or unbounded domains, since we never need the compactness of $\cV\hookrightarrow \cH$.

Below we present several examples which differ in essential points from the ones in~\cite{agresti2022critical}. The first one in Subsection~\ref{ss:reaction}, is on reaction-diffusion equations with nonlinearities which have higher powers than could be covered before. In particular, this enables us to include the Allen--Cahn equation in the weak setting for $d=2$, see Example~\ref{ex:seconddiff2} (which requires $\alpha_F=1/2$).  Moreover, in Example~\ref{ex:beta=1} we also discuss a new type of nonlinearity which we can consider by taking $\beta_F=1$. Another example with singular drift term  is presented in Example~\ref{ex:singular}.

In Subsection~\ref{ss:fluid}, a class of fluid dynamics equations will be presented.
Moreover, in  Subsections~\ref{ss:KS} and \ref{ss:SH} we present the stochastic  Kuramoto--Sivashinsky equation and the Swift--Hohenberg equation, which are both fourth order equations. For the Kuramoto--Sivashinsky equation we will see that $\alpha_F = 1/4$ is needed. We explain how the extra flexibility makes it possible to weaken the conditions on the nonlinearities compared to the known results in the literature.

\subsection{Reaction-diffusion equations in the weak setting}\label{ss:reaction}

In this section, we consider second order equations with Dirichlet boundary conditions in the weak setting in the PDE sense. The setting is similar to \cite[Section 5.3]{agresti2022critical} and \cite{LR15}.  The goal of the section is to show that the extended setting allows to include new nonlinearities with polynomial growth for low dimensions. Most notable is that cubic nonlinearities can also be included for $d=2$. These were excluded in the above two references due to the fact that the embedding $H^{1}(\Dom)\hookrightarrow L^\infty(\Dom)$ fails. We can circumvent this issue due to the flexibility in the condition on $F$. At the same time, the results will be presented for the more general case of L\'evy noise. For simplicity, the results are only presented for bounded domains, but they extend to unbounded domains as well.

Let $\Dom\subseteq \R^d$ be open and bounded.
Let
\[\cV = H^1_0(\Dom), \ \  \ \cH = L^2(\Dom) \ \ \ \text{and} \ \ \cV^* = H^{-1}(\Dom).\]
The equation we consider is of the form
\begin{align}
\label{eq:reaction_diffusion}
\left\{
\begin{aligned}
 du  &= \big[\dv(a\nabla u)  + f(\cdot, u) + \dv(\of(\cdot, u))\big]\, d t
\\  & \qquad \qquad \qquad \qquad
+ \sum_{n\geq 1}  \big[(b_{n}\cdot \nabla) u+ g_{n}(\cdot,u) \big]\, d w^n_t
\\  & \qquad \qquad \qquad \qquad
+ \int_Z\big[(c(\cdot,z)\cdot \nabla) u(\cdot-)+ h(\cdot,u(\cdot-),z) \big]\,  \wt{N}(dz,dt),\ \ &\text{ on }\Dom,\\
u &= 0, \ \ &\text{ on }\partial\Dom,
\\
u(0)&=u_{0},\ \ &\text{ on }\Dom,
\end{aligned}\right.
\end{align}
where $(w^n_t\,:\,t\geq 0)_{n\geq 1}$ are independent standard Brownian motions.

Next, we state our assumptions. We emphasize that, compared to the literature mentioned above, the cases $\rho_1\in (3,4]$ for $d=1$, and $\rho_1=2$ for $d=2$ are new.
\begin{assumption}\label{ass:2ndorder1}
Suppose that
\[
\rho_1 \in
 \Big[0,\frac{4}{d}\Big] \qquad \text{and} \qquad \rho_2,\rho_3\in \Big[0,\frac{2}{d}\Big],
\]
and
\begin{enumerate}[{\rm(1)}]
 \item\label{ass:2ndorder11} $a^{j,k}:\R_+\times \Omega\times \Dom\to \R$ and $b^j:=(b^{j}_{n})_{n\geq 1}:\R_+\times \Omega\times \Dom\to \ell^2$ are $\mathcal{P}\otimes \mathcal{B}(\Dom)$-measurable and uniformly bounded, and $c^j:\Omega\times\bR_+\times\Dom\times Z\to\bR$ is $\cP^-\otimes\mathcal{B}(\Dom)\otimes\cZ$-measurable and uniformly bounded.
\item\label{ass:2ndorder12}
There exists $\kappa>0$ such that a.e.\ on $ \R_+\times \Omega\times \Dom$ and for all $\xi\in \R^d$,
$$
\sum_{j,k=1}^d \Big(a^{j,k}(t,x)-\frac12\sum_{n\geq 1} b^j_{n}(t,x)b^k_{n}(t,x) - \frac{1}{2}\int_Z c^j(t,x,z)  c^k(t,x,z) \nu(dz) \Big)
 \xi_j \xi_k
\geq  \kappa |\xi|^2.
$$
\item\label{it:growth_nonlinearities}
The mappings $f:\R_+\times \Omega\times \Dom\times \R\to \R$, $\of:\R_+\times \Omega\times \Dom\times \R\to \R^d$ and $g:=(g_{n})_{n\geq 1}:\R_+\times \Omega\times \Dom\times \R\to \ell^2$
are $\mathcal{P}\otimes \mathcal{B}(\Dom)\otimes \mathcal{B}(\R)$-measurable, and the mapping $h:\Omega\times\bR_+\times\Dom\times \R\times Z\to\bR$ is $\cP^{-}\otimes \mathcal{B}(\Dom)\otimes \mathcal{B}(\R)\otimes\cZ$-measurable.
Moreover, there is a constant $C$ such that for $y, y' \in\R$ a.e.\ on $\Omega\times\R_+\times \Dom$,
\begin{align*}
 |f(\cdot,y)-f(\cdot,y')|
& \leq C(1+|y|^{\rho_1}+|y'|^{\rho_1})|y-y'|,
\\ |f(\cdot,y)|&\leq C(1+|y|^{\rho_1+1}),
\\ |\of(\cdot,y)-\of(\cdot,y')|
& \leq C(1+|y|^{\rho_2}+|y'|^{\rho_2})|y-y'|,
\\ |\of(\cdot,y)|& \leq C (1+|y|^{\rho_2+1}),
\\ \|g(\cdot,y)-g(\cdot,y')\|_{\ell^2} + \|h(\cdot,y,\cdot)-h(\cdot,y',\cdot)\|_{L^2(Z;\nu)}
&\leq C(1+|y|^{\rho_3}+|y'|^{\rho_3})|y-y'|,
\\ \|g(\cdot,y)\|_{\ell^2}+ \|h(\cdot,y,\cdot)\|_{L^2(Z;\nu)}  &\leq C(1+|y|^{\rho_3+1}). \ \
\end{align*}
\item\label{it:secondcoercity} There exist $M,C\geq 0$ and $\eta>0$ such that a.e.\ in $\Omega\times\R_+$ for all $u\in C^\infty_c(\Dom)$
\begin{align*}
&(a \nabla u, \nabla u)_{L^2(\Dom)} + (\of(\cdot, u), \nabla u)_{L^2(\Dom)} -  (f(\cdot, u),u)_{L^2(\Dom)}  \\
 & -(\tfrac12+\eta)  \sum_{n\geq 1} \|(b_n \cdot \nabla) u + g_n(\cdot, u)\|_{L^2(\Dom)}^2
\\
 & -(\tfrac12+\eta)  \int_Z \|(c \cdot \nabla) u + h(\cdot, u,z)\|_{L^2(\Dom)}^2 \nu(dz)
\geq  \kappa \|\nabla u\|^2_{L^2(\Dom)} - M\|u\|^2_{L^2(\Dom)} -C.
\end{align*}
\end{enumerate}
\end{assumption}

If $\overline{f}$ does not depend on the $x$-variable, then one always has $(\of(\cdot, u), \nabla u)_{L^2(\Dom)} = 0$ by the divergence theorem (see \cite[Lemma 5.12]{agresti2022critical}).

In order to formulate \eqref{eq:reaction_diffusion} as \eqref{eq:spde} we set $U = \ell^2$, $\cH = L^2(\Dom)$, $\cV = H^1_0(\Dom)$ and $\cV^* = H^{-1}(\Dom)$. Then $\cV_{\beta} \hookrightarrow  L^r(\Dom)$ for all $\beta\in (1/2,1)$ and $r\in (2, \infty)$ such that $-\frac{d}{r}\leq 2\beta-1-\frac{d}{2}$ (cf.\ \cite[Lemma A.7]{AV20_NS}).

Let $A_0:\R_+\times\Omega\to \calL(\cV,\cV^*)$, $B_0:\R_+\times \Omega\to \calL(\cV, \calL_2(U,\cH))$  and $C_0:\R_+\times \Omega\to \calL(\cV, L^2(Z,\cH;\nu))$  be given by
\begin{align*}
 A_0(t) u &= -\dv(a(t,\cdot) \nabla u), & (B_0(t) u)_n & = (b_{n}(t,\cdot)  \cdot \nabla) u,
 &  (C_0(t) u) & = (c(t,\cdot,z)  \cdot \nabla) u.
\end{align*}
Recall that the non-linearity $F$ can consist of several parts, see Remark~\ref{rem: ass nonlinear operators}~\eqref{it:sum assump}.
Let $F = F_1+F_2$,
where $F_1,F_2:\R_+\times\Omega\times \cV\to \cV^*$, $G:\R_+\times\Omega\times \cV\to \calL_2(U,\cH)$ and $H:\R_+\times\Omega\times \cV\to L^2(Z,\cH;\nu)$ be given pathwise by
\begin{align*}
F_1(t,u) &= f(t,\cdot,u(\cdot)), \qquad F_2(t,u) = \dv [\of(t,\cdot,u(\cdot))], \\  (G(t,u))_n &= g_n(t,\cdot,u(\cdot)), \qquad H(t,u)(z) = h(t,\cdot,u(\cdot),z).
\end{align*}
A process $u$ is called solution to \eqref{eq:reaction_diffusion} if
$u$ is a solution to \eqref{eq:spde} with the above definitions.

It remains to check the conditions of Theorem~\ref{thm:varglobal}. The coercivity condition \eqref{nonlinear coercivity condition} holds with $\phi$ and $\psi$ constant and $\psi_i = 0$. We check the local Lipschitz condition for $F$. The one for $F_1$ and $d\geq 3$, was already checked in~\cite[Section 5.3]{agresti2022critical} (where $\alpha_F=0$). For $d=1$ and $\rho_1\leq 3$, or $d=2$ and $\rho_1<2$, the latter result also applies. The same holds for $F_2$ and $G$ for any $d\geq 1$, and the condition for $H$ can be checked in a similar way. Hence, from now on we may assume $d\in \{1, 2\}$ and $\rho_1\geq 2$. Then we have
\begin{align*}
\|&F_1(t,u) - F_1(t,v)\|_\frac{1}{2} \lesssim \|(1+|u|^{\rho_1}+|v|^{\rho_1}) |u-v|\|_{L^{2}(\Dom)} & \text{(by Assumption~\ref{ass:2ndorder1})}
\\ & \lesssim (1+\|u\|^{\rho_1}_{L^{2(\rho_1+1)}(\Dom)}+\|v\|^{\rho_1}_{L^{2(\rho_1+1)}(\Dom)}) \|u-v\|_{L^{2(\rho_1+1)}(\Dom)} & \text{(by H\"older's inequality)}
\\ & \lesssim (1+\|u\|^{\rho_1}_{\beta_1}+\|v\|^{\rho_1}_{\beta_1})  \|(u-v)\|_{\beta_1} & \text{(by Sobolev embedding).}
\end{align*}
As explained above, the Sobolev embedding requires
$-\frac{d}{2(\rho_1 + 1)} \leq 2\beta_1 - 1 - \frac{d}{2}$.
The largest possible $\beta_1\in (1/2,1]$ we can take in the criticality condition \eqref{eq: criticality F} with $(\rho_F, \alpha_F, \beta_F)$ replaced by $(\rho_1, \tfrac{1}{2}, \beta_1)$ is $(1+\rho_1)(2\beta_1-1) = 2$. By the assumption $\rho_1\geq 2$, one can check that $\beta\in (1/2,1)$. Moreover, with this choice, the condition for the Sobolev embedding becomes
\[ -\frac{d}{2(\rho_1+1)} \leq \frac{2}{\rho_1+1} - \frac{d}{2}.\]
The latter is equivalent to $\rho_1\leq \frac{4}{d}$, which was our assumption.

The next result is now a direct consequence of Theorems~\ref{thm:varglobal} (or Corollary~\ref{cor:varglobal2} if $\rho_3 = 0$) and~\ref{thm:contdep}.
\begin{theorem}[Global well-posedness]
\label{thm:second}
Suppose that Assumption~\ref{ass:2ndorder1} holds. If $\rho_3 = 0$, also the case $\eta = 0$ is permitted. Let $u_0\in L^0(\Omega, L^2(\Dom))$ be strongly $\cF_0$-measurable.
Then \eqref{eq:reaction_diffusion} has a unique global solution
\begin{equation}\label{eq:solspacesecond}
u\in D([0,\infty),L^2(\Dom))\cap L^2_{\rm loc}([0,\infty),H^1_0(\Dom)) \ a.s.
\end{equation}
and for every $T>0$ there is a constant $C_T$ independent of $u_0$ such that
\begin{align}\label{eq:aprioriLpsec}
\E \|u\|_{D([0,T],L^2(\Dom))}^2 + \E\|u\|_{L^2(0,T),H^1_0(\Dom))}^2\leq C_T(1+\E\|u_0\|_{L^2(\Dom)}^2).
\end{align}
Moreover, $u$ depends continuously on the initial data $u_0$ in probability in the sense of Theorem~\ref{thm:contdep} with $\cH=L^2(\Dom)$ and $\cV=H^{1}_0(\Dom)$.
\end{theorem}

As a more concrete example, we consider a generalized Burger's equation with gradient noise, where for simplicity we take $(\omega,t,x)$-independent coefficients. In particular, it includes the Allen--Cahn equations for $d\in \{1, 2\}$.
\begin{example}\label{ex:seconddiff1}
Let $d\geq 1$ and let $\Dom$ be a bounded $C^1$-domain. Consider the problem
\begin{equation}
\label{eq:reaction_diffusionexdiff1}
\left\{
\begin{aligned}
d u  &= [\Delta u+ f(u) + \dv(\of(u))] \, d t+ \sum_{n\geq 1} \big[(b_n\cdot \nabla)u + g_{n}(u)\big] \, d w_t^n \\ &  \quad +  \int_Z \big[(c(z)\cdot \nabla)u(\cdot -) + h(u(\cdot -),z)\big] \, \wt N(dz,dt) &\quad & \text{ on }\Dom,\\
u &= 0& \quad & \text{ on }\partial\Dom,
\\
u(0)&=u_{0}&\quad & \text{ on }\Dom.
\end{aligned}\right.
\end{equation}
Here, $(b_n)_{n\geq 1}$  and $(c(z))_{z\in Z}$ are real valued such that stochastic parabolicity condition holds:
\[\theta := 1 - \frac12\|(b_n)_{n\geq 1}\|_{\ell^2}^2 - \frac{1}{2} \|c\|_{L^2(Z;\nu)}^2>0.\]
For the nonlinearities, we assume that there is a constant $C\geq 0$ such that for $y,y'\in \R$ it holds
\begin{align*}
|f(y)-f(y')| & \leq C (1+|y|^{\rho_1}+|y'|^{\rho_1})|y-y'|,
\\ |f(y)|&\leq C(1+|y|^{\rho_1+1}),
\\ |\of(y)-\of(y')| & \leq C (1+|y|^{\rho_2}+|y'|^{\rho_2})|y-y'|,
\\ |\of(y)|& \leq C(1+|y|^{\rho_2+1}),
\\ \|g(y)-g(y')\|_{\ell^2}+ \|h(y,\cdot)-h(y',\cdot)\|_{L^2(Z;\nu)}&\leq C|y-y'|,
\\ \|(g_n(y))_{n\geq 1}\|_{\ell^2}+\|h(y,\cdot)\|_{L^2(Z;\nu)} &\leq C(1+|y|),
\end{align*}
where $\rho_1\in [0,\frac{4}{d}]$, and $\rho_2\in [0,\frac{2}{d}]$ (cf.\ Assumption~\ref{ass:2ndorder1}).
Suppose that the following dissipativity condition holds: there is an $M\geq 0$ such that
\begin{align}\label{eq:dissLR}
y f(y)\leq M(1+|y|^2).
\end{align}
In particular, if $d\in \{1, 2\}$, Burgers type nonlinearities are included: take $\of(y) = y^2$. Moreover, if $d\in \{1,2\}$, Allen--Cahn type nonlinearities such as $f(y) = y-y^3$ are included as well. For $d=1$ one can even allow $f(y) = y - y^5$, possibly with additional terms $c_i y^i$ with $i\in \{1, \ldots, 4\}$.  Both $-y^3$ for $d=2$ and $-y^5$ for $d=1$ are critical, and were not included in previous settings.

As in~\cite[Example 5.13]{agresti2022critical}, one can check that Assumption~\ref{ass:2ndorder1} is satisfied. Thus, Theorem~\ref{thm:second} implies that for every strongly $\cF_0$-measurable $u_0:\Omega\to L^2(\Dom)$, there exists a unique global solution $u$ to \eqref{eq:reaction_diffusionexdiff1} which satisfies \eqref{eq:solspacesecond} and \eqref{eq:aprioriLpsec}, and the continuous dependency assertion of Theorem~\ref{thm:second}.
\end{example}

The functions $g$ and $h$ can also have superlinear growth. For simplicity we present this for the Allen--Cahn equation, and only in the case $b=0$, $c=0$ and $(\omega,t,x)$-independent coefficients.
\begin{example}[Allen--Cahn with quadratic noise]\label{ex:seconddiff2}
Let $d\in \{1, 2\}$. Consider the problem
\begin{equation}
\label{eq:reaction_diffusionexdiff2}
\left\{
\begin{aligned}
d u &=\big[ \Delta u -u^3+u \big]\, d t
+ \sum_{n\geq 1}  g_{n}(u) \, d w_t^n+ \int_Z h(u(\cdot -),z) \,  \wt N(dz,dt)&\quad &\text{ on }\Dom,\\
u &= 0 &\quad & \text{ on $\partial \Dom$},
\\
u(0)&=u_{0}&\quad &\text{ on }\Dom.
\end{aligned}\right.
\end{equation}
Here, we assume that there exists a $C_0\geq 0$ and $C_1\in [0,2)$ such that
\begin{align}
 \nonumber \|g(y) - g(y')\|_{\ell^2} + \|h(y,\cdot) - h(y',\cdot)\|_{L^2(Z;\nu)}&\leq C_0(1+|y|+|y'|)|y-y'|,
\\ \label{eq:coercsecondex}\|g(y)\|_{\ell^2}^2+ \|h(y,\cdot)\|_{L^2(Z;\nu)}^2&\leq C_0+C_1|y|^4,
\end{align}
where $y,y'\in \R$.
Then, Theorem~\ref{thm:second} implies that \eqref{eq:reaction_diffusionexdiff2} has a unique global solution $u$ as in \eqref{eq:solspacesecond} and \eqref{eq:aprioriLpsec} holds. Indeed, Assumptions~\ref{ass:2ndorder1}~\eqref{ass:2ndorder11},~\eqref{ass:2ndorder12},~\eqref{it:growth_nonlinearities}  are clearly satisfied with $\rho_1 = 2$ and $\rho_3 = 1$. For \eqref{it:secondcoercity} it remains to note that with $f(y) = -y^3 + y$ we have for $\eta>0$ small enough that a.e.\ on $\R_+\times \Omega\times \Dom$ for all $y\in \R$
\begin{align*}
y f(y) + (\tfrac12+\eta)\|g(y)\|_{\ell^2}^2+ (\tfrac12+\eta)\|h(y,\cdot)\|_{L^2(Z;\nu)}^2& \leq [(\tfrac12+\eta)C_1 - 1]|y|^4+ C(1+|y|^2)\\ & \leq C(1+|y|^2).
\end{align*}
Thus, \eqref{it:secondcoercity} follows.
\end{example}

Next, we present an example where $\beta_F= 1$ and $\alpha_F = \frac12$.  By allowing $\beta_F = 1$ we can take $F$ to act in a nonlinear way in $\nabla u$ as well. This is not possible if $\beta_F <1$ as in earlier works on the variational setting. The example below is non-physical as far as we know, but we include it as there might be new types of nonlinearities which can now also be treated. For simplicity we consider the same form of the equation \eqref{eq:reaction_diffusion} under the same Assumption~\ref{ass:2ndorder1}~\eqref{ass:2ndorder11} and \eqref{ass:2ndorder12} on $a, b, c$. For $g, h$ we assume that they are globally Lipschitz.
\begin{example}[Nonlinearity of gradient type]\label{ex:beta=1}
Consider \eqref{eq:reaction_diffusion} with $f$ replaced by the nonlinearity $f(u,\nabla u) = -u|\nabla u|$ and with $d=1$. For $(a,b,c)$ we suppose that Assumption~\ref{ass:2ndorder1}~\eqref{ass:2ndorder11} and \eqref{ass:2ndorder12} holds. Moreover, we assume that $g$ and $h$ satisfy for $y,y' \in \R$ the following:
\begin{align*}
 \|g(y)-g(y')\|_{\ell^2} + \|h(y,\cdot)-h(y',\cdot)\|_{L^2(Z;\nu)}&\leq C|y-y'|,
\\ \|(g_n(y))_{n\geq 1}\|_{\ell^2} + \|h(y,\cdot)\|_{L^2(Z;\nu)}  &\leq C(1+|y|).
\end{align*}
Then for every strongly $\cF_0$-measurable $u_0:\Omega\to L^2(\Dom)$ there exists a unique global solution $u$ to \eqref{eq:reaction_diffusion} which satisfies \eqref{eq:solspacesecond} and \eqref{eq:aprioriLpsec}, and the continuous dependency assertion of Theorem~\ref{thm:second} holds. To see this, it suffices as before to check that the conditions of Corollary~\ref{cor:varglobal2} are satisfied. Most conditions are clear.  We need to verify the mapping properties of $F(u) = -u|\nabla u|$. The coercivity $\lb F(u), u\rb \leq 0$ is obvious from the choice of the nonlinearity.

We check the mapping properties of $F$ with the choices $\beta_ F = 1$ and $\alpha_F = 1/2$. Note that for $u,v\in \cV$,
\begin{align*}
\|u|\nabla u|  - v|\nabla v|\|_{L^2(\Dom)} &  \leq \|(u-v)|\nabla u|\|_{L^2(\Dom)}   + \|v(|\nabla u| - |\nabla v|)\|_{L^2(\Dom)}
\\ & \leq \|u-v\|_{L^\infty(\Dom)}\|\nabla u\|_{L^2(\Dom)}   + \|v\|_{L^\infty(\Dom)} \| \nabla u - \nabla v\|_{L^2(\Dom)}
\\ & \lesssim (\|u\|_{\cV} + \|v\|_{\cV}) \|u-v\|_{\cV},
\end{align*}
where we used $\cV = H^1\hookrightarrow L^\infty$ for $d=1$ in the last step. Therefore, $\rho_F = 1$, and the subcriticality condition~\eqref{eq: criticality F} is satisfied.
\end{example}

\begin{example}[Singular drift]\label{ex:singular}
 Suppose that Assumption~\ref{ass:2ndorder1} holds.  Consider \eqref{eq:reaction_diffusion} with an additional term $A_S(t) u = A_{S,0}u + A_{S,1}u=\zeta_0 u + \sum_{j=1}^d \zeta_j \partial_j u$, where $\zeta_0\in L^1_{\rm loc}([0,\infty))$ and $\zeta_j\in L^2_{\rm loc}([0,\infty))$. Recall from Remark~\ref{rem: ass nonlinear operators}~\eqref{it:sum assump} that $A_S$ is indeed allowed to consist of a finite sum of operators with individual growth restrictions. One has
\[\|A_{S,1}(t) u\|_{L^2(\Dom)} \leq \zeta(t) \|u\|_{H^1(\Dom)}, \]
where $\zeta(t):= \Big(\sum_{i=1}^d |\zeta_j(t)|^2\Big)^{1/2}$. This shows that $A_{S,1}$ indeed satisfies the growth estimate of Assumption~\ref{assumption operators new} with $\alpha_{A,1} = 1/2$, $\beta_{A,1} = 1$, and $r_{A,1} = 2$. The $A_{S,0}$ term can be treated in the same way with $\alpha_{A,0} = 1/2$, $\beta_{A,0} = 1/2$, and $r_{A,1} = 1$. The coercivity condition of Assumption~\ref{ass:2ndorder1}~\eqref{it:secondcoercity} still holds with this additional term if the constant $M$ is replaced by $\phi:=-|\zeta_0|$ which is in $L^1_{\rm loc}([0,\infty))$. Indeed, the first order term cancels since $\int_{\Dom} u\partial_j u dx = \int_{\Dom} \partial_j (u^2/2) dx = 0$, where we benefit from the Dirichlet boundary conditions. The zeroth order term is taken care of via the $\phi$-term.
\end{example}

\subsection{Stochastic fluid dynamics models with L\'evy noise}\label{ss:fluid}

The general problem we consider has the form
\begin{equation}\label{eq:abstractfluid}
\begin{aligned}
d u + A_0 u \, dt &= \Phi(u, u) \, dt + (B_0 u + G(u)) \,dW + \int_Z(C_0 u(\cdot -) + H(u(\cdot -))) \, \wt N(dz, dt),
\\ u(0) &= u_0.
\end{aligned}
\end{equation}
Many fluid dynamics models fit into the above setting for a suitable bilinear map $\Phi:\cV_{\nicefrac{3}{4}}\times \cV_{\nicefrac{3}{4}}\to \cV^*$
(see e.g.\ \cite{AVsurvey,ChueshovMillet}).
For example, it includes the 2d Navier--Stokes equations on bounded or unbounded domains, but also 2d quasigeostrophic equations, 2d Boussinesq equations, $2d$ magneto-hydrodynamic equations, $2d$ magnetic B\'enard problem, 3d Leray $\alpha$-model for Navier--Stokes equations, or shell models of turbulence.
Since these papers explain many details on the specific fluid dynamics applications, we will only formulate the abstract results below.

\begin{assumption}\label{ass:bilinearFluid}
Let $A_0\in \calL(\cV, \cV^*)$, $B_0\in \calL(\cV, \calL_2(U,\cH))$ and $C_0\in \calL(V,L^2(Z,\cH;\nu))$.
\begin{enumerate}[{\rm (1)}]
\item\label{it1:bilinearFluid}   There exist $\kappa>0$ and $C\geq 0$ such that for all $v\in \cV$,
    \[\lb A_0v, v\rb  - \frac{1}{2} \|B_0 v\|_{\calL_2(U;\cH)}^2 - \frac{1}{2} \|C_0 v\|_{L^2(Z;\nu,\cH)}^2 \geq \kappa\|v\|_{\cV}^2- C \|v\|_{\cH}^2.\]

\item\label{it2:bilinearFluid}
$\Phi:\cV_{\nicefrac{3}{4}}\times \cV_{\nicefrac{3}{4}}\to \cV^*$ is bilinear and satisfies
\[\|\Phi(u, v)\|_{\cV^*}\leq C\|u\|_{\cV_{\nicefrac{3}{4}}} \|v\|_{\cV_{\nicefrac{3}{4}}}, \ \ \lb u, \Phi(u,u)\rb = 0, \ \ u,v\in \cV.\]

\item\label{it3:bilinearFluid} For some $\beta_G,\beta_H\in (1/2, 1)$, $G:\cV_{\beta_G}\to \calL_2(U,\cH)$ and $H:\cV_{\beta_H}\to L^2(Z,H;\nu)$  are globally Lipschitz
\end{enumerate}
\end{assumption}

To formulate \eqref{eq:abstractfluid} as \eqref{eq:spde} we let $A(v) = A_0 v - F(v)$ with $F:\cV_{\beta}\to \cV^*$ given $F(v) = \Phi(v,v)$. Then $F$ satisfies Assumption~\ref{assumption operators new} with $\rho_F=1$ and $\beta_F = \nicefrac{3}{4}$.

\begin{theorem}\label{thm:fluidabstract}
Suppose that Assumption~\ref{ass:bilinearFluid} holds.
Then for every $u_0\in L^0(\Omega,\cH)$ strongly $\cF_0$-measurable, \eqref{eq:abstractfluid} has a unique global solution $u\in L^2_{\rm loc}([0,\infty),\cV)\cap D([0,\infty),\cH)$ a.s. Moreover, for all $T\in (0,\infty)$
\begin{align*}
\E \sup_{t\in [0,T]}\|u(t)\|_{\cH}^{2} + \E\int_0^T \|u(t)\|^2_{\cV} dt & \leq C_{T} (1+\E\|u_0\|_{\cH}^{2}).
\end{align*}
Furthermore, the following continuous dependency on the initial data holds: if $u_0^n \in L^0(\Omega;\cH)$ strongly $\cF_0$-measurable are such that $\|u_0-u_0^n\|_{\cH}\to 0$ in probability, then for every $T\in (0,\infty)$,
\[\|u - u^n\|_{L^2([0,T],\cV)} + \|u - u^n\|_{D([0,T],\cH)}\to 0 \ \ \text{in probability},\]
where $u^n$ is the unique global solution to \eqref{eq:abstractfluid} with initial data $u_0^n$.
\end{theorem}
\begin{proof}
We apply Corollary~\ref{cor:varglobal2} and Theorem~\ref{thm:contdep}. It is straightforward to check that Assumption~\ref{assumption operators new} holds with $A_S = 0$. Indeed, for $F$ note that by the bilinearity and boundedness
\begin{align*}
\|F(u)-F(v)\|_{\cV^*} & = \|\Phi(u, u) - \Phi(v, v)\|_{\cV^*}
\\ & = \|\Phi(u, u-v)+\Phi(u-v,v)\|_{\cV^*}
\\ & \leq C(\|u\|_{\nicefrac{3}{4}} +\|v\|_{\nicefrac{3}{4}})\|u-v\|_{\nicefrac{3}{4}}.
\end{align*}
This leads to $\rho_F = 1$ and $\beta_F = \nicefrac{3}{4}$, which is critical. By the conditions on $G$ and $H$ we may take $\rho_G = 0 = \rho_H$.
For the coercivity condition \eqref{nonlinear coercivity conditionsimple3} it suffices to note that
\begin{align*}
\lb F(u), u \rb = \lb \Phi(u,u),u\rb=0. &\qedhere
\end{align*}
\end{proof}

\begin{remark}
The problem \eqref{eq:abstractfluid} also fits into the setting of \cite{brzezniak2014strong} if one assumes an additional smallness condition on the noise (see (1.2) in the latter paper).
\end{remark}

\subsection{Stochastic Kuramoto--Sivashinsky equation}\label{ss:KS}
We are going to study a fourth order equation called the Kuramoto--Sivashinsky equation. We consider the conservative form of the equation as studied in~\cite{ZhangKS}. It is used both for chemical
reactions and laminar flames. The nonlinearity is of Burgers' type, and can be handled in our new variational setting.

A stochastic version of the Kuramoto--Sivashinsky equation was considered in~\cite{Duan-etal}. We show how our results can be used to simplify their proofs and weaken the conditions on the nonlinearity (see Remark~\ref{rem:comparisonDuan} for a comparison). For simplicity, we only consider Gaussian noise, but the results can be extended to the case of L\'evy noise without difficulty. For simplicity, we formulate the results on a bounded $C^2$-domain $\Dom\subseteq \R^d$. With minor modifications, the result could be formulated for unbounded domains as well, since we do not need any compactness of embeddings.

On $\Dom$ consider the following fourth order equation:
\begin{equation}
\label{eq:KS}
\left\{
\begin{aligned}
d u  &= [-\Delta^2 u -\Delta u - \dv(\bar{f}(u))]\, d t + \sum_{n\geq 1}g_n(\cdot,u,\nabla u) \, d w_t^n&\quad & \text{ on } \Dom,
\\ u& =0 \ \  \text{ and } \ \ \Delta u =0&\quad & \text{ on }\partial \Dom,
\\ u(0)&=u_0&\quad  &\text{ on } \Dom.
\end{aligned}\right.
\end{equation}
As usual, the $w^n$ are independent standard Brownian motions with respect to our given filtration. The nonlinearity $\dv(\bar{f}(u))$ is of conservative type, which will be used below in \eqref{eq:conserKS}.

The only assumption we will need is the following:
\begin{assumption}\label{ass:KS}
Let $d\geq 1$, $\rho\in [0,6/d]$ and let $\Dom$ be a bounded $C^2$-domain. Suppose that $\bar{f}:\R\to \R^d$ and there exists a constant $C$ such that for all $y,y'\in \R$,
\begin{align*}
|\bar{f}(y)-\bar{f}(y')|
& \leq C (1+|y|^{\rho}+|y'|^{\rho})|y-y'|.
\end{align*}
Suppose that $g:\Omega\times\R_+\times\Dom\times\R\times\R^{d}\to \ell^2$ is $\cP\otimes \cB(\Dom)\otimes\cB(\R)\otimes\cB(\R^d)$-measurable, and there is a constant $C$ such that for all $y,y'\in \R$ and $v,v'\in \R^d$, a.e. on $\Omega\times\R_+\times \Dom$
\begin{align*}
\|g(\cdot,y,v) - g(\cdot,y',v')\|_{\ell^2} &\leq C|y-y'| + C|v-v'|,
\\ \|g(\cdot,y,v)\|_{\ell^2} &\leq C(1+|y|+|v|).
\end{align*}
\end{assumption}
The physical case corresponds to $\rho=1$, which is admissible for $d\leq 6$.

In order to show well-posedness, we formulate \eqref{eq:KS} in our setting and check the conditions of Corollary~\ref{cor:varglobal2}. Let $\cV = H^2(\Dom)\cap H^1_0(\Dom)$, $\cH = L^2(\Dom)$. Then for all $\beta\in (1/2, 1)$, $\cV_{\beta} \hookrightarrow H^{4\beta-2}(\Dom)$. Moreover, one can show that $\cV_{3/4} = H^1_0(\Dom)$, and thus by duality $\cV_{1/4} = H^{-1}(\Dom)$.

Let $A_0 v=\Delta^2 v +\Delta v$, $B_0 = 0$, $C_0=0$. Then the required coercivity condition follows from elliptic regularity theory (using the $C^2$-regularity of the domain) and the interpolation inequality:
\[\lb A_0 v, v\rb = \|\Delta v\|_{L^2(\Dom)}^2 -\| \nabla v \|_{L^2(\Dom)}^2 \geq \kappa \|v\|_{H^2(\Dom)}^2 - C \|v\|^2_{L^2(\Dom)}.\]

Define $F$, $G$ and $H$ by $F(v) = \dv(\bar{f}(v))$ and $G(v) = g(\cdot,v, \nabla v)$ and $H = 0$. Then
\begin{align*}
\|G(\cdot,u) - G(\cdot, v)\|_{\calL_2(U,\cH)} &= \|g(\cdot, u,\nabla u) - g(\cdot, v,\nabla v)\|_{L^2(\Dom;\ell^2)} \\ & \leq C\|u-v\|_{L^2(\Dom)} + C\|\nabla u - \nabla v\|_{L^2(\Dom)}\leq C' \|u-v\|_{3/4}.
\end{align*}
The growth condition can be proved in the same way.

To check the conditions on $F$ it suffices to check the locally Lipschitz estimate, since $F(0)$ is a constant and hence in $\cV^*$. First, let $d\geq 3$. By the Sobolev embedding and H\"older's inequality we find
\begin{align*}
\|F(u) - F(v)\|_{\cV^*} &\leq \|\bar{f}(u) - \bar{f}(v)\|_{H^{-1}(\Dom)}
\\ & \leq \|\bar{f}(u) - \bar{f}(v)\|_{L^r(\Dom)}
\\ & \leq C \|(1+|u|^{\rho}+|v|^{\rho})|u-v|\|_{L^r(\Dom)}
\\ & \leq C'(1+\|u\|_{L^{r(\rho+1)(\Dom)}}^{\rho}+\|v\|_{L^{r(\rho+1)(\Dom)}}^{\rho})\|u-v\|_{L^{r(\rho+1)(\Dom)}},
\end{align*}
where we set $-\frac{d}{r} = -1-\frac{d}{2}$. Since $d\geq 3$, we see that $r\in (1, 2)$. Taking the critical value $\beta_F = \frac{1}{2} + \frac{1}{2} \frac{1}{1+\rho}$ one can check that
$V_{\beta}\hookrightarrow H^{4\beta-2}(\Dom)\hookrightarrow L^{r(\rho+1)}(\Dom)$ if $4\beta_F-2 - \frac{d}{2} \geq -\frac{d}{r(\rho+1)} = -\frac{1}{\rho+1} - \frac{d}{2(\rho+1)}$. The latter condition is equivalent to $\rho \leq \tfrac{6}{d}$. It follows that $F$ satisfies the required locally Lipschitz condition with $\alpha_F = 0$ and $(\beta_F,\rho)$ as above.

Next, consider $d\in \{1, 2\}$. Without loss of generality we may assume $\rho\geq 1$. Taking $\alpha_F = 1/4$ (recall $\cV_{1/4} = H^{-1}(\Dom)$) and the corresponding critical value $\beta_F = \frac{1}{2} + \frac{3}{4(\rho+1)}$ which is in $(1/2,1)$, we find that
\begin{align*}
\|F(u) - F(v)\|_{1/4} &\leq \|\bar{f}(u) - \bar{f}(v)\|_{L^2(\Dom)}
\\ & \leq C \|(1+|u|^{\rho}+|v|^{\rho})|u-v|\|_{L^2(\Dom)}
\\ & \leq C'(1+\|u\|_{L^{2(\rho+1)}(\Dom)}^{\rho}+\|v\|_{L^{2(\rho+1)}(\Dom)}^{\rho})\|u-v\|_{L^{2(\rho+1)}(\Dom)}
\\ & \leq C'(1+\|u\|_{\beta_F}+\|v\|_{\beta_F})\|u-v\|_{\beta_F},
\end{align*}
where in the last step we applied the Sobolev embedding with
$4\beta-2 -\frac{d}{2}\geq -\frac{d}{2(\rho+1)}$ which again is valid due to $\rho\leq 6/d$.

To check the coercivity condition \eqref{nonlinear coercivity conditionsimple3} for $F$ note that for all $v\in C^2(\Dom)$ with $v=0$ on $\partial \Dom$ by integration by parts and the divergence theorem
\begin{align}\label{eq:conserKS}
\langle F(v), v \rangle  = \int_{\Dom} \bar{f}(v) \cdot \nabla v \,dx= \int_{\Dom} \dv(\Phi(v)) \,dx = 0,
\end{align}
where we set $\Phi(y) = \int_0^y \bar{f}(z) dz$.  By an approximation argument this extend to all $v\in H^2(\Dom)\cap H^1_0(\Dom)$.

From Corollary~\ref{cor:varglobal2} we conclude the following result.
\begin{theorem}[Global well-posedness]
\label{thm:SH}
Suppose that Assumption~\ref{ass:KS} holds. Let $u_0\in L^0(\Omega; L^2(\Dom))$ be strongly $\cF_0$-measurable.
Then \eqref{eq:KS} has a unique global solution
\begin{equation}\label{eq:solspaceKS}
u\in C([0,\infty),L^2(\Dom))\cap L^2_{\rm loc}([0,\infty),H^2(\Dom)\cap H^1_0(\Dom)) \ a.s.,
\end{equation}
and for every $T>0$ there exists a constant $C_T$ independent of $u_0$ such that
\[\E \|u\|_{C([0,T],L^2(\Dom))}^2 + \E\|u\|_{L^2([0,T],H^2(\Dom))}^2\leq C_T(1+\E\|u_0\|^2_{L^2(\Dom)}).\]
Furthermore, the following continuous dependency on the initial data holds: if $u_0^n \in L^0(\Omega;L^2(\Dom))$ are strongly $\cF_0$-measurable such that $\|u_0-u_0^n\|_{L^2(\Dom)}\to 0$ in probability, then for every $T\in (0,\infty)$,
\[\|u - u^n\|_{L^2([0,T],H^2(\Dom) )} + \|u - u^n\|_{C([0,T];L^2(\Dom))}\to 0 \ \ \text{in probability},\]
where $u^n$ is the unique global solution to \eqref{eq:abstractfluid} with initial data $u_0^n$.
\end{theorem}

\begin{remark}\label{rem:comparisonDuan} \
\begin{enumerate}
\item It is straightforward  to find a correspondence between the way the noise is modeled in~\cite{Duan-etal} and our setting. Indeed, if $\overline{g}$ denotes the nonlinearity in the latter paper, then we can take $g_n =c_n e_n \overline{g}$, where $(c_n)\in \ell^2$ and $\sup_{n\geq 1}\|e_n\|_{L^\infty(\Dom)}<\infty$ are coming from their Brownian noise term. The $L^\infty$-bound follows from their condition $(C)$. In this way $g$, satisfies Assumption~\ref{ass:KS}.

\item Without any further conditions, it is possible to let $g$ and $h$ depend on $\nabla^2 u$ as well, if the Lipschitz constant with respect to this variable is small enough. This can be proved by a fixed point argument.

\item More interesting might be that it is possible to include a linear term of the form $B u = \sum_{i,j=1}^d b_{n}^{i,j} \partial_i \partial_j u$. In order to ensure coercivity, one needs that for all $u\in H^2(\Dom)\cap H^1_0(\Dom)$ it holds
    \[\|\Delta u\|^2_{L^2(\Dom)} - \frac12 \sum_{n\geq 1} \Big\|\sum_{i,j=1}^d b_{n}^{i,j} \partial_i \partial_j u\Big\|^2_{L^2(\Dom)}\geq \kappa \|u\|_{H^2(\Dom)}^2 - C \|u\|_{L^2(\Dom)}^2.\]
    If the domain $\Dom$ is convex and $C^2$, then a sufficient condition for the latter is
\[\sum_{n\geq 1}\sum_{i,j,k,\ell=1}^d b_n^{i,j} b_n^{k,\ell} \xi_{i,j} \xi_{k,\ell}\leq (2-\kappa') |\xi|^2.\]
where $|\xi|^2 = \sum_{i,j=1}^d \xi_{i,j}^2$. To see this, one can apply Kadlec's formula (see \cite[Exercise 5.5.6-7]{TayPDE1} and \cite[Appendix]{AHHS}).

\item in~\cite{Duan-etal}, in the case where $g$ only depends on $u$, the assumption on $f$ is that $\rho=2$ if $1\leq d\leq 5$, and $\rho<6/d$ if $d\geq 6$. In the case $g$ also depends on $\nabla u$, they assume that the dependency has a small Lipschitz constant, and that $\rho=1$ if $d=1$ and $\rho<2/d$ if $d\geq 2$. The latter excludes the physical case $\rho=1$ if $d\geq 2$. Our conditions are more flexible and include the physical nonlinearity for $d\leq 6$.

\item In the above paper, it is claimed  that one even has that the solution mapping $u_0\mapsto u$ from $L^2(\Omega,L^2(\Dom))$ into $L^2(\Omega,L^2([0,T],H^2(\Dom))\cap C([0,T],L^2(\Dom)))$ is Lipschitz continuous. This was proved for the truncated equation, but we do not know how to extend this to the full problem. A partial result does hold. Indeed, by a standard argument involving uniform integrability, we can obtain such result with range space $L^q(\Omega,L^2([0,T],H^2(\Dom))\cap C([0,T],L^2(\Dom)))$ with arbitrary $q\in (0,2)$ (see \cite[Theorem 3.8]{agresti2022critical}).
\end{enumerate}
\end{remark}

\subsection{Remarks about the stochastic Swift--Hohenberg equation}\label{ss:SH}
The Swift--Hohenberg equation appears in several application areas and is a partial differential equations which has special pattern formations. In the stochastic case, it has been studied on a bounded interval with Dirichlet boundary conditions in~\cite{Gao}. The well-posedness theory was extended to more general bounded domains in~\cite[Section 5.6]{agresti2022critical}. Using the theory of our paper, it can be extended to the setting with L\'evy noise. Furthermore, we can weaken the condition on the nonlinearity considerably and are now also able to include the important case of fifth-order polynomials in the cases $d=1$ and $d=2$. in~\cite{agresti2022critical} these type of nonlinearities were excluded except if $d=1$. More generally, the condition on the nonlinearity $f$ in the latter work can be replaced by $\rho\leq \frac{8}{d}$. This gives many new cases for $d\in \{1, 2, 3, 4\}$.

To prove the above statement, by the cases already considered in~\cite{agresti2022critical} we may assume $d\in \{1, 2, 3, 4\}$ and $\rho\geq 2$. In the same way as we saw before Theorem~\ref{thm:second} (and similarly in Theorem~\ref{thm:SH}), one can check that $F(u) = f(u)$ satisfies our conditions with $\alpha=1/2$ and $\beta$ such that $(1 + \rho)(2\beta - 1) =2$.

\subsubsection*{Acknowledgments}
The authors thank Antonio Agresti for helpful discussions and Esm\'ee Theewis for useful comments. The authors also thank Istv\'an Gy\"ongy and Nicolai Krylov for valuable remarks and discussions. Finally, we thank
the referees for careful reading and helpful comments. 

\subsubsection*{Funding} The first-named author is supported by the Alexander von Humboldt foundation by a Feodor Lynen grant. The second-named and third-named authors are supported by the VICI subsidy VI.C.212.027 of the Netherlands Organisation for Scientific Research (NWO). This project has received funding from
the European Union’s Horizon 2020 research and innovation programme under the Marie Skłodowska-Curie grant agreement No 101034255 \euflag{}.


\begin{thebibliography}{11}

	\bibitem{agresti2022nonlinear}
	\textsc{A.~Agresti} and \textsc{M.~Veraar}.
	\newblock {\em Nonlinear parabolic stochastic evolution equations in critical spaces part {I}. {S}tochastic maximal regularity and local existence\/}.
	\newblock Nonlinearity~\textbf{35} (2022), no.~8, 4100--4210.
	\newblock DOI:10.1088/1361-6544/abd613.

	\bibitem{agresti2022nonlinear2}
	\textsc{A.~Agresti} and \textsc{M.~Veraar}.
	\newblock {\em Nonlinear parabolic stochastic evolution equations in critical spaces part {II}. {B}low-up criteria and instantaneous regularization\/}.
	\newblock J. Evol. Equ.~\textbf{22} (2022), no.~2.
	\newblock DOI:10.1007/s00028-022-00786-7.

	\bibitem{agresti2022critical}
	\textsc{A.~Agresti} and \textsc{M.~Veraar}.
	\newblock {\em The critical variational setting for stochastic evolution
		equations\/}.
	\newblock Probab. Theory Related Fields~\textbf{188} (2024), no.~3-4, 957--1015.
	\newblock DOI:10.1007/s00440-023-01249-x.

	\bibitem{AV20_NS}
	\textsc{A.~Agresti} and \textsc{M.~Veraar}.
	\newblock {\em Stochastic {N}avier-{S}tokes equations for turbulent flows in
		critical spaces\/}.
	\newblock Comm. Math. Phys.~\textbf{405} (2024), no.~2.
	\newblock DOI:10.1007/s00220-023-04867-7.

    \bibitem{AVsurvey}
	\textsc{A.~Agresti} and \textsc{M.~Veraar}.
	\newblock {\em Nonlinear SPDEs and maximal regularity: an extended survey\/}.
	\newblock NoDEA Nonlinear Differential Equations Appl.~\textbf{32} (2025), no.~6. \newblock DOI:10.1007/s00030-025-01090-2

	\bibitem{AHHS}
	\textsc{A.~Agresti}, \textsc{M.~Hieber}, \textsc{A.~Hussein}, and \textsc{M.~Saal}.
	\newblock {\em The stochastic primitive equations with transport noise and
		turbulent pressure\/}.
	\newblock Stoch. Partial Differ. Equ. Anal. Comput.~\textbf{12} (2024), no.~1, 53--133.
	\newblock DOI:10.1007/s40072-022-00277-3.

	\bibitem{BT73}
	\textsc{A.~Bensoussan} and \textsc{R.~Temam}.
	\newblock {\em \'Equations stochastiques du type {N}avier-{S}tokes\/}.
	\newblock J. Functional Analysis~\textbf{13} (1973), 195--222.
	\newblock DOI:10.1016/0022-1236(73)90045-1.

	\bibitem{BeLo}
	\textsc{J.~Bergh} and \textsc{J.~L{\"o}fstr{\"o}m}.
	\newblock Interpolation {S}paces. {A}n {I}ntroduction. {G}rundlehren der {M}athematischen {W}issenschaften, vol.~223.
	\newblock Springer, Berlin, 1976.

	\bibitem{brzezniak2014strong}
	\textsc{Z.~Brze\'zniak}, \textsc{W.~Liu}, and \textsc{J.~Zhu}.
	\newblock {\em Strong solutions for {SPDE} with locally monotone coefficients
		driven by {L}\'evy noise\/}.
	\newblock Nonlinear Anal. Real World Appl.~\textbf{17} (2014), 283--310.
	\newblock DOI:10.1016/j.nonrwa.2013.12.005.

    \bibitem{BT24} \textsc{G.~Barrera}, and \textsc{J.M. T\"olle}. \newblock {\em Ergodicity for locally monotone stochastic evolution equations with {L}\'evy noise\/}. \newblock Available on ArXiv: \url{https://arxiv.org/abs/2412.01381}.

	\bibitem{BudDup}
	\textsc{A.~Budhiraja} and \textsc{P.~Dupuis}.
	\newblock {\em A variational representation for positive functionals of
		infinite dimensional {B}rownian motion\/}.
	\newblock Probab. Math. Statist.~\textbf{20} (2000), no.~1, 39--61.

	\bibitem{ChueshovMillet}
	\textsc{I.~Chueshov} and \textsc{A.~Millet}.
	\newblock {\em Stochastic 2{D} hydrodynamical type systems: well posedness
		and large deviations\/}.
	\newblock Appl. Math. Optim.~\textbf{61} (2010), no.~3, 379--420.
	\newblock DOI:10.1007/s00245-009-9091-z.

		\bibitem{Duan-etal}
	\textsc{S.~Cui}, \textsc{J.~Duan}, and \textsc{W.~Wu}.
	\newblock {\em Global well-posedness of the stochastic generalized
		{K}uramoto-{S}ivashinsky equation with multiplicative noise\/}.
	\newblock Acta Math. Appl. Sin. Engl. Ser.~\textbf{34} (2018), no.~3, 566--584.
	\newblock DOI:10.1007/s10255-018-0769-3.

	\bibitem{Temametal}
	\textsc{J.~Cyr}, \textsc{P.~Nguyen}, \textsc{S.~Tang}, and \textsc{R.~Temam}.
	\newblock {\em Review of local and global existence results for stochastic
		{PDE}s with {L}\'evy noise\/}.
	\newblock Discrete Contin. Dyn. Syst.~\textbf{40} (2020), no.~10, 5639--5710.
	\newblock DOI:10.3934/dcds.2020241.

	\bibitem{da2014stochastic}
	\textsc{G.~Da Prato} and \textsc{J.~Zabczyk}.
	\newblock Stochastic equations in infinite dimensions. Encyclopedia of Mathematics and its Applications, vol.~44,
	\newblock Cambridge University Press, Cambridge, 2014.
	DOI:10.1017/CBO9780511666223.

	\bibitem{Gao}
	\textsc{P.~Gao}.
	\newblock {\em The stochastic {S}wift-{H}ohenberg equation\/}.
	\newblock Nonlinearity~\textbf{30} (2017), no.~9, 3516--3559.
	\newblock DOI:10.1088/1361-6544/aa7e99.

	\bibitem{geiss2024sharp}
	\textsc{S.~Geiss}.
	\newblock {\em Sharp convex generalizations of stochastic {G}ronwall
		inequalities\/}.
	\newblock J. Differential Equations~\textbf{392} (2024), 74--127.
	\newblock DOI:10.1016/j.jde.2024.02.018.

        \bibitem{gyongy1980stochastic}
        \textsc{I.~Gy{\"o}ngy and N.~Krylov}.      
         \newblock{\em On stochastic equations with respect to semimartingales I.},
        \newblock{Stochastics: An International Journal of Probability and Stochastic Processes~\textbf{4} (1980), no.~1, 1--21},

	\bibitem{gyongy2021ito}
	\textsc{I.~Gy{\"o}ngy} and \textsc{S.~Wu}.
	\newblock {\em On It{\^o} formulas for jump processes\/}.
	\newblock Queueing Systems~\textbf{98} (2021), no.~3, 247--273.

    \bibitem{gyongy2025once}
    \textsc{I.~Gy{\"o}ngy and N.~Krylov}.
    \newblock{\em Once again on evolution equations with monotone operators in {H}ilbert spaces and applications.}
    \newblock{Stoch. Partial Differ. Equ., Anal. Comput. (2025), 1--52.}

	\bibitem{he2019semimartingale}
	\textsc{S.~He}, \textsc{J.~Wang}, and \textsc{J.~Yan}.
	\newblock Semimartingale theory and stochastic calculus,
	\newblock Kexue Chubanshe (Science Press), Beijing; CRC Press, Boca
	Raton, FL, 1992.

	\bibitem{hornung2019quasilinear}
	\textsc{L.~Hornung}.
	\newblock {\em Quasilinear parabolic stochastic evolution equations via
		maximal {$L^p$}-regularity\/}.
	\newblock Potential Anal.~\textbf{50} (2019), no.~2, 279--326.
	DOI:10.1007/s11118-018-9683-9.

	\bibitem{hytonen2016analysis}
	\textsc{T.~Hyt\"{o}nen}, \textsc{J.~van~Neerven}, \textsc{M.~Veraar}, and \textsc{L.~Weis}.
	\newblock Analysis in {B}anach spaces. {V}ol. {I}. {M}artingales and {L}ittlewood-{P}aley theory. Ergebnisse der Mathematik und ihrer Grenzgebiete, vol.~63,
	\newblock Springer, Cham, 2016.
	\newblock DOI:10.1007/978-3-319-48520-1.

	\bibitem{ikeda2014stochastic}
	\textsc{N.~Ikeda} and \textsc{S.~Watanabe}.
	\newblock Stochastic differential equations and diffusion processes. {N}orth-{H}olland {M}athematical {L}ibrary, vol.~24,
	\newblock North-Holland Publishing Co., Amsterdam-New York; Kodansha,
	Ltd., Tokyo, 1981.

	\bibitem{jacod2013limit}
	\textsc{J.~Jacod} and \textsc{A.~Shiryaev}.
	\newblock Limit theorems for stochastic processes. {G}rundlehren der mathematischen Wissenschaften, vol.~288,
	\newblock Springer-Verlag, Berlin, 2003.
	DOI:10.1007/978-3-662-05265-5.

	\bibitem{KalSzt}
	\textsc{O.~Kallenberg} and \textsc{R.~Sztencel}.
	\newblock {\em Some dimension-free features of vector-valued martingales\/}.
	\newblock Probab. Theory Related Fields~\textbf{88} (1991), no.~2, 215--247.
	DOI:10.1007/BF01212560.

	\bibitem{karatzas2014brownian}
	\textsc{I.~Karatzas} and \textsc{S.~Shreve}.
	\newblock Brownian motion and stochastic calculus. Graduate Texts in Mathematics, vol.~113,
	\newblock Springer-Verlag, New York, 1991.
	DOI:10.1007/978-1-4612-0949-2.

	\bibitem{Krysimple}
	\textsc{N.~Krylov}.
	\newblock {\em A relatively short proof of {It{\^o}}'s formula for {SPDEs} and its applications\/}.
	\newblock Stoch. Partial Differ. Equ., Anal. Comput.~\textbf{1} (2013), no.~1, 152--174.
	DOI:10.1007/s40072-013-0003-5.

	\bibitem{KR79}
	\textsc{N.~Krylov} and \textsc{B.~Rozovskii}.
	\newblock {\em Stochastic evolution equations\/}.
	\newblock In: Current problems in mathematics, Vol. 14 (Russian). Akad. Nauk SSSR, Moscow, 1979.

	\bibitem{kumar2024well}
	\textsc{A.~Kumar} and \textsc{M.~Mohan}.
	\newblock {\em Well-posedness of a class of stochastic partial differential equations with fully monotone coefficients perturbed by L{\'e}vy noise\/}.
	\newblock Analysis and Mathematical Physics~\textbf{14} (2024), no.~3.

	\bibitem{Lio69}
	\textsc{J.-L.~Lions}.
	\newblock Quelques m\'ethodes de r\'esolution des probl\`emes aux
	limites non lin\'eaires.
	\newblock Dunod, 1969.

	\bibitem{LR15}
	\textsc{W.~Liu} and \textsc{M.~R\"ockner}.
	\newblock Stochastic {P}artial {D}ifferential {E}quations: {A}n {I}ntroduction.
	\newblock Springer, Cham, 2015.
	\newblock DOI:10.1007/978-3-319-22354-4.

	\bibitem{marinelli2016maximal}
	\textsc{C.~Marinelli} and \textsc{M.~R\"ockner}.
	\newblock {\em On the maximal inequalities of {B}urkholder, {D}avis and
		{G}undy\/}.
	\newblock Expo. Math.~\textbf{34} (2016), no.~1, 1--26.
	DOI:10.1016/j.exmath.2015.01.002.

	\bibitem{metivier1976equation}
	\textsc{M.~M\'etivier} and \textsc{G.~Pistone}.
	\newblock {\em Sur une \'equation d'\'evolution stochastique\/}.
	\newblock Bull. Soc. Math. France~\textbf{104} (1976), no.~1, 65--85.

	\bibitem{meyer2002cours}
	\textsc{P.~Meyer}.
	\newblock Un cours sur les int\'egrales stochastiques. Lecture Notes in Math., vol.~511,
	\newblock Springer, Berlin-New York, 1976.

	\bibitem{Par75}
	\textsc{E.~Pardoux}.
	\newblock {\em \'Equations aux d\'eriv\'ees partielles stochastiques nonlin\'eares monotones: \'etude de solutions fortes de type It\^o\/}.
	\newblock PhD thesis, Universit\'e Paris--Orsay, 1975.

	\bibitem{portal_veraar}
	\textsc{P.~Portal} and \textsc{M.~Veraar}.
	\newblock {\em Stochastic maximal regularity for rough time-dependent problems\/}.
	\newblock Stoch PDE: Anal Comp~\textbf{7} (2019), no.~4, 541--597.
	\newblock DOI:10.1007/s40072-019-00134-w.

	\bibitem{RSZ}
	\textsc{M.~R\"{o}ckner}, \textsc{S.~Shang}, and \textsc{T.~Zhang}.
	\newblock {\em Well-posedness of stochastic partial differential equations
		with fully local monotone coefficients\/}.
	\newblock Math. Ann.~\textbf{390} (2024), no.~3, 3419--3469.
	\newblock DOI:10.1007/s00208-024-02836-6.

	\bibitem{TayPDE1}
	\textsc{M.~Taylor}.
	\newblock Partial differential equations {I}. {B}asic theory. Applied Mathematical Sciences, vol.~115,
	\newblock Springer, New York, 2011.

	\bibitem{TayPDE2}
	\textsc{M.~Taylor}.
	\newblock Partial differential equations {II}. {Q}ualitative studies of
	linear equations. Applied Mathematical Sciences, vol.~116,
	\newblock Springer, New York, 2011.

	\bibitem{TV24_large}
	\textsc{E.~Theewis} and \textsc{M.~Veraar}.
	\newblock {\em Large Deviations for Stochastic Evolution Equations in the Critical Variational Setting\/}.
	\newblock Available on ArXiv: \url{https://arxiv.org/abs/2402.16622}.

	\bibitem{yaroslavtsev2019martingales}
	\textsc{I.~Yaroslavtsev}.
	\newblock {\em Martingales and stochastic calculus in Banach spaces\/}.
	\newblock PhD thesis, TU Delft, 2019.

	\bibitem{ZhangKS}
	\textsc{L.~Zhang}.
	\newblock {\em Decay of solutions of the multidimensional generalized {Kuramoto}- {Sivashinsky} system\/}.
	\newblock IMA J. Appl. Math.~\textbf{50} (1993), no.~1, 29--42.
	\newblock DOI:10.1093/imamat/50.1.29.

    \bibitem{zhu2010study}
	\textsc{J.~Zhu}.
	\newblock {\em A study of {SPDEs} w.r.t.~compensated {P}oisson random measures and related topics\/}.
	\newblock PhD thesis, University of York, 2010.

\end{thebibliography}
\end{document}